\documentclass[a4paper,11pt]{article} 

\usepackage{graphicx}
\usepackage{multicol,multirow}
\usepackage{hhline}
\usepackage{mathrsfs}
\usepackage{verbatim}
\usepackage{float}
\usepackage{bm}
\usepackage[utf8]{inputenc} 
\usepackage[english]{babel} 
\usepackage{braket}
\usepackage{amsmath}
\usepackage{amssymb}
\usepackage{amsthm}
\usepackage{multirow}
\usepackage{threeparttable}
\usepackage{tikz-cd}
\usepackage{array}
\usepackage{longtable}
\usepackage{caption}
\usepackage[toc,page]{appendix}
\usepackage{rotating}
\usepackage[breaklinks,colorlinks,allcolors=blue]{hyperref}
\usepackage{tablefootnote}
\newcolumntype{C}[1]{>{\centering\arraybackslash}p{#1}}
\newcolumntype{P}[1]{>{\centering\arraybackslash}p{#1}}

\DeclareMathOperator{\codim}{codim}
\DeclareMathOperator{\NE}{NE}
\DeclareMathOperator{\Nef}{Nef}

\DeclareMathOperator{\Sing}{Sing}

\DeclareMathOperator{\Eff}{Eff}
\DeclareMathOperator{\Ext}{Ext}
\DeclareMathOperator{\Pic}{Pic}
\DeclareMathOperator{\Bl}{Bl}
\theoremstyle{plain}
\newtheorem{proposition}{Proposition}[section]
\newtheorem{theorem}[proposition]{Theorem}
\newtheorem{lemma}[proposition]{Lemma}
\newtheorem{corollary}[proposition]{Corollary}
\theoremstyle{definition}

\newtheorem{remark}[proposition]{Remark}
\newtheorem{example}[proposition]{Example}

\theoremstyle{plain}
\newcommand{\thistheoremname}{}
\newtheorem*{genericthm}{\thistheoremname}

\title{\large\textbf{ON CASAGRANDE-DRUEL FANO VARIETIES\\ WITH LEFSCHETZ DEFECT 2}}
\author{\normalsize PIER ROBERTO PASTORINO}
\date{}

\begin{document}
\maketitle

\begin{abstract}
The larger the Lefschetz defect $\delta_X$ of a smooth complex Fano variety $X$, the more information we can deduce  about the geometry of $X$. By \cite{Cas12}, if $\delta_X\ge 4$, then $X\simeq S\times T$ where $\dim(T)=\dim(X)-2$. The structure of $X$ when $\delta_X=3$ was described in \cite{CRS22}. In this paper, we study the case $\delta_X=2$. In particular, we focus on Fano varieties with $\delta_X=2$ arising from the construction introduced by C. Casagrande and S. Druel in \cite{CD15}, which we refer to as \textbf{Construction A}. We show that among the 19 families of Fano 3-folds with $\delta_X=2$ classified by Mori and Mukai, 15 arise from such construction. Moreover, we construct all Fano 4-folds with $\rho_X\ge 4$ and $\delta_X=2$ admitting such a structure, obtaining 145 distinct families in total. Combined with \cite[Theorem 1.1]{Sec23}, this yields a complete classification of all Casagrande-Druel Fano 4-folds with $\delta_X=2$. To broaden the scope, we also study a generalized version of Construction A, which we call \textbf{Construction B}, and we find that 18 out of the 19 families of Fano 3-folds with $\delta_X=2$ arise from it.   
\end{abstract}
\setcounter{tocdepth}{1}
\tableofcontents

\section{Introduction}
Let $X$ be a smooth complex Fano variety. The \emph{Lefschetz defect} of $X$, denoted by $\delta_X$, is an invariant introduced in \cite{Cas12} that relates the Picard number of $X$ to that of its prime divisors. Given a prime divisor $D$ in $X$, let us consider the natural push-forward $\iota_{*}: N_1(D)\to N_1(X)$ given by the inclusion $\iota: D\hookrightarrow X$, and set 
	\[
	N_1(D,X):=\iota_*(N_1(D))\subseteq N_1(X).
	\] 
	The Lefschetz defect of $X$ is defined as
	\[
	\delta_X=\max\{\codim N_1(D,X) \;|\;D\subset X  \text{ is a prime divisor in } X \}.
	\] 
	It was recently proved \cite{Ver25} that the Lefschetz defect is invariant under smooth deformation.
	If we denote by $\rho_X$ the Picard number of $X$, it is clear that $\delta_X\in \{0,...,\rho_X-1\}$. Recent developements in the study of the Lefschetz defect have highlighted the fact that large values of $\delta$ carry strong information about the geometric structure of the variety.  For instance, we have the following key result about Fano varieties with $\delta\ge 4$.
	\begin{theorem}[\cite{Cas12}, Theorem 3.3]\label{TheoremDelta4} Let $X$ be a smooth Fano variety with Lefschetz defect $\delta_X\ge 4 $. Then $X\simeq S\times T$, where $S$ is a del Pezzo surface with $\rho_S=\delta_X+1$.
	\end{theorem}   
	The case of a Fano variety $X$ with $\delta_X=3$ was studied in \cite{CRS22}, where it was shown that there exists a smooth Fano variety $T$ with $\dim T=\dim X-2$ such that $X$ is obtained from $T$ by means of two possible explicit constructions. In both cases, the construction consists of the composition of a $\mathbb{P}^2$-bundle $Z\to T$ with the blow-up of three pairwise disjoint smooth, irreducible, codimension 2 subvarieties in $Z$.  
	
	In this paper, we turn to the case $\delta=2$. We begin by analyzing the structure of Fano varieties with $\delta=2$ in arbitrary dimension from a theoretical perspective. We then specialize to the case of three and four-dimensional Fano varieties, where our results yield several new examples of Fano 4-folds with $\delta=2$.
	
	The main known result about the structure of Fano varieties with $\delta=2$ is the following.
	\begin{theorem}[\cite{Cas14}, Theorem 5.22]\label{C14Thm}
		Let $X$ be a smooth Fano variety with $\delta_X=2$, then one of the following holds.
		\begin{itemize}
			\item[(i)] There exists a sequence of flips $X\dashrightarrow
			X'$ and a conic bundle $\phi:X'\to Z$ where $X'$ and $Z$ are smooth, $\rho_X-\rho_Z=2$, and $\phi$ factors as $X'\stackrel{\sigma}{\to} Y\stackrel{\pi}{\to}Z$, where $\pi$ is a smooth $\mathbb{P}^1$-fibration and $\sigma$ is the blow-up of a smooth irreducible codimension 2 subvariety $A_Y$ in $Y$, such that $A_Y$ is a section over its image under $\pi$. 
			\item[(ii)] There is an equidimensional fibration in del Pezzo surfaces $\psi: X\to T $,  where $T$ has locally factorial, canonical singularities, $\codim \Sing(T)\ge 3$ and $\rho_X-\rho_T=3$. 
		\end{itemize}
	\end{theorem}  
	Examples suggest that $(i)$ could hold for every smooth Fano variety with $\delta=2$, and this is actually the case for smooth Fano 3-folds. In this paper we focus on smooth Fano varieties $X$ with $\delta=2$ satisfying a stronger version of $(i)$, namely the existence of a conic bundle $\phi\colon X\to Z$, with $\rho_X-\rho_Z=2$, such that $\phi$ factors through a smooth $\mathbb{P}^1$-fibration over $Z$, that would be $X=X'$ in $(i)$, without the need for any sequence of flips. Let $\pi\colon Y\to Z$ be the $\mathbb{P}^1$-fibration and $\sigma\colon X\to Y$ be the blow-up of $A_Y\subset Y$ such that $\phi=\pi\circ\sigma$. We study two particular instances of this setup. We assume that $A_Y$ lies in a section $S_Y$ of $\pi$; in particular this implies that $\pi$ is a $\mathbb{P}^1$-bundle\footnote{A morphism $\pi\colon Y\to Z$ will be called a ``$\mathbb{P}^1$-bundle over $Z$'' if $Y$ is the projectivization of a rank 2 vector bundle on $Z$ and $\pi$ is the natural projection.}. Under these assumptions, we say that $X$ arises from \emph{Construction B} (see Section \ref{SettingBCenter} for a more precise definition). If moreover there is another section of $\pi$ disjoint from $S_Y$, we say that $X$ arises from \emph{Construction A} (see Section \ref{SettingA} for a more precise definition). Although this setup might at first sight appear restrictive, we will see that it describes the structure of most families of Fano 3-folds with $\delta=2$, and yields several examples of families of Fano 4-folds with $\delta=2$. 
	
	We begin by studying Construction A. 
	Notice that Construction A coincides with the so called Casagrande-Druel construction (see \cite{CD15}), and seems a natural starting point for studying the structure of Fano varieties with $\delta=2$. Indeed, it is known that if $X$ is a Fano variety with $\dim(X)\ge 3$, $\delta_X=2$ and $\rho_X=3$ (the minimal possible value for $\rho$ when $\delta=2$), then $X$ necessarily arises from the Casagrande-Druel construction \cite[Theorem 3.8]{CD15}. This result was also applied in \cite{Sec23} in order to classify all families of Fano 4-folds with $\rho=3$ and $\delta=2$.  Variations and special cases of such construction have recently appeared in \cite{CDFKM23} and \cite{Mal24}, with a focus on K-stability issues.
	
	Given a smooth variety $X$ arising from Construction A, our first goal is to identify criteria ensuring that $X$ is a Fano variety with $\delta_X=2$, expressed in terms of conditions on $Z$. Such criteria are provided by the combination of Proposition \ref{ConstructionA} (see also \cite[Proposition 2.3]{CDFKM23}) and Lemma \ref{deltaLemma}. 	 
	These results yield effective tools for studying Fano varieties with $\delta=2$ in an inductive way from lower to higher dimension.
	
	In contrast, the study of Construction B turns out to be more involved. Let $X$ be a smooth variety arising from Construction B. In the pursuit of conditions for $X$ to be Fano that can be tested on $Z$, we prove the following theorem: 
	\begin{theorem}(Theorem \ref{ConstructionB})\label{ConstructionBIntro}
		Let $Z$ be a smooth Fano variety, $\pi\colon  Y=\mathbb{P}(\mathcal{E})\to Z$ a $\mathbb{P}^1$-bundle and $S_Y\subset Y$ a section of $\pi$ associated to a short exact sequence of this form:
		\begin{align}\label{SectionSequence}
			0\to\mathcal{O}_Z\longrightarrow\mathcal{E}\longrightarrow\mathcal{O}_Z(D)\to 0
		\end{align}  
		Let $A$ be a smooth irreducible codimension 1 subvariety of $Z$ and $\sigma\colon  X\to Y$ the blow-up of $A_Y:=\pi^{-1}(A)\cap S_Y$ in $Y$.  		 
		Consider the following conditions 
		\begin{itemize}
			\item[(I)] $-K_Z+D-A$ is ample on $Z$.
			\item[(II)] $\mathcal{E}\otimes\mathcal{O}_Z(-K_Z-D)$ is  ample on $Z$.
			\item[(III)] $\mathcal{E}|_A\otimes \mathcal{O}_Z(-K_Z-D)|_{A}$ is ample on $A$.
		\end{itemize}
		
		Then:  
		\begin{align*}
			\begin{cases}
				\text{(I)}\\
				\text{(II)}
			\end{cases}
			\implies 
			\text{ } X \text{ is Fano }\implies 
			\begin{cases}
				\text{(I)}\\
				\text{(III)}.
			\end{cases}
		\end{align*}
	\end{theorem}

	When specialized to the setting of Construction A, conditions $(I)$ and $(II)$ in Theorem \ref{ConstructionBIntro} are both necessary and sufficient for $X$ to be Fano, and we recover Proposition \ref{ConstructionA}.
	However, examples show that, in general, conditions \emph{(I)}, \emph{(II)} are not necessary and conditions \emph{(I)}, \emph{(III)} are not sufficient for $X$ to be Fano. The difficulty in bridging the gap between necessary and sufficient conditions in Theorem \ref{ConstructionBIntro} is due to the \emph{asymmetric} nature of Construction B as opposed to the perfect symmetry of Construction A. To make this statement precise, we introduce \emph{Maruyama's elementary transformations} \cite{Mar82} and explain their relation to our constructions (see Section \ref{MaruyamaSection}).
	
	Regarding the determination of the Lefschetz defect in this setting, the same tools developed for Construction A remain effective and applicable (see Lemma \ref{deltaLemmaB}).

	All these results allow us to list all families of Fano 3-folds with $\delta=2$ that arise from Construction A and Construction B. It turns out that among the 19 families of Fano 3-folds with $\delta=2$ in the Iskovskikh-Mori-Mukai classification (see \cite{Isk77}, \cite{Isk78}, \cite{MM81} and \cite{MM03}), only 4 cannot be obtained via Construction A (see Section \ref{Standard3folds}), and only one family cannot be obtained via Construction B. Members of this latter family (\#4-1 in \cite[Table 6.1]{Ara+23}) can be described as the blow up of a $(2,2,2)$-curve in $\mathbb{P}^1\times\mathbb{P}^1\times\mathbb{P}^1$ that is the smooth complete intersection of two $(1,1,1)$-divisors. Thus they admit a similar conic bundle structure as the one described by Construction B, but in this case the blown up locus $A_Y\subset Y:=(\mathbb{P}^1)^3$ is a $(2,2,2)$-curve that does not lie in any section of the $\mathbb{P}^1$-bundle $\pi\colon Y\to Z:=\mathbb{P}^1\times\mathbb{P}^1$. 
	
	Turning to the four dimensional case, it is known that if $X$ is a Fano 4-fold with $\delta_X=2$, then $3\le\rho_X\le 6$ (see \cite[Theorem 6.1]{CS24}). In this paper, we list all families of Fano 4-folds with $\delta=2$ and $\rho\ge4$ that can be obtained through Construction A, while the case $\rho=3$ was settled in \cite{Sec23}. As a result of Theorem \ref{Rho6}, Theorem \ref{AllRho5} and Theorem \ref{AllRho4}, we obtain the following:
	\begin{theorem}
		There are 145 distinct families of smooth Fano 4-folds with $\delta=2$ and $\rho\ge 4$ arising from Construction A. Among them, 94 have $\rho=4$ (see Table \ref{FinalTableRho4}), 50 have $\rho=5$ (see Table \ref{FinalTableRho5}) and one of them has $\rho=6$ (the product $\Bl_2\mathbb{P}^2\times\Bl_2\mathbb{P}^2$). Moreover,
        \begin{itemize}
            \item 43 of those families are toric (24 have $\rho=4$, 18 have $\rho=5$, 1 has $\rho=6$).
            \item 93 of those families are not toric nor  products of lower dimensional varieties ( 66 have $\rho=4$ and 27 have $\rho=5$).
        \end{itemize}
	\end{theorem}
	
	As a side comment, it is worth noticing that the Fano polytopes of all toric Fano 4-folds with $\delta=2$ arising from Construction A are of type $Q_n$ (if $\rho=5$), $H_n$ or $I_n$ (if $\rho=4$) for some $n\in\mathbb{N}$ (see \cite{Bat99}).  
	
	Among the 93 families of non-toric and non-product varieties, we searched for potential correspondences with some known databases of Fano 4-folds in the literature. We found no matches in \cite{FTT24} or in \cite{BFMT25}. However, we did identify some families in \cite{CKP15} that share some numerical invariants (for example $c_1^4$ and $h^0(-K)$) as some families in our list. It is worth noticing that the Hodge numbers of Fano families in \cite{CKP15} are not generally known, which makes a precise comparison with our families challenging and therefore beyond the scope of this paper.   
	
	Future projects may include the study of Fano 4-folds with $\delta=2$ arising from Construction B. 
	
	In conclusion, Construction A accounts for the majority of smooth Fano 3-folds with $\delta=2$ and yields 145 families of smooth Fano 4-folds with $\delta=2$ and $\rho\ge 4$. The more general Construction B captures all but one of the 19 families of smooth Fano 3-folds with $\delta=2$, and may ultimately provide a more coherent and uniform description of these varieties than the traditional case-by-case classification by Mori-Mukai. 
	
	\paragraph{Structure of the paper.} In Sections 2 and 3 we present Constructions A and B respectively and study conditions for them to yield a Fano variety. In Section 4 we explore the relation of Constructions A and B with the setting of \emph{Maruyama's elementary transformation}. In Section 5 we study the iteration of two constructions of type A. In Sections 6 and 7 we list all smooth Fano 3-folds with $\delta=2$ arising from Constructions A and B respectively. In Section 8 we give some preliminary results for the study of Fano 4-folds arising from Construction A. In Section 9 we show the techniques we use to compute the invariants of the Fano 4-folds we construct through Construction A. In Sections 10, 11 and 12 we list all smooth Fano 4-folds with $\delta=2$ arising from Construction A with $\rho=6$, $\rho=5$ and $\rho=4$ respectively. Finally, in Appendix A we give a complete list of all Fano 4-folds arising from Construction A with $\rho\in\{4,5\}$.   
	
	\paragraph{Notation and conventions.}
	We work over the field of complex numbers. Let $X$ be a smooth projective variety of arbitrary dimension. 
	
	$N_1(X)$ (respectively $N^1(X)$) is the real vector space of one-cycles (respectively, Cartier divisors) with real coefficients, modulo numerical equivalence, and $\dim N_1(X)=\dim N^1(X)=\rho_X$ is the Picard number of $X$. 
	We denote by $\Nef(X)$ the nef cone of $X$ and by $\Eff(X)$ the effective cone of $X$.
	
	The symbol $\sim$ stands for linear equivalence of divisors, while $\equiv$ stands for numerical equivalence of one-cycles, as well as divisors. 
	
	$\NE(X)\subset N_1(X)$ is the convex cone generated by the classes of effective curves. An \emph{extremal ray} $R$ is a one-dimensional face of $\NE(X)$. When $X$ is Fano, the \emph{length} of $R$ is \sloppy${\ell(R)=\min\{-K_X\cdot C\;|\;C\text{  is a rational curve in }R\}}$. 
	
	A \emph{contraction} is a surjective morphism $\varphi:X\to Y$ with connected fibers, where $Y$ is normal and projective. A contraction is \emph{elementary} if $\rho_X-\rho_Y=1$. 
	
	A \emph{conic bundle} is a contraction of fiber type $X\to Y$ where every fiber is one-dimensional and isomorphic to a plane conic.
	\\
	
	In $X=\mathbb{P}^1\times\mathbb{P}^1\times\mathbb{P}^1$, denote by $p_i\colon X\to \mathbb{P}^1$ the projection onto the $i$-th factor, and $H_i$ the class of $p_i^{*}\mathcal{O}_{\mathbb{P}^1}(1)$ in $N^1(X)$ for $i=1,2,3$. Let $f_1$, $f_2$, $f_3$ be the generators of $N_1(X)$ given by the classes of $\mathbb{P}^1\times\{*\}\times\{*\}$, $\{*\}\times\mathbb{P}^1\times\{*\}$, $\{*\}\times\{*\}\times\mathbb{P}^1$  respectively, so that $H_i\cdot f_j= \delta_{ij}$.
	By ``$(\alpha_1,\alpha_2,\alpha_3)$-divisor in $X$'' (with $\alpha_i\in\mathbb{Z}$) we mean a Cartier divisor $D$ in $X$ such that $[D]= \alpha_1H_1+\alpha_2H_2+\alpha_3H_3$. We also write $\mathcal{O}_X(D)=\mathcal{O}(\alpha_1,\alpha_2,\alpha_3)$. A curve $C$ in $X$ is called a ``$(\beta_1,\beta_2,\beta_3)$-curve'' if $[C]= \beta_1f_1+\beta_2f_2+\beta_3f_3$.	 
	
	Consider $W=\mathbb{P}^1\times\mathbb{P}^1$. We denote by $p_1$ and $p_2$ the projections of $W$ respectively onto the first and second factor. Let $\mathbf{l_i}$ be the numerical class of $p_i^{*}\mathcal{O}_{\mathbb{P}^1}(1)$ for $i=1,2$. Then $C$ is a ``$(a,b)$-curve in $W$'' if $C\equiv a\mathbf{l_1}+b\mathbf{l_2}$. We also write $\mathcal{O}_W(C)=\mathcal{O}(a,b)$. 
	\newline
	\newline
		We denote by $\Bl_m\mathbb{P}^2$ the blowup of $\mathbb{P}^2$ along $m$ points. 
		\\
	
	Smooth Fano 3-folds have been classified in \cite{Isk77}, \cite{Isk78}, \cite{MM81}, \cite{MM03} into 105 families. Throughout this paper, those families will be labeled as \#1-1,\#1-2,\#1-3,...,\#9-1,\#10-1, following the classical notation (see, for example, Table 6.1 in \cite{Ara+23}).  
	
	In tables \ref{3-17} to \ref{FinalTableRho5} (see Sections \ref{Standard4foldsRho4}, \ref{Standard4foldsRho5} and Appendix A), we list all Fano 4-folds $X$ arising from Construction A with $\rho\in \{4,5\}$. Under the column ``Toric?'', when $X$ is toric we write the type of the Fano polytope of the family of $X$ described in \cite{Bat99}; under the column ``Product'', when $X$ is a product of two lower dimensional Fano varieties, we make those varieties explicit. When one of those varieties is a Fano 3-fold, we refer to it by using the classical notation.

	\section{Construction A}
	\subsection{Setting}\label{SettingA}
	In this section we present Construction A. 
	Let $Z$ be a smooth projective variety with ${\dim(Z)=n-1\ge 2}$. Fix a Cartier divisor $D$ and a smooth irreducible hypersurface $A$ on $Z$. Let \sloppy ${Y:=\mathbb{P}_{Z}(\mathcal{O}_Z\oplus \mathcal{O}_Z(D))}$, with the projection $\pi\colon Y\to Z$. Then $\dim(Y)=n$ and $\rho_Y=\rho_Z+1$. 
	Consider $G_Y\subset Y$ and $\hat{G}_Y\subset Y$ the sections of $\pi$ corresponding respectively to the quotients ${\mathcal{O}_Z\oplus\mathcal{O}_Z(D)\twoheadrightarrow \mathcal{O}_Z}$   
	and ${\mathcal{O}_Z\oplus\mathcal{O}_Z(D)\twoheadrightarrow \mathcal{O}_Z(D)}$.   
	We have $G_Y\simeq Z\simeq \hat{G}_Y$, and it is not difficult to prove that $G_Y\cap \hat{G}_Y=\emptyset$, $\mathcal{N}_{G_Y/Y}\simeq \mathcal{O}_Z(-D)$ and $\mathcal{N}_{\hat{G}_Y/Y}\simeq \mathcal{O}_Z(D)$. 
	
	If $X$ is the blow up of $Y$ along $A_Y:=\pi^{-1}(A)\cap \hat{G}_Y$, we say that $X$ arises from \textbf{Construction A}. To summarize this construction (which only depends on the choice of $Z$, and the divisors $A,D\subset Z$), we will adopt the following notation: \[X=\mathscr{C}_A(Z;A,D).\] Observe that $X$ is a smooth projective variety with $\dim(X)=n$ and $\rho_X=\rho_Z+2$. Let $\sigma\colon X\to Y$ denote the blow-up map. Let $G$ and $\hat{G}$ be the strict transforms in $X$ of $G_Y$ and $\hat{G}_Y$ respectively. Then $\mathcal{N}_{G/X}\simeq \mathcal{O}_Z(-D)$ and $\mathcal{N}_{\hat{G}/X}\simeq \mathcal{O}_Z(D-A)$. 
	
	We call $D_1:=D$ and $D_2:=A-D$ so that $A\sim D_1+D_2$. 
	The composition $\phi:=\pi\circ\sigma$ is a conic bundle and has another factorization  $\phi=\hat{\pi}\circ\hat{\sigma}$ where $\hat{\sigma}: X\to \hat{Y}$ and $\hat{\pi}:\hat{Y}\to Z$ have the following properties:
	\begin{itemize}
		\item $\hat{Y}= \mathbb{P}_Z(\mathcal{O}_Z\oplus\mathcal{O}_Z(D_2))$ and $\hat{\pi}$ is the projection onto $Z$. 
		\item $\hat{\sigma}(G)$ and $\hat{\sigma}(\hat{G})$ are disjoint  sections of $\hat{\pi}$ in $\hat{Y}$. 
		\item $\mathcal{N}_{\hat{\sigma}(G)/\hat{Y}}\simeq \mathcal{O}_Z(D_2)$ and $\mathcal{N}_{\hat{\sigma}(\hat{G})/\hat{Y}}\simeq \mathcal{O}_Z(-D_2)$
		\item $\hat{\sigma}$ is the blow up of $\hat{Y}$ along $A_{\hat{Y}}:= \hat{\pi}^{-1}(A)\cap \hat{\sigma}(G)$.
		\item Let $E$ and $\hat{E}$ be the exceptional divisors respectively of $\sigma$ and $\hat{\sigma}$. Then $E$ and $\hat{E}$ coincide respectively with the strict transforms of $\hat{\pi}^{-1}(A)$ and $\pi^{-1}(A)$ in $X$.   		
	\end{itemize}
	The existence of such an alternative symmetric factorization is proved by M. Maruyama for a more general setting in \cite{Mar82}. The new variety $\hat{Y}$ is called  \emph{Maruyama's elementary transformation} of $Y$ and is birational to $Y$ through $\psi:=\hat{\sigma}\circ\sigma^{-1}$. 
	\begin{figure}[H]
		\centering
		\begin{tikzcd}
			&X \arrow[swap]{dl}{\sigma} \arrow{dr}{\hat{\sigma}} &  \\
			Y:=\mathbb{P}(\mathcal{O}_Z\oplus\mathcal{O}_Z(D_1))\arrow[swap]{dr}{\pi}\arrow[dotted]{rr}{\psi}&  & \hat{Y}:=\mathbb{P}(\mathcal{O}_Z\oplus\mathcal{O}_Z(D_2))\arrow{dl}{\hat{\pi}}\\
			&Z&
		\end{tikzcd}
		\caption{Alternative factorization of Construction A.}
		\label{figure:1}
	\end{figure}
	Using the notation we introduced earlier, we can write $\mathscr{C}_A(Z;A,D_1)\simeq\mathscr{C}_A(Z;A,D_2)$. 
	\begin{remark}\label{Intersection table}
		Let $e$ (resp.\ $\hat{e}$) be the class of a one-dimensional fiber of $\sigma$ (resp.\ of $\hat{\sigma}$) in $\NE(X)$. The classes $e$ and $\hat{e}$ generate the two-dimensional closed subcone $\NE(\phi)\subset\NE(X)$ of the curves contracted by $\phi$ and the following intersection table holds.
		\begin{table}[H]
			\begin{center}
				\begin{tabular}{|c|c|c|c|c|}
					\hline
					$\cdot$ & $E$ & $\hat{E}$ & $G$ & $\hat{G}$ \\
					\hline\hline
					$e$ & -1 & 1 &0 &1 \\
					\hline
					$\hat{e}$ & 1 & -1 &1&0\\
					\hline
				\end{tabular}
			\end{center}
			\caption{Intersection Table in $X$}
			\label{table:1}
		\end{table}
	\end{remark}
	\begin{remark}\label{Remark-K_X}
		From the known description of the anticanonical divisor of a blow-up and of a $\mathbb{P}^1$-bundle, we have 
		${-K_X\sim\sigma^{*}(-K_Y)-E}$ and ${-K_Y\sim\pi^{*}(-K_Z)+G_Y+\hat{G}_Y}$.
		Hence, since $\sigma^{*}G_Y\sim G$ and $\sigma^{*}(\hat{G}_Y)\sim\hat{G}+E$, we have
		\begin{equation}\label{-KX_Formula}
			-K_X\sim\sigma^{*}(-K_Y)-E\sim\phi^{*}(-K_Z)+\sigma^{*}G_Y+\sigma^{*}\hat{G}_Y-E\sim\phi^{*}(-K_Z)+G+\hat{G}.
		\end{equation}
	\end{remark}
	
	\subsection{Conditions for being Fano}
	Observe that the conic bundle $X\to Z$ given by Construction A has no nonreduced fibers, so that by \cite[Proposition 4.3]{Wis91} if $X$ is Fano, then $Z$ is Fano as well. We now give some conditions for $X$ to be Fano, that can be tested on $Z$, after assuming that $Z$ is Fano (see also \cite[Proposition 2.3]{CDFKM23}).
	\begin{proposition} \label{ConstructionA} 
		Let $Z$ be a smooth Fano variety, $A$ a smooth irreducible hypersurface in $Z$ and $D_1$, $D_2$ two Cartier divisors in $Z$ such that $A\sim D_1+D_2$. 
		
		Consider $X=\mathscr{C}_A(Z;A,D_1)(\simeq\mathscr{C}_A(Z;A,D_2))$. Then $X$ is Fano if and only if $-K_Z-D_i$ is ample for $i=1,2$.  
	\end{proposition}
	\begin{proof}
		First assume $X$ is Fano. Then $-K_X|_{G}$ must be ample on $G$ and $-K_X|_{\hat{G}}$ must be ample on $\hat{G}$. Since $\mathcal{N}_{G/X}\simeq \mathcal{O}_Z(-D_1)$  and $\mathcal{N}_{\hat{G}/X}\simeq \mathcal{O}_Z(-D_2)$, by the adjunction formula we get that $-K_Z-D_1$ and  $-K_Z-D_2$  are ample on $Z$. 
		
		Conversely, assume that $-K_Z-D_1$ and $-K_Z-D_2$ are ample on $Z$. We show that $-K_X$ is ample on $X$.

		First we show that $-K_X$ is strictly nef, i.e.~it has positive intersection with every irreducible curve on $X$.  Observe that by the adjunction formula, $-K_X|_{G}$ is ample on $G$ because $-K_Z-D_1$ is ample on $Z$,  and $-K_X|_{\hat{G}}$ is ample on $\hat{G}$ because $-K_Z-D_2$ is ample on $Z$. Now, let $\Gamma$ be any irreducible curve in $X$. Since $-K_X$ is ample on $G$ and $\hat{G}$, if $\Gamma$ is contained in the support of one of those divisors, we have $-K_X\cdot\Gamma>0$ by Kleiman's criterion of ampleness. Suppose that $\Gamma$ is not contained in $G\cup\hat{G}$. Then we have 
		$\phi^{*}(-K_Z)\cdot \Gamma\ge0$, $G\cdot\Gamma\ge 0$ and $\hat{G}\cdot\Gamma\ge 0$,
		because $G$, $\hat{G}$ are effective and $\phi^{*}(-K_Z)$ is nef. 
		By (\ref{-KX_Formula}), it follows that 
		\begin{align*}
			-K_X\cdot \Gamma =\phi^{*}(-K_Z)\cdot \Gamma+G\cdot\Gamma+\hat{G}\cdot\Gamma\ge 0.
		\end{align*}
		Now, if $\Gamma$ is not contracted by $\phi$, then $\phi_{*}\Gamma$ is an irreducible curve in $Z$ and $\phi^{*}(-K_Z)\cdot\Gamma=-K_Z\cdot\phi_{*}\Gamma>0$ because $-K_Z$ is ample. This implies $-K_X\cdot \Gamma>0$. On the other hand, if $\Gamma$ is contracted by $\phi$, its class in $\NE(X)$ sits in the 2-dimensional subcone $\NE(\phi)=\mathbb{R}_{\ge0}e+\mathbb{R}_{\ge 0}\hat{e}$, which is $K$-negative because $-K_X\cdot e=1=-K_X\cdot\hat{e}$. This proves that $-K_X$ is strictly nef. 
		
		Let us now prove that $-K_X$ is big on $X$. 
		Observe that, by (\ref{-KX_Formula}), the top self intersection of $-K_X$ can be written as follows:
		\begin{align*}
			(-K_X)^n=(-K_X)^{n-1}\cdot (\phi^{*}(-K_Z)+G+\hat{G}).
		\end{align*}
		Now, $(-K_X)^{n-1}\cdot \phi^{*}(-K_Z)\ge 0$ because both $-K_X$ and $\phi^{*}(-K_Z)$ are nef divisors in $X$. Moreover, $(-K_X)^{n-1}\cdot G>0$, and $(-K_X)^{n-1}\cdot \hat{G}>0$ because $-K_X|_{G}\sim(-K_Z-D_1)$ and $-K_{X}|_{\hat{G}}\sim (-K_Z-D_2)$ are ample by assumption (here we identify $G$ and $\hat{G}$ with $Z$). 
		This implies that $(-K_X)^n>0$, hence $-K_X$ is big. 
		
		In conclusion, since $-K_X$ is strictly nef and big, it is ample on $X$ by the base-point-free theorem.
	\end{proof}
	
	We now study the Lefschetz defect of $X$ in relation to the Lefschetz defect of $Z$.
	\begin{lemma} \label{deltaLemma}
		Suppose $X=\mathscr{C}_A(Z;A,D)$ is Fano. Then we have:
		\begin{itemize}
			\item[(i)] $\delta_X\ge 2$.
			\item[(ii)]  $\delta_X\ge \delta_Z$. Moreover, if $\codim(N_1(A,Z))=\delta_Z$, then $\delta_X\ge \delta_Z+1$.  
			\item[(iii)] If $\delta_X\ge 4$, then $\delta_X\in\{\delta_Z,\delta_Z+1\}$. 
		\end{itemize}
		In particular, if $\delta_Z\le 2$, then $2\le\delta_X \le 3$. 
	\end{lemma}
	\begin{proof}
		It follows from Lemma \ref{deltaLemmaB}.
	\end{proof}
	
	Given a smooth Fano variety $X=\mathscr{C}_A(Z;A,D)$, we can also make some comments about curves of anticanonical degree 1 in $Z$.

	\begin{remark}\label{Degree1curve}
		Let $X=\mathscr{C}_A(Z;A,D)$ be a smooth Fano variety arising from Construction A. Suppose $C$ is an irreducible curve in $Z$ such that $-K_Z\cdot C=1$. Then by Proposition \ref{ConstructionA} it follows that 
		\begin{align*}
			\begin{cases}
				0<(-K_Z+D-A)\cdot C=1+D\cdot C-A\cdot C\\
				0<(-K_Z-D)\cdot C=1-D\cdot C 
			\end{cases}\\
		\end{align*} 
		i.e.~$ A\cdot C\le D\cdot C\le 0$. Therefore there are only two possible cases:
		\begin{enumerate}
			\item $A\cdot C\ge 0\implies A\cdot C=D\cdot C=0 \implies $ $C\subset A$ or $C\cap A=\emptyset$. 
			\item $A\cdot C<0\implies $ $C\subset A$.   
		\end{enumerate}
		
	\end{remark}
	
	More necessary conditions for $Z$ to be the base of a construction of type A yielding a Fano variety $X$ will be given in Section  \ref{Standard4folds} (see Remark \ref{RemarkStandard4folds}). 

\section{Construction B}
	In this section we present Construction B, which is a generalization of Construction A.
	\subsection{Setting} \label{SettingBCenter}
	Let $Z$ be a smooth projective variety with $\dim(Z)=n-1\ge 2$, $\mathcal{E}$ a rank $2$ vector bundle on $Z$, $Y:=\mathbb{P}_Z(\mathcal{E})$ and $\pi\colon Y\to Z$ the projection. Let $A\subset Z$ be a smooth irreducible hypersurface and assume that $\pi$ admits a section, which we denote by $S_Y\subset Y$. Such section $S_Y$ corresponds to a surjection $\beta:\mathcal{E}\twoheadrightarrow \mathcal{L}$ onto a line bundle $\mathcal{L}$ on $Z$.  
	Consider $A_Y:=\pi^{-1}(A)\cap S_Y$, and let $\sigma\colon X\to Y$ be the blow up of $Y$ along $A_Y$. The composition $\phi:=\pi\circ \sigma$ is a conic bundle. Observe that $X$ is a smooth projective variety, $\dim(X)=\dim(Y)=\dim(Z)+1$ and $\rho_X=\rho_Y+1=\rho_Z+2$. We will refer to this setting as \textbf{Construction B}, and we will use the following notation:
	\[X=\mathscr{C}_B(Z;A,S_Y)=\mathscr{C}_B(Z;A,\beta:\mathcal{E}\twoheadrightarrow \mathcal{L}).\] 
	
	Observe that if we further assume the existence of a section of $\pi$ disjoint from $S_Y$, then we recover the same setting presented in Construction A, which is therefore a special instance of this new more general construction.

	\begin{remark}\label{BirationalGeometryB}
		Let $X=\mathscr{C}_B(Z;A,S_Y)$. Consider the exceptional divisor $E$ of $\sigma\colon X\to Y$, and the strict transform $\hat{E}$ of $\pi^{-1}(A)$ in $X$. Both have a natural $\mathbb{P}^1$-bundle structure over $A$. In $N_1(X)$, let $e$ (resp.~$\hat{e}$) denote the class of a one dimensional fiber in $E$ (resp.\ in $\hat{E}$). Let $S$ be the strict transform of $S_Y$ in $X$. The following intersection table holds:
		\begin{center}
			\begin{tabular}{|c | c | c |c|c|} 
				\hline
				$\cdot$ & $-K_X$ & $E$ & $\hat{E}$ & $S$\\ 
				\hline
				$e$ &1& -1 & 1 &1 \\ 
				\hline
				$\hat{e}$&1 & 1 & -1 &0\\
				\hline
			\end{tabular}
		\end{center}
		The cone $\NE(\phi)$ of the curves contracted by $\phi$ is a two-dimensional $K_X$-negative subcone of $\NE(X)$ generated by $e$ and $\hat{e}$, which span two extremal rays. In particular, the curves in $\hat{e}$ are contracted by a divisorial contraction $\hat{\sigma}\colon X\to \hat{Y}$ onto a projective variety $\hat{Y}$, such that  $\hat{\sigma}$ is the blow-up of a codimension 2 subvariety in $\hat{Y}$ and $\hat{E}$ is its exceptional divisor.     
	\end{remark}   
	
	\subsection{Conditions for being Fano} 	
	Observe that the conic bundle $X\to Z$ given by Construction B has no nonreduced fibers, so that by \cite[Proposition 4.3]{Wis91} if $X$ is Fano, then $Z$ is Fano as well.
	We want to study conditions for $X$ to be Fano that can be tested on $Z$, provided that $Z$ is Fano. 
	\begin{theorem}\label{ConstructionB}
		Let $X=\mathscr{C}_B(Z;A,S_Y)$, where $Z$ is Fano and the section $S_Y$ is associated to a short exact sequence of this form:
		\begin{align}\label{SectionSequence}
			0\to\mathcal{O}_Z\longrightarrow\mathcal{E}\longrightarrow\mathcal{O}_Z(D)\to 0
		\end{align}  
		Observe that $\mathcal{O}_Z(D)\simeq\det(\mathcal{E})$ is isomorphic to the normal bundle $\mathcal{N}_{S_Y/Y}$, and $\mathcal{O}_Y(S_Y)\simeq\mathcal{O}_Y(1)$. 	 
		Consider the following conditions 
		\begin{itemize}
			\item[(I)] $-K_Z+D-A$  ample on $Z$.
			\item[(II)] $\mathcal{E}\otimes\mathcal{O}_Z(-K_Z-D)$  ample on $Z$.
			\item[(III)] $\mathcal{E}|_A\otimes \mathcal{O}_Z(-K_Z-D)|_{A}$ ample on $A$.
		\end{itemize}
		
		Then:  
		\begin{align*}
			\begin{cases}
				\text{(I)}\\
				\text{(II)}
			\end{cases}
			\implies 
			\text{ } X \text{ is Fano }\implies 
			\begin{cases}
				\text{(I)}\\
				\text{(III)}.
			\end{cases}
		\end{align*}
	\end{theorem}  
	\begin{proof}
		Let $E$ denote the exceptional divisor in $X$, and $S$ the  strict transform of $S_Y$ in $X$. We know that $\sigma^{*}S_Y\sim S+E$. Hence 
		\begin{align*}
			\mathcal{N}_{S/X}&=\mathcal{O}_S(S) \simeq\mathcal{O}_S(\sigma^{*}S_Y)\otimes\mathcal{O}_S(-E) \simeq \mathcal{O}_{S_Y}(S_Y)\otimes\mathcal{O}_{S_Y}(-A_Y)\simeq \mathcal{O}_Z(D-A).
		\end{align*}
		Therefore, by the adjunction formula,
		$\mathcal{O}_S(-K_X)=\mathcal{O}_S(-K_S)\otimes\mathcal{N}_{S/X}\simeq\mathcal{O}_Z(-K_Z)\otimes\mathcal{O}_Z(D-A)$. Hence 
		\begin{center}
			$-K_X$ ample $\implies$ $-K_X|_S$ ample $\iff$ condition \emph{(I)}. 
		\end{center}
		Let $\hat{E}$ denote the strict transform in  $X$ of $\pi^{-1}(A)$. Then $\pi^{-1}(A)\sim\hat{E}+E$, and \begin{align*} \mathcal{N}_{\hat{E}/X}=\mathcal{O}_{\hat{E}}(\hat{E})&\simeq \mathcal{O}_{\hat{E}}(\pi^{-1}(A))\otimes\mathcal{O}_{\hat{E}}(-E)\\&\simeq \mathcal{O}_{\pi^{-1}(A)}(\pi^{-1}(A))\otimes\mathcal{O}_{\pi^{-1}(A)}(-A_Y)\\&\simeq \pi^{*}\mathcal{O}_A(A)\otimes\mathcal{O}_{\pi^{-1}(A)}(-1)
		\end{align*}
		where we denote by $\mathcal{O}_{\pi^{-1}(A)}(-1)$ the dual of the tautological line bundle of $\pi^{-1}(A)=\mathbb{P}_A(\mathcal{E}|_A)$.
		Recall that $\hat{E}$ is the exceptional divisor of some blow-up morphism that contracts the fibers of the natural $\mathbb{P}^1$-bundle structure of $\hat{E}$ over $A$. (see Remark \ref{BirationalGeometryB}). 
		Therefore $\mathcal{O}_{\hat{E}}(1)=\mathcal{O}_{\hat{E}}(-\hat{E})=\pi^{*}\mathcal{O}_A(-A)\otimes\mathcal{O}_{\pi^{-1}(A)}(1)$, which implies ${\hat{E}=\mathbb{P}_A\bigl(\mathcal{E}|_A\otimes \mathcal{O}_A(-A)\bigr)}$. Then one can check that $\mathcal{O}_{\hat{E}}(K_{\hat{E}})\simeq\mathcal{O}_{\hat{E}}(2\hat{E})\otimes\phi|_{\hat{E}}^{*}\bigl(\mathcal{O}_A(K_Z+D-A)\bigr)$.
		By adjunction, we get
		\begin{align*}
			\mathcal{O}_{\hat{E}}(-K_X)\simeq& \mathcal{O}_{\hat{E}}(-K_{\hat{E}})\otimes\mathcal{O}_{\hat{E}}(\hat{E})
			\simeq\mathcal{O}_{\hat{E}}(-\hat{E})\otimes\phi_{\hat{E}}^{*}\bigl(\mathcal{O}_{A}(-K_Z-D+A)\bigr)
		\end{align*}
		which is the tautological line bundle for  \[\mathbb{P}_{A}\bigl(\mathcal{E}|_A\otimes\mathcal{O}_A(-A)\otimes\mathcal{O}_A(-K_Z-D+A)\bigr)=\mathbb{P}_{A}\bigl(\mathcal{E}|_A\otimes\mathcal{O}_A(-K_Z-D)\bigr).\] Thus, $-K_X|_{\hat{E}}$ is ample if and only if condition \emph{(III)} is satisfied. 
		This  concludes  the proof of the statement: $X$ Fano $\implies$ \emph{(I)}+\emph{(III)}. 
		
		Observe that $E\simeq \mathbb{P}_{A_Y}(\mathcal{N}_{A_Y/Y}^{*})$, and $\mathcal{N}_{A_Y/Y}\simeq\mathcal{O}_{A_Y}(\hat{S}_{Y})\oplus\mathcal{O}_{A_Y}(\pi^{-1}(A))\simeq \det\mathcal{E}|_A\oplus\mathcal{O}_A(A)$ because $A_Y$ is the smooth complete intersection of the divisors $S_Y$ and $\pi^{-1}(A)$. By similar computation as for $-K_X|_{\hat{E}}$, we can conclude that $-K_X|_{E}$ is ample if and only if $(-K_Z+D-A)|_A$ is ample on $A$. 
		
		Now, suppose conditions \emph{(I)} and \emph{(II)} are satisfied. 
		Observe that 
		\begin{align*}
			\mathcal{O}_X(-K_X)\simeq& \mathcal{O}_X\bigl(\sigma^{*}(-K_Y)\bigr)\otimes\mathcal{O}_X(-E)\simeq \mathcal{O}_X\bigl(2\sigma^{*}(S_Y)\bigr)\otimes\phi^{*}\bigl(\mathcal{O}_Z(-K_Z)\otimes \det(\mathcal{E})^{*}\bigr)\otimes\mathcal{O}_X(-E)\\
			\simeq& \mathcal{O}_X\bigl(2\sigma^{*}(S_Y)+\phi^{*}(-K_Z-D)-E\bigr).
		\end{align*}  
		Now, consider the divisor $W_Y:= S_Y+\pi^{*}(-K_Z-D)$ in $Y$. It is the tautological divisor for ${\mathbb{P}(\mathcal{E}\otimes\mathcal{O}_Z(-K_Z-D))}$ and \emph{(II)} implies that $W_Y$ is ample. Then there exist $m,c\in\mathbb{N}$ such that 
		\begin{align*}
			W_X:=m\sigma^{*}W_Y-cE
		\end{align*}
		is ample. 
		Observe that $-K_X\sim \sigma^{*}(W_Y)+\sigma^{*}(S_Y)-E\sim \sigma^{*}(W_Y)+S$. Therefore 
		\[
		-mK_X\sim m\sigma^{*}(W_Y)+mS\sim W_X+cE+mS
		\]
		which is the sum of an ample divisor and an effective one. This proves that \emph{(II)} implies that $-K_X$ is big. 
		
		Let $\Gamma$ be an irreducible curve in $X$. 
		
		\begin{itemize}
			\item If $\Gamma\subset S$, then \emph{(I)} implies that $-K_X\cdot \Gamma>0$, by Kleiman's criterion of ampleness. 
			\item If $\Gamma\subset E$, then again \emph{(I)} implies that $-K_X\cdot \Gamma>0$.
			\item If $\Gamma$ is not contained in $S$ nor in $E$, then $\sigma_{*}\Gamma$ is an irreducible curve in $Y$ and \emph{(II)} implies $-K_X\cdot\Gamma=S\cdot\Gamma+\sigma^{*}(W_Y)\cdot\Gamma=S\cdot\Gamma+W_Y\cdot\sigma_{*}\Gamma>0$, because $S\cdot\Gamma\ge 0$ and $W_Y$ is ample on $Y$. 
		\end{itemize}
		This proves that \emph{(I)} and \emph{(II)} imply that $-K_X$ is strictly nef. Since it is also big, we conclude that \emph{(I)} and \emph{(II)} imply that $-K_X$ is ample, by the base-point-free theorem. 
	\end{proof}
	\begin{remark}
		If  $\mathcal{E}\simeq \mathcal{O}_Z\oplus\mathcal{O}_Z(D)$, one checks easily that condition $(I)$ and $(II)$ in Theorem \ref{ConstructionB} become equivalent to requiring that $-K_{Z}-(A-D)$ and $-K_Z-D$ are ample on $Z$, so that we recover both necessary and sufficient conditions for $X$ to be Fano in Proposition \ref{ConstructionA}. 
		Moreover, condition \emph{(III)} in Theorem \ref{ConstructionB} becomes equivalent to requiring that $(-K_Z-D)|_A$ is ample on $A$.
		
		From this observation, it follows that conditions \emph{(I)} and \emph{(III)} in Theorem \ref{ConstructionB} are NOT sufficient for $X$ to be Fano. For example, take $Z=\mathbb{P}^1\times\mathbb{P}^1$, $A\equiv\mathbf{l_1}$, $D\equiv2\mathbf{l_1}$, where $\mathbf{l_1}$ and $\mathbf{l_2}$ are the classes of the two rulings in $N^1(Z)$, and consider $X=\mathscr{C}_B(Z;A,\mathcal{O}_Z\oplus\mathcal{O}_Z(D)\twoheadrightarrow\mathcal{O}_Z(D))=\mathscr{C}_A(Z;A,D)$. Conditions $(I)$ and $(III)$ are satisfied, but $-K_Z-D=2\mathbf{l}_2$ is not ample, thus $X$ is not Fano by Proposition \ref{ConstructionA}. 
		
		As for conditions \emph{(I)} and \emph{(II)}, we will see in Remark \ref{Remark(2,1,1)} that they are NOT necessary for $X$ to be Fano.  
	\end{remark}

	We now study the Lefschetz defect of $X$.
	
	\begin{lemma} \label{deltaLemmaB}
		Suppose $X=\mathscr{C}_B(Z;A,S_Y)$ is Fano. Then we have:
		\begin{itemize}
			\item[(i)] $\delta_X\ge 2$.
			\item[(ii)]  $\delta_X\ge \delta_Z$. Moreover, if $\codim(N_1(A,Z))=\delta_Z$, then $\delta_X\ge \delta_Z+1$.  
			\item[(iii)] If $\delta_X\ge 4$, then $\delta_X\in\{\delta_Z,\delta_Z+1\}$.  
		\end{itemize}
		In particular,  If $\delta_Z\le 2$ then $2\le\delta_X \le 3$.
	\end{lemma}
	\begin{proof}
		Let $\iota:S\hookrightarrow X$ be the inclusion of $S$ in $X$. We notice that, since $\phi\circ \iota:S\to Z$ is an isomorphism, $\dim(N_1(S,X))= \rho_S=\rho_Z=\rho_X-2$. Therefore $\delta_X\ge \codim(N_1(S,X))=2$, which proves $(i)$.
		
		Since $\phi$ is surjective, we have $\delta_X\ge \delta_Z$ (see \cite[Remark 3.3.18]{Cas12}). Now, consider the exceptional divisor $E$ in $X$. We know that $\ker(\sigma_*:N_1(X)\to N_1(Y))$ is $1$-dimensional (generated by the class of a $1$-dimensional fiber of $\sigma$) and is contained in $N_1(E,X)$. We also have $\sigma_{*}N_1(E,X)=N_1(A_Y,Y)$. Moreover, since $S_Y$ is a section of $\pi$, we have $N_1(S_Y)\simeq N_1(S_Y,Y)$, which implies that $N_1(A_Y, S_Y)\simeq N_1(A_Y,Y)$. 
		Then
		\begin{align*}
			\dim(N_1(E,X))&=\dim(N_1(A_Y,Y))+1= \dim(N_1(A_Y,S_Y))+1=\dim(N_1(A,Z))+1.
		\end{align*}  
		Therefore
		\begin{align*}
			\codim(N_1(E,X))&=\rho_X-\dim(N_1(A,Z))-1
			&=\rho_Z+2-\dim(N_1(A,Z))-1
			&=\codim(N_1(A,Z))+1.
		\end{align*}
		So if $\codim(N_1(A,Z))=\delta_Z$, we have $\delta_X\ge\delta_Z+1$, which completes the proof of $(ii)$.
		
		Suppose now that $\delta_X\ge 4$. Then by Theorem \ref{TheoremDelta4}, $X\simeq T\times X'$ where $T$ is a del Pezzo surface and $\rho_T=\delta_X+1\ge 5$. By \cite[Lemma 2.10]{Rom19}, we have $Z\simeq Z_1\times Z_2$ and $\phi$ is the product of two contractions $\phi_1: T\to Z_1$ and $\phi_2:X'\to Z_2$. If $Z_1=\{*\}$, then  $\dim(Z)=\dim(Z_2)\le\dim(X')=\dim(X)-2$, a contradiction, because Construction B requires $\dim(Z)=\dim(X)-1$- Thus $Z_1$ is a Fano variety of dimension 1 or 2. Recall that if $V_1$ and $V_2$ are two Fano varieties, $\delta_{V_1\times V_2}=\max\{\delta_{V_{1}},\delta_{V_2}\}$ \cite[Lemma 5]{Cas23}. There are three cases for $Z_1$. 
		\begin{enumerate}
			\item $Z_1=\mathbb{P}^1$. In that case
			$2=\rho_X-\rho_Z=(\rho_T-1)+(\rho_{X'}-\rho_{Z_2})\ge 4$, a contradiction.
			\item $Z_1=T$, $\phi_1=id_T$. In that case, $\delta_Z\ge\delta_T=\delta_X$. By $(ii)$, we have $\delta_Z\le\delta_X$. Hence $\delta_X=\delta_Z$.
			\item $\rho_{Z_1}=\rho_{T}-1$ and $\phi_1$ is divisorial; $\rho_{Z_2}=\rho_{X'}-1$.  In that case $Z_1$ is still a surface and we have
			\[
			\delta_Z\ge\delta_{Z_1}=\rho_{Z_1}-1=\rho_T-2=\delta_X-1.
			\]
		\end{enumerate}
		By $(ii)$, we  then have $\delta_{X}\in\{\delta_Z,\delta_Z+1\}$, which proves $(iii)$.
	\end{proof}
\section{Maruyama's elementary transformation}\label{MaruyamaSection}
	In this section we explore the connections between Construction A, Construction B, and the setting of Maruyama's elementary transformations (see \cite{Mar82}). It turns out that working in Maruyama's framework provides an explicit embedding for any variety arising from Construction B into the fibered product of two projective bundles. This embedding sheds light on the symmetries and asymmetries of the structure of varieties $X$ arising from Construction B, and might be a good starting point for trying to fill the gap between necessary and sufficient conditions for $X$ to be Fano in Theorem \ref{ConstructionB}, or at least for better understanding the reason behind such a gap.
	
	\subsection{Setting}\label{Maruyama}
	We present the setting of the Maruyama construction described in \cite{Mar82}. 
	Let $\mathcal{E}$ be a locally free sheaf on a locally noetherian scheme $Z$ and $\mathcal{F}$ be a locally free sheaf on an effective Cartier divisor $A$ of $Z$. If there is a surjective homomorphism $\psi: \mathcal{E}\to\mathcal{F}$ of $\mathcal{O}_Z$-modules, then $\mathcal{E}'=\ker(\psi)$ is  a locally free sheaf on $Z$. Indeed, consider $p\in Z-A$. Then clearly $\mathcal{F}_p=0$ and $\ker(\psi)_p\simeq\mathcal{E}_p$, which is a free $\mathcal{O}_{Z,p}$-module. On the other hand, if $p\in A$, then let $R$ denote the local ring $\mathcal{O}_{Z,p}$. We have $\mathcal{O}_{A,p}=R/(f)$ for some nonzero divisor $f\in R$, $\mathcal{E}_p\simeq R^{\oplus k}$ and $\mathcal{F}_p\simeq \left(R/(f)\right)^{\oplus l}$ for some $k,l\in\mathbb{N}$. Recall that 
	\[
	0\to R\xrightarrow{\cdot f} R\to R/(f)\to 0
	\] 
	is a projective resolution of $R/(f)$, so that the projective dimension of $R/(f)$ is $pd_R(R/(f))\le 1$, which yields $pd_R(\mathcal{F}_p)\le 1$. Since $\mathcal{E}_p$ is projective, it follows that $\mathcal{E}_p'$ is also projective, hence free with rank $k$.   
	
	The procedure to obtain $\mathcal{E}'$ from $\mathcal{E}$ is said to be the \emph{(sheaf  theoretic) Maruyama's elementary transformation} of $\mathcal{E}$ along $\mathcal{F}$ and we denote it by $\mathcal{E}'=\text{elm}_{Z}(\mathcal{E},\mathcal{F})$. The sheaf $\mathcal{E}'$ is said  to be the \emph{elementary transform} of $\mathcal{E}$ along $\mathcal{F}$. Such elementary transformation yields the exact commutative diagram in Figure \ref{figure:2},
	\begin{figure}[H]
		\begin{center}
			\begin{tikzcd}[sep=small]
				& 0 & 0 & &\\
				0\ar[r]&\mathcal{F}'\ar[u]\ar[r]& \mathcal{E}|_{A}\ar[r]\ar[u]&\mathcal{F}\ar[r]&0\\
				0\ar[r]&\mathcal{E}'\ar[u,"\psi'"]\ar[r]&\mathcal{E}\ar[u]\ar[r,"\psi"]&\mathcal{F}\ar[r]\ar[u,equal]&0\\
				&\mathcal{E}(-A)\ar[u]\ar[r,equal]&\mathcal{E}(-A)\ar[u]&&\\
				&0\ar[u]&0\ar[u]&&
			\end{tikzcd}
		\end{center}
		\caption{Maruyama's elementary transformation diagram.}
		\label{figure:2}
	\end{figure} \noindent
	where the leftmost vertical exact sequence gives us the \emph{inverse} of the given transformation, i.e.~${\mathcal{E}(-A)=\text{elm}_Z(\mathcal{E}',\mathcal{F}')}$. 
	Observe that $\mathbb{P}(\mathcal{F})$ and $\mathbb{P}(\mathcal{F}')$ are closed subschemes of $\mathbb{P}(\mathcal{E})$ and  $\mathbb{P}(\mathcal{E}')$ respectively, because $\psi$ and $\psi'$ are surjective. We have the following geometric interpretation of the above operation:  
	\begin{theorem}[\cite{Mar73}, Theorem 1.1]\label{MaruyamaThm}
		Let $\sigma\colon  X\to \mathbb{P}(\mathcal{E})$ (or  $\sigma':X'\to\mathbb{P}(\mathcal{E}')$ ) be the blowing-up along $\mathbb{P}(\mathcal{F})$ (or  $\mathbb{P}(\mathcal{F}')$ resp.).Then we have an isomorphism $h:X\to X'$ of $Z$-schemes. For the projection $\pi\colon \mathbb{P}(\mathcal{E})\to Z$ (or  $\pi':\mathbb{P}(\mathcal{E}')\to Z$ ), the proper transform of $\sigma^{-1}\bigr(\pi^{-1}(A)\bigr)$ (or  $(\sigma')^{-1}\bigr((\pi')^{-1}(A)\bigr)$ resp.) is equal to the exceptional divisor of $\sigma'$ (resp.~of $\sigma$).   
	\end{theorem}
	
	\begin{figure}[H]
		\begin{center}
			\begin{tikzcd}
				X\ar[d,"\sigma"']\ar[rr,"\sim","h"']&&X'\ar[d,"\sigma'"]\\\mathbb{P}(\mathcal{E})\ar[dr,"\pi"']\ar[rr,dashed]&& \mathbb{P}(\mathcal{E}')\ar[dl,"\pi'"]\\
				&Z&			
			\end{tikzcd}
		\end{center}
        \caption{The geometry of Maruyama's elementary transformation.}
        \label{MarTheoremFigure}
	\end{figure}
	The birational map $\sigma'\circ h\circ  \sigma^{-1}:\mathbb{P}(\mathcal{E})\dashrightarrow\mathbb{P}(\mathcal{E}')$ is called \emph{(geometric) Maruyama's elementary transformation} of $\mathbb{P}(\mathcal{E})$ along $\mathbb{P}(\mathcal{F})$.   
	
	Call $Y:=\mathbb{P}(\mathcal{E})$ and $Y':=\mathbb{P}(\mathcal{E}')$. Consider $\mathcal{Q}:=Y\underset{Z}{\times}Y'$ and let $p:\mathcal{Q}\to Y$ and $p':\mathcal{Q}\to Y'$ be the projections onto its factors. Then we have canonical isomorphisms 
	\begin{align*}
		\mathbb{P}_{Y}(\pi^{*}\mathcal{E'})\simeq\mathcal{Q}\simeq\mathbb{P}_{Y'}({\pi'}^{*}\mathcal{E})
	\end{align*}
	so that the following diagram commutes:
	\begin{figure}[H]
		\begin{center}
			\begin{tikzcd}
				\mathcal{Q}\ar[r,"p'"]\ar[d,"p"']&\mathbb{P}(\mathcal{E}')\ar[d,"\pi'"]\\
				\mathbb{P}(\mathcal{E})\ar[r,"\pi"']&Z
			\end{tikzcd}
		\end{center}
	\end{figure}
	
	The idea of the proof of Theorem \ref{MaruyamaThm} is based on the fact that there are natural embeddings $i:X\hookrightarrow \mathcal{Q}$ and $i':X'\hookrightarrow \mathcal{Q}$ whose images coincide in $\mathcal{Q}$. 
	
	\begin{remark}
		Let $q:\mathcal{Q}\to Z$ be the composition $p\circ\pi=p'\circ\pi'$. If $\mathcal{E}$ is a rank 2 vector bundle on $Z$ and $\mathcal{F}$ is a line bundle on $A$, then $X$ is embedded in $\mathcal{Q}$ as a codimension 1 subvariety, and it can be proved that 
        \begin{equation}\label{FormulaX}
		X\sim p^{*}\mathcal{O}_{Y}(1)+(p')^{*}\mathcal{O}_{Y'}(1)+q^{*}(A-\det\mathcal{E})\sim p^{*}\mathcal{O}_{Y}(1)+(p')^{*}\mathcal{O}_{Y'}(1)-q^{*}\det\mathcal{E}'  
		\end{equation}
        as a Cartier divisor in $\mathcal{Q}$.
        
As we will see in Section \ref{MaruyamaB}, a variety X arising from Construction B can be interpreted in the framework of Maruyama's elementary transformation and embedded as a divisor into a fiber product of two projective bundles, exactly as in (\ref{FormulaX}). One may expect this description to be useful in studying necessary and sufficient conditions for X to be Fano. However, it is not straightforward to extract effective information from it, and we will not make further use of it in this paper. Its investigation is therefore left to future work. At present, our strongest result in this direction remains Theorem \ref{ConstructionB}, although the gap between the necessary and sufficient conditions is still unresolved.    
	\end{remark}
	\subsection{Relation to Construction A}\label{MaruyamaA}
	In the setting of Section \ref{Maruyama}, take $Z$ a smooth Fano variety, $D_1$ and $D_2$ two Cartier divisors in $Z$ such that $A\sim D_1+D_2$ is the class of a smooth irreducible hypersurface. Consider  $\mathcal{E}=\mathcal{O}_Z\oplus\mathcal{O}_Z(D_1)$, $\mathcal{F}=\mathcal{O}_A(D_1)$ and $\psi: \mathcal{E}\to \mathcal{F}$ the surjection given by the projection onto the second factor and the restriction to $A$. The kernel of this morphism of $\mathcal{O}_Z$-modules is $\mathcal{E}'=\mathcal{O}_Z\oplus\mathcal{O}_Z(-D_2)$. Therefore, in this case we have $\text{elm}_Z(\mathcal{O}_Z\oplus\mathcal{O}_Z(D_1),A)=\mathcal{O}_Z\oplus\mathcal{O}_Z(-D_2)$.
	
	Call $Y_i=\mathbb{P}(\mathcal{O}_Z\oplus\mathcal{O}_Z(D_i))$, $i=1,2$.
	Observe that in this setting, the first row of the commutative diagram in Figure \ref{figure:2} is a short exact sequence that defines $\mathbb{P}(\mathcal{F})$ as a section of the $\mathbb{P}^1$-bundle $\pi|_{\mathbb{P}_A(\mathcal{E}|_A)}\colon \mathbb{P}_A(\mathcal{E}|_A)\to A$. That sequence is nothing but the restriction to $A$ of the short exact sequence representing the section $\hat{G}_{Y_1}\in |\mathcal{O}_{Y_1}(1))|$ in $Y_1$:
	\begin{align*}
		0\to \mathcal{O}_Z\to\mathcal{O}_Z\oplus\mathcal{O}_Z(D_1)\to\mathcal{O}_Z(D_1)\to 0.
	\end{align*}
	This  means that $\mathbb{P}(\mathcal{F})$ coincides with the complete intersection of $\pi^{-1}(A)=\mathbb{P}_A(\mathcal{E}|_A)$ and $\hat{G}_{Y_1}$ in $Y_1$. By an analogous argument on $\mathbb{P}(\mathcal{F}')$ in $Y_2\simeq\mathbb{P}(\mathcal{E}')$, it is clear that Theorem \ref{MaruyamaThm} applied  to this setting yields  ${\mathscr{C}_A(Z;A,D_1)\simeq\mathscr{C}_A(Z;A,D_2)}$, as we mentioned in Section \ref{SettingA}.  
	
	\subsection{Relation to Construction B}\label{MaruyamaB}
	We now work in the setting of Construction B (see Section \ref{SettingBCenter}). Consider the surjection $\psi:\mathcal{E}\twoheadrightarrow \mathcal{L}|_A$ given by the composition of the surjection $\mathcal{E}\twoheadrightarrow\mathcal{L}$ defining the section $S_Y$, and the restriction to $A$. 
	
	Let $\mathcal{E}'=\text{elm}_Z(\mathcal{E},\mathcal{L}|_A)
	$. Observe that $\mathbb{P}(\mathcal{L}_A)$ coincides with the codimension 2 subvariety $A_Y$ in $Y=\mathbb{P}(\mathcal{E})$. We call $Y':=\mathbb{P}(\mathcal{E}')$. By Theorem \ref{MaruyamaThm}, we have a commutative diagram as in Figue \ref{MarTheoremFigure}:
	in this case, $\sigma$ is the blowing up of $A_Y=\mathbb{P}(\mathcal{L}_A)$ in $Y$ and $\sigma'$ is the blowing up of $\mathbb{P}(\mathcal{F}')$ in $Y'$. Observe that $A_Y$ is the complete intersection of $\pi^{-1}(A)$ and the section $S_Y$ corresponding to the surjection $\mathcal{E}\twoheadrightarrow\mathcal{L}$, so that $X$ actually coincides with $\mathscr{C}_B(Z;A,\mathcal{E}\twoheadrightarrow\mathcal{L})$,  while $\mathbb{P}(\mathcal{F}')$ might very well not lie in any section of $\pi'$. This explains the asymmetrical nature of the conditions appearing in Theorem \ref{ConstructionB} in contrast to the symmetry of the conditions appearing in Proposition \ref{ConstructionA}. Even the natural embedding of $X$ in $\mathcal{Q}=Y\underset{Z}{\times}Y'$ does not help in finding more satisfying conditions for $X$ to be Fano in the setting of Construction B. 
	
	We now give an example of a Fano 3-fold $X$ arising from a construction of type B, whose factorization given by Theorem \ref{MaruyamaThm} does not describe another construction of type B, i.e.~the blown-up locus in $Y'=\mathbb{P}(\mathcal{E}')$ does not lie in any section of the $\mathbb{P}^1$-bundle $\pi'\colon Y'\to Z$. For this purpose, we first prove the following lemma.
	\begin{lemma}\label{LemmaFactorization}
		Let $X=\mathscr{C}_B(Z;A,\mathcal{E}\twoheadrightarrow\mathcal{L})$ be a smooth variety arising from Construction B, such that $A$ meets non-trivially every curve in $Z$. Let $Y=\mathbb{P}_Z(\mathcal{E})$, and consider $Y'=\mathbb{P}_Z(\mathcal{E}')$, where $\mathcal{E'}=\text{elm}_Z(\mathcal{E},\mathcal{L}|_A)$. Consider the commutative diagram arising from Theorem \ref{MaruyamaThm} (Figure \ref{MarTheoremFigure}). Then the composition $\pi'\circ\sigma'$ yields a construction of type B if and only if $\pi\circ\sigma$ and  $\pi'\circ\sigma'$ both yield constructions of type A.  
	\end{lemma}
	\begin{proof}
		Suppose that $\pi'\circ\sigma'\colon X\to Z$ yields a construction of type B. This means that the blown up locus $A_{Y'}=\mathbb{P}(\mathcal{F}')$ in $Y'$ lies in a section $S_{Y'}$ of $\pi'$. Now, we know that $A_{Y'}$ is disjoint from the section $\sigma'(S)$, where $S$ is the strict transform of $S_Y$ in $X$,  because $S\cap \hat{E}=\emptyset$ in $X$ (see Section \ref{SettingBCenter}). If $\sigma'(S)\cap S_{Y'}=\emptyset$, then the composition $\pi'\circ\sigma'$ yields a construction of type A for $X$, which implies that $\pi\circ \sigma$ yields a contruction of type A as well, by the discussion in Section \ref{MaruyamaA}. 
		
		On the other hand, if $\sigma'(S)\cap S_{Y'}\ne\emptyset$, then since $A$ meets non-trivially every curve in $Z$, we have $A_{Y'}\cap \sigma'(S)\ne\emptyset$, which is a contradiction.     
	\end{proof}

	\begin{example}\label{MaruyamaBexample}
		Let $X$ be the blow up of a curve of degree $(1,1,1)$ in $\mathbb{P}^1\times\mathbb{P}^1\times\mathbb{P}^1$. As we will see in Section \ref{General3folds}, a curve of degree $(1,1,1)$ in $Y:=\mathbb{P}^1\times\mathbb{P}^1\times\mathbb{P}^1$ is contained in a section $S_Y$ of the projection onto the second and third factors $\pi\colon Y\to Z=\mathbb{P}^1\times\mathbb{P}^1$, and such section corresponds to the surjection  ${\mathcal{O}_Z\oplus\mathcal{O}_Z\twoheadrightarrow\mathcal{O}_Z(1,0)}$ (unique up to multiplication by scalars). If we let $A$ be the diagonal in ${Z=\mathbb{P}^1\times\mathbb{P}^1}$, we have $X=\mathscr{C}_B(Z;A,S_Y)$. 
		In order to find the  alternative factorization of the composition $\phi=\pi\circ\sigma$ provided by Maruyama, we need to consider the surjection ${\mathcal{O}_Z\oplus\mathcal{O}_Z\twoheadrightarrow\mathcal{O}_Z(1,0)|_A}$. Then we have \[\mathcal{E}'=\text{elm}_Z(\mathcal{O}_Z\oplus\mathcal{O}_Z,\mathcal{O}_Z(1,0)|_A)=\ker(\mathcal{O}_Z\oplus\mathcal{O}_Z\twoheadrightarrow\mathcal{O}_Z(1,0)|_A)=\mathcal{O}_Z(-1,0)\oplus\mathcal{O}_Z(0,-1)\] and \[\mathcal{F}'=\ker(\mathcal{O}_A\oplus\mathcal{O}_A\twoheadrightarrow\mathcal{O}_Z(1,0)|_A)=\mathcal{O}_Z(-1,0)|_A\simeq\mathcal{O}_{\mathbb{P}^1}(-1).\] We denote by $Y'=\mathbb{P}(\mathcal{E}')$, $\pi':Y'\to Z$ the projection and $\sigma'$ the blow up of $\mathbb{P}(\mathcal{F}')$ in $Y'$. Theorem \ref{MaruyamaThm} yields $\pi\circ\sigma=\phi=\pi'\circ\sigma'$.
		\begin{figure}[H]
			\begin{center}
				\begin{tikzcd}
					X\ar[r,"\sigma'"]\ar[d,"\sigma"']\ar[dr, "\phi"]&Y'\ar[d,"\pi'"]\\
					Y\ar[r,"\pi"']&Z 
					
				\end{tikzcd}
			\end{center}
		\end{figure}
		 Observe that $X$ does not appear in the list of Fano 3-folds arising from Construction A obtained in Section \ref{Standard3folds}. Since the diagonal $A\subset \mathbb{P}^1\times\mathbb{P}^1$ is ample, by Lemma \ref{LemmaFactorization}, we conclude that the alternative factorization $\pi'\circ\sigma'$ given by the Maruyama transformation does not yield a construction of type B. 
		
	\end{example}   
	\begin{remark}
		Consider the following Fano 3-folds:
		\begin{itemize}
			\item $X_1$= Blow up of a $(1,1,2)$-curve in $\mathbb{P}^1\times\mathbb{P}^1\times\mathbb{P}^1$. 
			\item $X_2$= Blow up of a $(1,1,3)$-curve in
			$\mathbb{P}^1\times\mathbb{P}^1\times\mathbb{P}^1$. 
		\end{itemize}
		If we consider their descriptions as varieties arising from Construction B (see Section \ref{General3folds}), it is clear that their factorizations through the construction given by Maruyama do not yield another construction of type B. This can be seen with similar computations as in Example \ref{MaruyamaBexample}. 
	\end{remark}

\section{Recursion}\label{Recursion}
	In this section, we study varieties obtained by performing two successive constructions of type A.
	\subsection{Setting}\label{SettingRecursion}
	Let $X$ be a smooth variety arising from Construction A, i.e.~$X=\mathscr{C}_A(Z;A,D)$.  Set ${Y:=\mathbb{P}_Z(\mathcal{O}_Z\oplus\mathcal{O}_Z(D))}$, let $\pi: Y\to Z$ be the projection, and let $G_Y$, $\hat{G}_Y$ be the two disjoint sections of $\pi$ satisfying
	$\mathcal{N}_{G_Y/Y}\simeq\mathcal{O}_Z(-D)$ and  $\mathcal{N}_{\hat{G}_Y/Y}\simeq\mathcal{O}_Z(D)$. Set $A_Y:=\pi^{-1}(A)\cap \hat{G}_{Y}$, let ${\sigma:X\to Y}$ be the blow-up of $Y$ along $A_Y$, and denote by $\phi$ the composition $\pi\circ\sigma$.
	
	Now consider $V=\mathscr{C}_A(X;B,L)$. We denote by $W$ the variety $\mathbb{P}_X(\mathcal{O}_X\oplus\mathcal{O}_X(L))$. Let $\rho: W\to X$ be the projection, and let $H_W$, $\hat{H}_W$ be the two disjoint sections of $\rho$ satisfying
	$\mathcal{N}_{H_W/W}\simeq\mathcal{O}_X(-L)$ and  $\mathcal{N}_{\hat{H}_W/W}\simeq\mathcal{O}_X(L)$. Consider ${B_W:=\rho^{-1}(B)\cap \hat{H}_{W}}$, ${\tau:V\to W}$ the blow-up of $W$ along $B_W$, and denote by $\psi$ the composition $\rho\circ\tau$.
	
	We are therefore in the presence of a tower like the one in Figure \ref{DoubleTower}.
	\begin{figure}[H]
		\begin{center}
			\begin{tikzcd}[sep=small]
				V\arrow{d}{\tau}\\
				W=\mathbb{P}_{X}(\mathcal{O}_X\oplus\mathcal{O}_X(L))\arrow{d}{\rho}&\supset& B_{W}:=\rho^{-1}(B)\cap \hat{H}_{W}\\
				X\arrow{d}{\sigma}&\supset &B\text{, }L\\
				Y=\mathbb{P}_{Z}(\mathcal{O}_Z\oplus\mathcal{O}_Z(D))\arrow{d}{\pi}&\supset& A_{Y}:=(\pi)^{-1}(A)\cap \hat{G}_{Y}\\
				Z&\supset& A, D
			\end{tikzcd}  
		\end{center} 
		\caption{Iteration of two constructions of type A.} 
		\label{DoubleTower}
	\end{figure}
	Our goal is to find conditions for $V$ to be Fano that can be tested on $Z$. This will be particularly useful in Section \ref{Case2}. In order to do so, we first establish some preliminary results.
	
	One important consequence of Remark \ref{Degree1curve} is the following lemma, which is stated and proved in the more general setting where $X$ arises from Construction B.
	\begin{lemma}\label{A',D'}
		Let $V=\mathscr{C}_A(X;B,L)$ be a smooth Fano variety and assume that $X=\mathscr{C}_B(Z;A,S_Y)$. Then there exist $A'$ and  $D'$ in $Z$ such that $B=\phi^{*}A'$ and $L=\phi^{*}D'$. 
	\end{lemma}
	\begin{proof}
		For $X$ we use the notation from Section \ref{SettingBCenter}. We know that $\text{Pic}(X)=\mathbb{Z}E\oplus\sigma^{*}\text{Pic}(Y)=\mathbb{Z}E\oplus\mathbb{Z}{\sigma}^{*}S_{Y}\oplus\phi^{*}\text{Pic}(Z)=\mathbb{Z}E\oplus\mathbb{Z}S\oplus{\phi}^{*}\text{Pic}(Z)$. Recall that $\phi|_{E}:E\to A$ and $\phi|_{\hat{E}}:\hat{E}\to A$ are both $\mathbb{P}^1$-bundles. Let $e$ denote the class of a 1-dimensional fiber in $E$ and $\hat{e}$ the class of a one-dimensional fiber in $\hat{E}$.  We know that $-K_X\cdot e=1=-K_X\cdot \hat{e}$ (see Remark \ref{BirationalGeometryB}). Therefore, by Remark \ref{Degree1curve} we have $L\cdot e=B\cdot e=0$ or $B\cdot e<0$ and analogously $L\cdot \hat{e}=B\cdot \hat{e}=0$ or $B\cdot \hat{e}<0$. However if $B\cdot e<0$, then $F\subset B$ for every fiber $F$ of $\sigma$, i.e.~$E\subset B$, which implies $E=B$ since they are both irreducible divisors in $X$. But then $B\cdot \hat{e}=E\cdot \hat{e}=1$, which is a contradiction. Therefore $L\cdot e=B\cdot e=0$ and similarly it can be proved that $L\cdot \hat{e}=B\cdot\hat{e}=0$  and $B\ne \hat{E}$, so that $B,L\in \NE(\phi)^{\perp}$.
		Thus there exists an effective Cartier divisor $A'\subset Z$ such that $B=\phi^{*}A'$ and there exists a Cartier divisor $D'\subset Z$ such that $L=\phi^{*}D'$ (see, for example, \cite[Theorem 7.39(c)]{Deb01}).
	\end{proof}
	\begin{proposition}\label{SwitchProposition}
		Let $Z$ be a smooth projective variety. Consider $A$, $A'$ two distinct smooth irreducible hypersurfaces and $D$, $D'$ two Cartier divisors in $Z$. Let ${\phi:X=\mathscr{C}_A(Z;A,D)\to Z}$ and \sloppy ${\phi':X'=\mathscr{C}_A(Z;A',D')\to Z}$ be the conic bundles induced by Construction A. Then
		\begin{align*}
			\mathscr{C}_A\bigr(\mathscr{C}_A(Z,A,D);{\phi}^{*}A',{\phi}^{*}D'\bigr)\simeq\mathscr{C}_A\bigr(\mathscr{C}_A(Z,A',D');(\phi')^{*}A,(\phi')^{*}D\bigr)
		\end{align*} 
		i.e.~the following diagram commutes:
		
		\begin{equation}\label{DoubleDiagram}
			\begin{tikzcd}
				V\simeq V' \arrow[d,"\psi"']\arrow[dr,"\Phi"]\arrow[r,"\psi'"]& X' \arrow[d, "\phi'"] \\
				X\arrow[r,"\phi"'] & Z 
			\end{tikzcd}
		\end{equation}
		where $\psi:V=\mathscr{C}_A(X;\phi^{*}A',\phi^{*}D')\to X$ and $\psi':V'=\mathscr{C}_A(X';(\phi')^{*}A,(\phi')^{*}D)\to X'$ are the conic bundles induced by Construction A.
	\end{proposition}
	\begin{proof}
		In $X=\mathscr{C}_A(Z;A,D)$ consider the exceptional divisors $E$, $\hat{E}$, and the two disjoint sections $G$, $\hat{G}$ as  in Section \ref{SettingA}. In $\NE(X)$ consider the classes $e$ and $\hat{e}$ of the irreducible one-dimensional fibers  in $E$ and $\hat{E}$ respectively, so that ${\NE(\phi)=\mathbb{R}_{\ge0}e+\mathbb{R}_{\ge0}\hat{e}}$. Analogously, in $V=\mathscr{C}_A(X;\phi^{*}A',\phi^{*}D')$ consider the exceptional divisors $F$, $\hat{F}$, and the two disjoint sections $H$, $\hat{H}$. Let $f$ and $\hat{f}$ be the classes of the irreducible one-dimensional fibers in the exceptional divisors $F$ and $\hat{F}$ respectively, so that ${\NE(\psi)=\mathbb{R}_{\ge0}f+\mathbb{R}_{\ge0}\hat{f}}$.
		
		We denote by $\Phi$ the composition $\phi\circ\psi$. Observe that $\NE(\Phi)$ has dimension 4. We now show that $\Phi$ is $K_{V}$-negative, so that $\NE(\Phi)$ is a rational polyhedral extremal face of $\NE(V)$. 
		
		If  $[C]\in \NE(\Phi)$ is the class of an irreducible curve $C$ that is contracted by $\psi$, then $[C]\in\NE(\psi)$ and by construction it is already known that $-K_{V}\cdot C>0$. On the other hand, if $C$ is an irreducible curve contracted by $\Phi$ but not by $\psi$, it means that $\psi_{*}C$ is an irreducible component $T$ of a fiber of $\phi$ and its class in $\NE(X)$ lies in ${\NE(\phi)=\mathbb{R}_{\ge0}e+\mathbb{R}_{\ge0}\hat{e}}$. 
		In order to prove that $-K_{V}\cdot C>0$, it suffices to show that $-K_{V}$ is positive on $\psi^{-1}(T)\subset V$. Observe that there are only two possible cases for $T$:
		\begin{enumerate}
			\item $T\subset \phi^{*}A'$ ($\iff \phi(T)\in A'$).
			\item $T\cap  \phi^{*}A'=\emptyset$ ($\iff \phi(T)\notin A'$).
		\end{enumerate}       
		Set $B:=\phi^{*}A'$, $L:=\phi^{*}D'$, consider the $\mathbb{P}^1$-bundle $\rho:W=\mathbb{P}(\mathcal{O}_X\oplus\mathcal{O}_X(L))\to X$ and the blow-up $\tau:V\to W$ along the subvariety $B_{W}:=\rho^{-1}(B)\cap \hat{H}_{W}$. 
		
		In Case 1, ${\rho^{-1}(T)=\mathbb{P}_T(\mathcal{O}_T\oplus\mathcal{O}_T(L))=\mathbb{P}^1\times\mathbb{P}^1}$ because $L\cdot T=D'\cdot \phi_{*}T=0$. Consider $T_{W}:=\rho^{-1}(T)\cap \hat{H}_{W}\subset  B_{W}$. Then $\psi^{-1}(T)= \text{Bl}_{T_W}(\rho^{-1}(T))$. However $T_W$ is a codimension 1 subvariety of $\rho^{-1}(T)$, hence $\psi^{-1}(T)\simeq \rho^{-1}(T)\simeq \mathbb{P}^1\times\mathbb{P}^1$. Let $\mathbf{l}_1$, $\mathbf{l}_2$ be the embeddings in $\NE(V)$ of the two rulings in $\NE(\psi^{-1}(T))$ such that $\psi_{*}(\mathbf{l}_1)=[T]$ and $\psi_{*}(\mathbf{l}_2)=0$, where $\mathbf{l}_1$ is the class of the transform of $T_W$ in $V$. Then clearly 
		\begin{itemize}
			\item $-K_{V}\cdot \mathbf{l}_1=
			\phi^{*}(-K_{X})\cdot \mathbf{l}_1+H\cdot\mathbf{l}_1+\hat{H}\cdot\mathbf{l}_1= -K_{X}\cdot T+0+L\cdot T=-K_{X}\cdot T >0$
			\item $-K_{V}\cdot \mathbf{l}_2>0$
		\end{itemize}    
		because $-K_{X}$ is positive on $\NE(\phi)$ and $-K_{V}$ is positive on $\NE(\psi)$.
		
		In Case 2, it is easy to see that $\psi^{-1}(T)\simeq \rho^{-1}(T)\simeq \mathbb{P}^1\times\mathbb{P}^1$, and through similar computations as in Case 1, we can conclude that $-K_{V}$ is positive in $\psi^{-1}(T)$.

		This concludes the proof that $\NE(\Phi)$ is a (4-dimensional) closed rational polyhedral extremal face of $\NE(V)$. 
		
		We have $\psi_{*}\NE(\Phi)=\NE(\phi)$ through $\psi$, and faces of $\NE(\phi)$ are in bijection with faces of $\NE(\Phi)$ containing $\NE(\psi)$ (see \cite[\S 2.5]{Cas08}). Consider the extremal ray $\alpha:=\mathbb{R}_{\ge 0}e$ of $\NE(\phi)$, and let $\hat{\alpha}$ be the unique face of $\NE(\Phi)$ containing $\NE(\psi)$ such that $\psi_{*}(\hat{\alpha})=\alpha$. We have  ${\dim\hat{\alpha}=\dim\alpha+\dim\NE(\psi)=3}$. Let $\tilde{\alpha}$ be an extremal ray in $\hat{\alpha}$ not contained in $\NE(\psi)$ such that $\psi_{*}(\tilde{\alpha})=\alpha$.  By definition, there exists an irreducible curve whose class $f'\in\NE(V)$ spans $\tilde{\alpha}$ and $\psi_{*}(f')=e$ in $\NE(\phi)$. Therefore any irreducible representative $T'$ of $f'$ must lie in $\psi^{-1}(T)\simeq\mathbb{P}^1\times\mathbb{P}^1$, for some irreducible representative $T$ of $e$, so that in $\NE(\psi^{-1}(T))$ the class of $T'$ spans the ray generated by the \emph{horizontal} ruling $\mathbf{l_1}$. This is true because of the extremality of $\tilde{\alpha}$. In other words, we can take $T'=\psi^{-1}(T)\cap \hat{H}$ as a curve representing $f'$. Set $F':=\psi^{*}(E)$ and observe that ${F'\cdot f'=E\cdot e=-1}$, ${H\cdot f'=0}$ and ${\hat{H}\cdot f' = \phi^{*}D\cdot e= 0}$. Clearly $F'$ is covered by curves in $f'$, and actually coincides with the locus of $\tilde{\alpha}$. Therefore the elementary contraction $\tau':V\to W'$ associated to $\tilde{\alpha}$ (which exists because $\NE(\Phi)$ is $K_{V}$-negative) is divisorial and has fibers of dimension at most 1, so that it is in fact the blow up of a smooth irreducible codimension 2 subvariety in $W'$, with exceptional divisor $F'$. 
		
		In the same way, we can prove that there exists a class $\hat{f}'$ in $\NE(V)$ that spans an extremal ray and such that $\psi_{*}(\hat{f}')=\hat{e}$, $\hat{F}'\cdot\hat{f}'=-1$ (where $\hat{F}'=\psi^{*}(\hat{E})$), and $H\cdot \hat{f}'=0=\hat{H}\cdot\hat{f}'$. Let $\hat{\tau}':V\to \hat{W}'$ be the blow up induced by the extremal ray spanned by $\hat{f}'$. Observe that $f'$ and $\hat{f}'$ are linearly independent in $\NE(\Phi)$ because $e$ and $\hat{e}$ are linearly independent in $\NE(\phi)$. 
		
		Consider the classes $f$ and $\hat{f}$ in $F$ and $\hat{F}$ respectively. Observe that the following intersection table
		\begin{figure}[H]
			\centering
			\begin{tabular}{|c||c|c|c|c|}
				\hline
				$\cdot$ & $f$ & $\hat{f}$ & $f'$ & $\hat{f}'$\\
				\hline \hline
				$\hat{H}$ & 1 & 0 &0&0\\
				\hline
				$H$ & 0 & 1 & 0&0\\
				\hline
				$\phi^{*}\hat{G}$ & 0 & 0 & 1 & 0\\
				\hline
				$\phi^{*}G$ & 0& 0& 0&1\\
				\hline
			\end{tabular}
		\end{figure}  \noindent
		ensures that $f$, $\hat{f}$, $f'$, $\hat{f}'$ are linearly independent in $\NE(\Phi)$, which has dimension 4, hence every element in $\NE(\Phi)$ can be expressed in a unique way as a linear combination of those four classes. Also, observe that ${F\cdot f'=0=F\cdot \hat{f}'}$ and ${\hat{F}\cdot f'=0=\hat{F}\cdot \hat{f}'}$ as a consequence of the fact that $B\cdot e=0=B\cdot\hat{e}$ in $X$.    
		
		We already know that the 2-dimensional face $\mathbb{R}_{\ge 0}f+\mathbb{R}_{\ge0}\hat{f}$ in $\NE(\Phi)$ is extremal and coincides with $\NE(\psi)$. 
		We are now going to prove that $\NE(\Phi)=\mathbb{R}_{\ge0}f+\mathbb{R}_{\ge0}\hat{f}+\mathbb{R}_{\ge 0}f'+\mathbb{R}_{\ge 0}\hat{f}'$.
		
		Suppose $a\in\NE(\Phi)$. Then there exist coefficients $a_1,a_2,a_3,a_4\in\mathbb{R}$ such that 
		\begin{equation*}
			a=a_1f+a_2\hat{f}+a_3f'+a_4\hat{f}'
		\end{equation*}
		and we show that $a_i\ge 0$ for all $i=1,..,4$. 
		
		Observe that $\psi_{*}(a)=a_3e+a_4\hat{e}$ belongs to $\psi_{*}\NE(\Phi)=\NE(\phi)$, hence $a_3,a_4\ge 0$. Suppose $a_1<0$. Then $\hat{H}\cdot a=a_1<0$ implies that every representative of $a$ lies in $\hat{H}$, hence $0=H\cdot a=a_2$. But then $\hat{F}\cdot a=a_1-a_2=a_1<0$, so that every representative of $a$ lies in $\hat{H}\cap \hat{F}$, which is empty by construction. This proves that $a_1\ge 0$, and similarly it can be proved that $a_2\ge 0$. 
		
		In particular, the 2-dimensional face $\mathbb{R}_{\ge 0}f'+\mathbb{R}_{\ge0}\hat{f}'$ in $\NE(\Phi)$ is extremal. This yields the existence of a contraction $\psi':V\to X'$ such that ${\NE(\psi')=\mathbb{R}_{\ge 0}f'+\mathbb{R}_{\ge0}\hat{f}'}$, which by rigidity implies the existence of another contraction $\phi':X'\to Z$ that makes the following diagram commute.
		
		\begin{figure}[H]
			\begin{center}
				\begin{tikzcd}
					V  \arrow[d,"\psi"']\arrow[r,"\psi'"]\ar[dr,"\Phi"]&X' \arrow[d, dashed, "\phi'"] \\
					X\arrow[r,"\phi"'] &Z
				\end{tikzcd}
			\end{center}
		\end{figure}
		Observe that $\hat{H}':= \psi^{*}(\hat{G})$ and $H':= \psi^{*}(G)$ are two disjoint sections of $\psi'$, so that $\psi'$ realizes $V$ as a result of  Construction A, with $V=\mathscr{C}_A(X',(\phi')^{*}A,(\phi')^{*}D)$. Moreover, since $X'$ is a copy of $H'$ (and $\hat{H}'$), it can be seen that $\phi':X'\to Z$ yields $X'=\mathscr{C}_A(Z;A',D')$, where $G':={\psi'}_{*}(H'\cap H)$ and $\hat{G}':={\psi'}_{*}(H'\cap \hat{H})$ are two disjoint sections.  
	\end{proof}
	
	\subsection{Conditions for being Fano}\label{doubleConditions}
	We consider a smooth variety $V=\mathscr{C}_A(X;B,L)$ and set $L_1:=L$, $L_2:=B-L_1$. Suppose that $X$ also arises from Construction A, so that $X=\mathscr{C}_A(Z;A,D)$, with $D_1:=D$, $D_2:=A-D$. We assume that $Z$ is Fano.
	By Lemma \ref{A',D'} we know that there exist two divisors $A'$ and $D'$ in $Z$ such that $B=\phi^{*}A'$, $L=\phi^{*}D'$. Consider $D'_1:= D'$ and $D'_2:=A'-D'$. Let ${\phi':X'=\mathscr{C}_A(Z;A',D')\to Z}$ and \sloppy ${\psi':V'=\mathscr{C}_A(Z;(\phi')^{*}A,(\phi')^{*}D)\to Z}$ be the conic bundles induced by Construction A. By Proposition \ref{SwitchProposition}, $V\simeq V'$ and the diagram (\ref{DoubleDiagram}) commutes.
	
	We want to find conditions for $V$ to be Fano that can be tested on $Z$. By Proposition \ref{ConstructionA} and \cite[Proposition 4.3]{Wis91} we have that $V$ is Fano if and only if $-K_X-L_i$ is ample on $X$ for $i=1,2$ and $X$ is Fano. Observe that the equality $-K_{X}\sim\phi^{*}(-K_Z)+ G+\hat{G}$ (see Remark \ref{Remark-K_X}) implies that
	\begin{align*}
		-K_X-L_i\sim\phi^{*}(-K_Z-D_i')+G+\hat{G}\;\;\; \text{  for }i=1,2.
	\end{align*}
	Therefore, if we assume that $-K_Z-D_i'$ is ample for a fixed $i\in\{1,2\}$, following similar steps as in the proof of Proposition \ref{ConstructionA} we can show that
	\begin{align*}
		-K_X-L_i \text{ is ample on X} \iff -K_Z-D_i'-D_j \text{ are ample on $Z$ for }j=1,2.
	\end{align*}
	
	In conclusion, we obtain the following proposition:
	\begin{proposition}\label{DoubleProp} Assume that $Z$ is Fano. Then
		$V$ is Fano if and only if $-K_Z-D_i'-D_j \text{ are ample for all }i,j\in\{1,2\}$ and $-K_Z-D_i'$ are ample for $i=1,2$ (or equivalently $-K_Z-D_j$ are ample for $j=1,2$ ).
	\end{proposition}

\section{Fano 3-folds arising from Construction A} \label{Standard3folds}
	In this section we study how many families of Fano 3-folds with $\delta=2$ arise from Construction A.
	
	Let $X$  be  a Fano 3-fold with $\delta_X=2$. By \cite[Lemma 5.1]{Del14}, $\delta_X\in\{\rho_X-1,\rho_X-2\}$, so $\rho_X\in\{3,4\}$. If $X$ admits a construction of type A arising from a del Pezzo surface $Z$, we have $\rho_Z=\rho_X-2$, hence ${\rho_Z\in\{1,2\}}$. 
	\begin{itemize}
		\item The case $\rho_Z=1$ (i.e.~$Z\simeq \mathbb{P}^2$) has already been studied by C. Casagrande and S. Druel. With such a choice for $Z$, Construction A yields all families of Fano 3-folds $X$ with $\delta_X=2$ and $\rho_X=3$ \cite[Theorem 3.8]{CD15}. These are 6 among the 31 families of Fano 3-folds with $\rho_X=3$ (see \cite[Table 12.4]{IP99}). The remaining ones have $\delta_X=1$.    
		\item If $\rho_Z=2$, then $Z\simeq \mathbb{F}_1$ or $Z\simeq \mathbb{P}^1\times\mathbb{P}^1$. 
	\end{itemize}   
	For  each of the latter two del Pezzo surfaces, we study all possible choices of $D$ and  $A$ that yield $X=\mathscr{C}_A(Z;A,D)$ Fano, by using Proposition \ref{ConstructionA}. We then assign to every resulting Fano 3-fold its number in \cite[Table 6.1]{Ara+23}.   
	
	This study enables us to fully answer the question of how many families of Fano 3-folds with $\delta_X=2$ arise from Construction A. 
	\begin{itemize}
		\item $\rho_X=3\longrightarrow$ 6/6 families arise from Construction A.
		\item $\rho_X=4\longrightarrow$ 9/13 families arise from Construction A.
	\end{itemize}
	
	\begin{remark}\label{-K_X^3}
		Given $X=\mathscr{C}_A(Z;A,D_i)$, using that \sloppy ${-K_X\sim \phi^{*}(-K_Z)+G+\hat{G}}$ we can compute:
		\begin{align}
			-K_X^3=D_1^2+D_2^2+6(-K_Z)^2-3(-K_Z)\cdot (D_1+D_2).
		\end{align} 
	\end{remark}     
	\subsection{The case $Z\simeq\mathbb{F}_1$} In $N_1(Z)$, let $\mathbf{f}$ be the class of the strict transform of a line passing through the blown-up point  $p\in\mathbb{P}^2$ and denote by $\mathbf{e}$ the class of the exceptional divisor. Let $\mathbf{l}$ be the class of the transform of a generic line in $\mathbb{P}^2$, so that $\mathbf{l}=\mathbf{f}+\mathbf{e}$.  We have $\Eff(Z)=\NE(Z)=\mathbb{R}_{\ge 0}\mathbf{f}\oplus\mathbb{R}_{\ge 0}\mathbf{e}$, $\Nef(Z)=\mathbb{R}_{\ge 0}\mathbf{f}\oplus\mathbb{R}_{\ge 0}\mathbf{l}$ and $-K_Z=3\mathbf{f}+2\mathbf{e}$.
	This is enough to find all couples $(A,D)$ satisfying conditions in Proposition \ref{ConstructionA}.
	
	In the end we get the following 4 known families of Fano 3-folds, where we used
	Remark \ref{-K_X^3} to compute $-K_X^3$.
	\begin{table}[H] 
		\centering
		\begin{tabular}{||c|c||c|c||}
\hline
$A$ & $D$ & $-K_X^3$ &
\# of the family of $X=\mathscr{C}_A(Z;A,D)$ in \cite[Table 6.1]{Ara+23}\\
\hline\hline
\multirow{2}{*}{$\mathbf e$}
& $-\mathbf f/\mathbf f+\mathbf e$ & 46 & \#4-12\\
\hhline{|~|-|-|-|}
& $\mathbf e/0$ & 44 & \#4-11\\
\hline
$\mathbf f+\mathbf e$ & $\mathbf f+\mathbf e/0$ & 40 & \#4-9\\
\hline
$2\mathbf f+2\mathbf e$ & $\mathbf f+\mathbf e$ & 32 & \#4-4\\
\hline
\end{tabular}
		\centering
		\caption{Choices of $A$, $D$ in $Z=\mathbb{F}_1$.}
		\label{table:1}
	\end{table}	
	\begin{remark}
		When $(-K_X)^3=32$, the corresponding family can be determined using geometric arguments. 
	\end{remark}
\subsection{The case $Z\simeq \mathbb{P}^1\times\mathbb{P}^1$.} In $N_1(Z)$, let $\mathbf{l_1}$ and $\mathbf{l_2}$ be the classes of the two rulings of $Z$. They generate $N_1(Z)$. Moreover, we have $\Eff(Z)=\Nef(Z)=\NE(Z)=\mathbb{R}_{\ge 0}\mathbf{l_1}\oplus\mathbb{R}_{\ge 0}\mathbf{l_2}$,  and $-K_Z=2\mathbf{l_1}+2\mathbf{l_2}$. This is enough to find all couples $(A,D)$ satisfying conditions in Proposition \ref{ConstructionA}.
	
	We recover the following known 7 families of Fano 3-folds.

	\begin{table}[H]
		\centering
		\begin{tabular}{||c | c || c | c ||} 
			\hline
			$A$ & $D$ & $-K_X^3$ & \# of the family of $X=\mathscr{C}_A(Z;A,D)$ in \cite[Table 6.1]{Ara+23} \\ [0.5ex] 
			\hline\hline
			\multirow{2}{*}{$\mathbf{l_1}$} 	& $\mathbf{l_1}+\mathbf{l_2}/-\mathbf{l_2}$ & 44 & \#4-11 \\ 
\hhline{|~|-|-|-|}
			& $\mathbf{l_1}/0$ & 42 & \#4-10 \\
\hhline{|~|-|-|-|}
			& $\mathbf{l_1}-\mathbf{l_2}/\mathbf{l_2}$ & 40 & \#4-9 
			\\
			\hline
			\multirow{2}{*}{$\mathbf{l_1}+\mathbf{l_2}$} & $0/\mathbf{l_1}+\mathbf{l_2}$ & 38 & \#4-8  \\  
\hhline{|~|-|-|-|}
			& $\mathbf{l_1}/\mathbf{l_2}$ & 36 & \#4-7   \\
			\hline 
			$\mathbf{l_1}+2\mathbf{l_2}$ & $\mathbf{l_2}/\mathbf{l_1}+\mathbf{l_2}$ & 32 &\#4-5 \\
			\hline
			$2\mathbf{l_1}+2\mathbf{l_2}$ & $\mathbf{l_1}+\mathbf{l_2}$ & 28& \#4-2 \\
			\hline 
		\end{tabular}
		\caption{Choices of $A$, $D$ in $Z=\mathbb{P}^1\times\mathbb{P}^1$}
		
	\end{table}

	\section{Fano 3-folds arising from Construction B}\label{General3folds}
	
	As one can check directly by looking at the list by Mori and Mukai \cite{MM81}, \cite{MM03}, the remaining families of Fano 3-folds with $\delta=2$ that do not arise from Construction A can be described as blow-ups of $\mathbb{P}^1\times\mathbb{P}^1\times\mathbb{P}^1$ along a smooth curve:
	
	\begin{itemize}
		\item (\#4-6) $X=$ Blow up of a curve of degree $(1,1,1)$. 
		\item (\#4-3) $X=$ Blow up of a curve of degree $(1,1,2)$.
		\item (\#4-13)  $X=$ Blow up of a curve of degree $(1,1,3)$.
		\item (\#4-1)  $X=$ Blow up of a curve of degree $(2,2,2)$ that is the complete intersection of two divisors of type $(1,1,1)$.
	\end{itemize}
	
	Observe that the blow up of a curve of degree $(1,1,3)$ was missing in the first classification by Mori and Mukai \cite{MM81}, and was only discovered years later (see \cite{MM03}). 
	
	In this section, we study each of those families in order to determine whether they arise from Construction B or not. We will show that families \#4-6,\#4-3 and \#4-13 all arise from Construction B, while family \#4-1 cannot be obtained through any construction of type B, and is therefore the \emph{only} family of Fano 3-folds with $\delta=2$ for which that happens. 
	
	Throughout this section, in order to mirror the notation used to describe Construction B in Section \ref{SettingBCenter}, $Y$ will always denote the variety $\mathbb{P}^1\times\mathbb{P}^1\times\mathbb{P}^1$, $\sigma\colon  X\to Y$ will denote the blow up of the curve that defines $X$.   
	\paragraph{\fbox{$X=$Blow-up of a $(1,1,k)$ curve in $Y=\mathbb{P}^1\times\mathbb{P}^1\times\mathbb{P}^1$, for $k=1,2,3$.}}
	Set $Y=\mathbb{P}^1\times\mathbb{P}^1\times\mathbb{P}^1$ and
	consider the projections onto its $i$th and $j$th factors $\pi_{ij}\colon Y\to \mathbb{P}^1\times\mathbb{P}^1$ for $(i,j)\in\{(1,2),(1,3),(2,3)\}$. Let $S_Y$ be any irreducible divisor in $|\mathcal{O}_Y(1,1,0)|=|\pi_{12}^{*}(\mathcal{O}(1,1))|$. Then $S_Y\simeq\mathbb{P}^1\times\mathbb{P}^1$, and it is a section of $\pi_{23}$. Set $Z:=\mathbb{P}^1\times\mathbb{P}^1$. The short exact sequence of $\mathcal{O}_Z$-sheaves associated to the section $S_Y$ is of the following form: 
	\begin{equation*}
		0\to \mathcal{O}_Z(-1,0)\longrightarrow \mathcal{O}_Z\oplus\mathcal{O}_Z\longrightarrow \mathcal{O}_Z(1,0)\to 0.
	\end{equation*}
	
	Let $A_Y$ be any irreducible $(1,1,k)$-curve in $Y$. By restricting the projection onto the first and second factor $\pi_{12}|_{A_Y}\colon A_Y\to \mathbb{P}^1\times\mathbb{P}^1$, one gets a morphism of degree one onto an irreducible $(1,1)$-curve in $\mathbb{P}^1\times\mathbb{P}^1$. In particular, $A_Y$ is contained in some irreducible divisor $S_Y\in|\mathcal{O}(1,1,0)|$. 
	
	Now we denote by $A$ its image through the projection onto the second and third factors $\pi_{23}(A_Y)\subset Z$. One can easily check that $A$ is an irreducible $(k,1)$-curve in $Z$, so that $A_Y$ is also contained in $\pi^{-1}(A)\in|\mathcal{O}_Y(0,1,k)|$.

	This shows that if $X$ is obtained through the blow-up of an irreducible $(1,1,k)$-curve in $Y$, then ${X=\mathscr{C}_B(Z; A, S_Y)}$, and the families \#4-6, \#4-3 and \#4-13 all arise from Construction B. Observe that this construction for $k=1,2,3$ satisfies conditions \emph{(I)} and \emph{(II)} of Theorem \ref{ConstructionB}. 
	\begin{remark}\label{Remark(2,1,1)}
		We now show that a member of the second family (the blow up of a $(1,1,2)$-curve in $\mathbb{P}^1\times\mathbb{P}^1\times\mathbb{P}^1$) can be obtained through another construction of type B, for which condition \emph{(II)} of Theorem \ref{ConstructionB} is no longer satisfied. Again, set $Y=\mathbb{P}^1\times\mathbb{P}^1\times\mathbb{P}^1$ and $Z=\mathbb{P}^1\times\mathbb{P}^1$, and consider the projection $\pi_{12}:Y\to Z$ onto the first and second factors. Now consider the following section of $\pi_{12}$:
		\begin{align*}
			\mathbb{P}^1\times\mathbb{P}^1&\longrightarrow \mathbb{P}^1\times \mathbb{P}^1\times\mathbb{P}^1\\
			(x,y)&\mapsto (x,y,y^2)
		\end{align*} 
		and consider the image $S_Y\in|\mathcal{O}(0,2,1)|$. Observe that $Y=\mathbb{P}(\mathcal{O}_Z\oplus\mathcal{O}_Z)$ and  the short exact sequence of $\mathcal{O}_Z$-sheaves associated to the section $S_Y$ is of the following form: 
		\begin{equation*}
			0\to \mathcal{O}_Z(-2,0)\longrightarrow \mathcal{O}_Z\oplus\mathcal{O}_Z\longrightarrow \mathcal{O}_Z(2,0)\to 0
		\end{equation*}
		which becomes 
		\begin{equation*}
			0\to \mathcal{O}_Z\longrightarrow \mathcal{O}_Z(2,0)\oplus\mathcal{O}_Z(2,0)\longrightarrow \mathcal{O}_Z(4,0)\to 0
		\end{equation*}
		after normalization. 
		Let $A$ be a $(1,1)$-curve in $Z$, and let $A_Y$ be the complete intersection of $\pi^{-1}(A)$ and $S_Y$. A straightforward computation shows that $A_Y$ is a curve of degree $(1,1,2)$ in $Y$. We know that the blow up $X$ of such a curve is a Fano 3-fold, because its description appears in the classification by Mori-Mukai. However, we have just provided a construction of type B for $X$ for which condition \emph{(II)} in Theorem \ref{ConstructionB} 
		\begin{align*}
			\text{\emph{(II)}}\;\mathcal{E}\otimes\mathcal{O}_Z(-K_Z-D)\text{ is ample on }Z
		\end{align*} 
		is not satisfied. Indeed,
		in this setting $\mathcal{E}=\mathcal{O}(2,0)\oplus\mathcal{O}(2,0)$ and $\mathcal{O}_Z(D)=\mathcal{O}(4,0)$, so we have 
		$\mathcal{E}\otimes\mathcal{O}_Z(-K_Z-D)=\mathcal{O}(0,2)\oplus\mathcal{O}(0,2)$, which is not an ample vector bundle on $Z$. 	
		This proves that conditions \emph{(I)} and \emph{(II)} in Theorem \ref{ConstructionB} are sufficient but not necessary for $X$ to be Fano.  
		
	\end{remark}
	\paragraph{\fbox{$X=$ Blow-up of a $(2,2,2)$-curve in $\mathbb{P}^1\times\mathbb{P}^1\times\mathbb{P}^1$.}}
	Set $Z=\mathbb{P}^1\times\mathbb{P}^1$ and consider the projection $\pi_{23}\colon Y\to Z$ onto the second and third factors. We denote by $A_Y$ a curve of degree $(2,2,2)$ that is the complete intersection of two $(1,1,1)$-divisors in $Y$, and by $A$ its image in $Z$ through $\pi_{23}$. Let $X$ be the blow up of $A_Y$ in $Y$. It is not difficult to prove that $A$ has to be a curve of degree $(2,2)$ in $Z$. 
	Suppose there exists a section $S_Y$ of $\pi_{23}$ containing $A_Y$, so that $X=\mathscr{C}_B(Z;A,S_Y)$. Up to normalization, the short exact sequence associated to $S_Y$ has to be of the following form
	\begin{equation}\label{ExactSection}
		0\to \mathcal{O}\rightarrow \mathcal{O}(a,b)\oplus\mathcal{O}(a,b)\rightarrow \mathcal{O}(2a,2b)\to 0
	\end{equation}	
	for some $a,b\in\mathbb{N}$. If $a=b=0$, $A_Y$ lies in a section of $\pi_{23}$ that admits a second disjoint section. This means that $X$ arises from Construction A, which we excluded in the previous section. Therefore we can assume that $a\ne0$, and the exact sequence (\ref{ExactSection}) does not split. It follows that one necessary condition for $S_Y$ to exist is that $\Ext^1(\mathcal{O}(2a,2b),\mathcal{O})$ is nonzero. Observe that
	$\Ext^1(\mathcal{O}(2a,2b),\mathcal{O})\simeq\bigr(H^0(\mathbb{P}^1,\mathcal{O}(-2+2a))\otimes H^0(\mathbb{P}^1,\mathcal{O}(-2b))\bigr)\oplus \bigr(H^0(\mathbb{P}^1,\mathcal{O}(-2a))\otimes H^0(\mathbb{P}^1,\mathcal{O}(-2+2b))\bigr)$,
	which is nonzero if and only if  
	\begin{align}\label{ExtConditions}
		\begin{cases}
			a\ge 1\\
			b\le0	
		\end{cases} \text{ or } \begin{cases}
			a\le0\\
			b\ge 1.
		\end{cases}
	\end{align}   
	Now consider condition \emph{(I)} in Theorem \ref{ConstructionB}:
	\begin{equation*}
		(I) -K_Z+D-A\text{ is ample on }Z.
	\end{equation*}
	It is a necessary condition for $X$ to be Fano. In this setting, we have $\mathcal{O}_Z(-K_Z)=\mathcal{O}(2,2)$, $\mathcal{O}_Z(D)=\mathcal{O}(a,b)$, $\mathcal{O}_Z(A)=\mathcal{O}(2,2)$. Then
	\begin{align*}
		\mathcal{O}_Z(-K_Z+D-A)=\mathcal{O}(2+a-2,2+b-2)=\mathcal{O}(a,b)
	\end{align*}
	which is ample if and only if $a>0$ and $b>0$. However, this is in contrast with conditions (\ref{ExtConditions}), so we conclude that $X$, which we know is Fano, cannot arise from such a construction, and the blown-up of such $(2,2,2)$-curve in $Y$ cannot be contained in any section of $\pi_{23}$.    
	
	Observe that this only proves that $X$ does not arise from Construction B with this specific choice of $\mathbb{P}^1$-bundle $Y$ and blown up locus $A_Y$, but $X$ could arise from a construction of type B with different data. We now prove that this is not the case. More precisely, we will prove that any divisorial contraction of $X$ must be the blow-up of a $(2,2,2)$-curve in $Y=\mathbb{P}^1\times\mathbb{P}^1\times\mathbb{P}^1$, which we already proved that cannot be part of a construction of type B.    
	
	Consider the alternative description of $X$ as a $(1,1,1,1)$-divisor in $\mathbb{P}^1\times\mathbb{P}^1\times\mathbb{P}^1\times\mathbb{P}^1$. We now briefly study the elementary contractions of $X$. Observe that, by Lefschetz's Theorem, since $X$ is an ample divisor in $(\mathbb{P}^1)^4$, we have an isomorphism of the spaces of curves $N_1(X)\simeq N_1((\mathbb{P}^1)^4)$. Moreover, we can identify the cones of effective curves $\NE(X)=\NE((\mathbb{P}^1)^4)$ (see, for example, \cite[Theorem 2.10]{Wan22}) 
	
	It follows that the only elementary contractions of such variety are blow-downs $\sigma\colon X\to Y:=\mathbb{P}^1\times\mathbb{P}^1\times\mathbb{P}^1$ induced by the pojections of $(\mathbb{P}^1)^4$ onto any 3 of its factors. The blown-up locus is always a curve of type $(2,2,2)$ in $\mathbb{P}^1\times\mathbb{P}^1\times\mathbb{P}^1$ that is the smooth complete intersection of two $(1,1,1)$-divisors.  If Construction B applies, $Z$ must necessarily be isomorphic to $\mathbb{P}^1\times\mathbb{P}^1$, $\pi\colon Y\to Z$ is the projection onto two of the factors of $(\mathbb{P}^1)^3$, say the second and third factors, and $A:=\pi(A_Y)$ is a $(2,2)$-curve in $Z$. However we already proved that this cannot be the case.

\section{Preliminaries for the study of Fano 4-folds arising from Construction A} \label{Standard4folds}
	In this section our goal is to provide some useful theoretic results that will allow us to give a complete list of all families of Fano 4-folds arising from Construction A with $\rho\ge 4$ and $\delta=2$.

	\begin{remark}\label{RemarkStandard4folds}
		Suppose $Z$ is a smooth Fano variety with $\dim(Z)=n-1\ge 1$ such that there exist two Cartier divisors $A$ and $D$ in $Z$ giving rise to a smooth Fano variety $X$ with $\dim(X)=n$ and $\rho_X=\rho_Z+2$ through Construction $A$. 
		Since $A$ is effective, some extremal ray $R$ of $\NE(Z)$ must satisfy $A\cdot R> 0$. Let $C$ be an irreducible curve whose class in $\NE(Z)$ generates $R$ and let $c_{R}:Z\to W$ be the elementary extremal contraction associated to $R$ (we have $\rho_W=\rho_Z-1$). In Remark \ref{Degree1curve}, we proved that if $C$ has  anticanonical degree 1, then $A\cdot C\le 0$, a contradiction. Therefore, the length of $R$ is $\ell(R)\ge 2$.
		
		We now fix $n=4$ (i.e.~$\dim(Z)=3$), and $\rho_Z\ge 2$. In that case, by means of the known classification of elementary contractions of Fano 3-folds (see \cite[Theorem 1.4.3]{IP99}), we analyze in detail all possible types of contraction with length larger than 1:
		\begin{itemize}
			\item[(i)] If $c_{R}\colon Z\to W$ is a fiber type  contraction, then $1\le\dim(W)<\dim(Z)=3$. If $\dim(W)=1$, then $\rho_Z=2$, $W\simeq\mathbb{P}^1$ and $c_R$ is either a $\mathbb{P}^2$-bundle or a quadric bundle.
			On the other hand, if $\dim(W)=2$ then $W$ is a smooth del Pezzo surface and $c_R$ is a $\mathbb{P}^1$-bundle over $W$.
			\item[(ii)] Suppose $c_{R}\colon Z\to W$ is a divisorial contraction. Then $c_R$ has to be the blow up of a smooth point in $W$. Let $E$ be the exceptional divisor of $c_R$. We have $E\simeq \mathbb{P}^2$, which implies that $\dim(N_1(E,Z))=1$ and $\delta_{Z}\ge \codim(N_1(E,Z))=\rho_{Z}-1$. But $\delta_Z\in\{0,\dots,\rho_Z-1\}$, so we get $\delta_Z=\rho_Z-1$. 
		\end{itemize}  
		
	\end{remark}
	\begin{theorem}\label{theorem4folds}
		Let $X=\mathscr{C}_A(Z;A,D)$ be a Fano 4-fold. 
		
		If $\rho_X=4$, then one of the following holds:
		\begin{itemize}
			\item (A.1) There exists a $\mathbb{P}^1$-bundle $Z\to \mathbb{P}^2$.
			\item (A.2) There exists a quadric bundle $Z\to \mathbb{P}^1$. 
			\item (A.3) There exists a $\mathbb{P}^2$-bundle $Z\to \mathbb{P}^1$. 
			\item (A.4)  $Z\simeq \text{Bl}_{p}\mathcal{Q}_3$, where $\mathcal{Q}_3$ is a smooth quadric threefold in $\mathbb{P}^4$ and $p$ is a point in $\mathcal{Q}_3$. 
		\end{itemize}
		If $\rho_X=5$, then one of the following holds:
		\begin{itemize}
			\item (B.1) There exists a $\mathbb{P}^1$-bundle $Z\to W$ (where $W$ is a del Pezzo surface with $\rho_W=2$). 
			\item (B.2) $Z$ arises from a construction of type A with base $\mathbb{P}^2$.  
		\end{itemize} 
		Finally, if $\delta_X=2$, then $\rho_X\le6$. Moreover $\rho_X=6$ and $\delta_X=2$ if and only if $Z\simeq\text{Bl}_2(\mathbb{P}^2)\times\mathbb{P}^1$.  
	\end{theorem}
	
	\begin{proof}
		Suppose $\rho_X=4$ (i.e.~$\rho_Z=2$), and consider the elementary contraction $c_R\colon Z\to W$ introduced in Remark \ref{RemarkStandard4folds}. Then there are two possibilities:
		\begin{itemize}
			\item[(i)] $c_R: Z\to W$ is a fiber-type contraction. 
			\item[(ii)] $c_R:Z\to W$ is the blow up of a smooth point in $W$.  
		\end{itemize}  
		Case (i) yields (A.1), (A.2), (A.3). By the known classification of Fano varieties that can be realized as the blow-up of a smooth point \cite{BCW02}, case (ii) yields (A.4) or $Z\simeq\Bl_p\mathbb{P}^3$, which falls in (A.1). 
		
		Suppose $\rho_X= 5$ (i.e.~$\rho_Z= 3$). 
		Again, consider the elementary contraction $c_R\colon Z\to W$ introduced in Remark \ref{RemarkStandard4folds}, which satisfies (i) or (ii) above. Since $\rho_Z=3$, Case (i) can only occur with $\dim(W)=2$, which yields (B.1). 
		Consider case (ii). If $\rho_X=5$, then $\rho_Z=3$ and $\delta_Z=\rho_Z-1=2$. We know every family of Fano 3-folds with such invariants: they are 6 families all arising from Construction A over $\mathbb{P}^2$ (see \cite[Example 22]{Cas23}). Among them, only 3 arise from the blow up of a smooth point, namely \#3-14, \#3-19 and \#3-26 in \cite[Table 6.1]{Ara+23}. 
		
		Finally, observe that $\delta_Z\le\delta_X$ by Lemma \ref{deltaLemma}, and $\delta_Z\in\{\rho_Z-1,\rho_Z-2\}$ because $Z$ has dimension 3 (see \cite[Lemma 5.1]{Del14}). Hence  if $\delta_X=2$, then $\rho_Z-2\le\delta_Z\le\delta_X=2$, which yields
		${\rho_X=\rho_Z+2\le 6}$. If $\rho_X=6$, then $\rho_Z=4$ and, if case (ii) applies, $\delta_Z=\rho_Z-1=3$. But then $\delta_X\ge\delta_Z=3$, so we rule out that option. Therefore $Z$ must have a $\mathbb{P}^1$-bundle structure over a del Pezzo surface $W$ with $\rho_W=\rho_Z-1=3$, i.e.~$W\simeq \text{Bl}_{2}\mathbb{P}^2$. By the known classification of Fano projective bundles of rank 2 over a del Pezzo surface (see \cite{SW90}), $Z$ must then be isomorphic to $\text{Bl}_{2}\mathbb{P}^2\times\mathbb{P}^1$.
		This concludes the proof. 
	\end{proof}
	
	In order to check if the Fano 4-folds we construct are toric, we need the following Lemma.
	
	\begin{lemma}\label{toricLemma}
		Let $X=\mathscr{C}_A(Z;A,D)$. Then $X$ is toric if and ony if $Z$ is toric and $A\subset Z$ is invariant under the torus action. 
	\end{lemma}
	\begin{proof}
		Suppose $Z$ is toric and $A$ is invariant under the action of the torus. We then have that $Y=\mathbb{P}_Z(\mathcal{O}_Z\oplus\mathcal{O}_Z(D))$  is a smooth toric variety \cite[\S 7.3]{CLS11}. If $\pi\colon  Y\to Z$ is the projection onto $Z$, we have that $\pi^{-1}(A)$ and the sections $G_Y$, $\hat{G}_Y$ are stable under the action of the torus. Therefore $A_Y:=\pi^{-1}(A)\cap \hat{G}_Y$ is also stable under the action of the torus, and the blow-up $X$ of $A_Y$ in $Y$ is a toric variety \cite[\S 3.3]{CLS11}. 
		
	Conversely, suppose that $X$ is toric. Then observe that the action of the torus on $X$ preserves the extremal rays in $\overline{\NE(X)}$, because the set of extremal rays is finite. In particular the exceptional divisor of $\sigma\colon X\to Y$ is torus-invariant. Moreover, as a consequence of Blanchard's lemma \cite[Proposition 4.2.1]{BSU13}, there exists a unique action of the torus on $Y$ such that $\sigma\colon X\to Y$ is equivariant with respect to that action.  These facts imply that $Y$ is toric and the blown up locus $A_Y$ is stable under the action of the torus. Similarly, it can then be proved that $Z$ is toric and $A$ is stable under the action of the torus on $Z$.      
	\end{proof}

\section{Numerical invariants}
	In this section we show the tools we use to compute the numerical invariants of the Fano 4-folds constructed in Sections \ref{Standard4foldsRho5} and \ref{Standard4foldsRho4}. Every such Fano variety $X$ arises from Construction A. We use the notation from Section \ref{SettingA}. 
	
	First we want to compute the Hodge numbers. To do so, we will use the Hodge polynomial of a smooth variety $W$. The Hodge polynomial \( e(W)(u, v) \) is given by:
	
	\[
	e(W)(u, v) := \sum_{p, q} h^{p, q}(W) u^p v^q
	\]
	where \( h^{p, q}(W) \) are the Hodge numbers.
	
	\begin{lemma}\label{HodgeLemma}
		
		Let \( W \) be a smooth variety, let \( \mathbb{P}(\mathcal{E}) \) be a \( \mathbb{P}^n \)-bundle over \( W \), and let \( W^\prime \) be the blow-up of \( W \) along a subvariety \( V \) of codimension \( c \). Then the Hodge polynomials satisfy the following relationships:
		
		\begin{align*}
			e(\mathbb{P}(\mathcal{E})) = e(W) \cdot e(\mathbb{P}^n)\;\;\;\;\;\text{ and }\;\;\;\;\;
			e(W^\prime) = e(W) + e(V) \cdot \left[ e(\mathbb{P}^{c-1}) - 1 \right]
		\end{align*}
		
	\end{lemma}
	\begin{proof}
		It follows from \cite[Introduction 0.1]{Che96} and simple computations. 
	\end{proof}
	
		The following results are true when $Z$ is rationally connected, so in particular they apply when $Z$ is Fano.
	
	\begin{corollary}
		Let $Y=\mathbb{P}(\mathcal{O}_Z\oplus\mathcal{O}_Z(D))$ a $\mathbb{P}^1$-bundle over a threefold $Z$ as in Section \ref{SettingA}. Assume that $Z$ is a smooth rationally connected variety. Then
		$h^{0,1}(Y)=h^{0,2}(Y)=h^{0,3}(Y)=0$, $h^{1,2}(Y)=h^{1,2}(Z)$, $h^{2,2}(Y)=2h^{1,1}(Z)$ and $h^{1,3}(Y)=0$.
		
		Let $X\to Y$ be the blow up of a codimension 2 subvariety $A_Y\subset Y$ as in Section \ref{SettingA}. Then 
		$h^{0,1}(X)=h^{0,2}(X)=h^{0,3}(X)=0$,
		$h^{1,2}(X)=h^{1,2}(Z)+h^{0,1}(A)$,
		$h^{2,2}(X)=2h^{1,1}(Z)+h^{1,1}(A)$, and
		$h^{1,3}(X)=h^{0,2}(A)$.
	\end{corollary}
	\begin{proof}
		After noticing that, since $Z$ is rationally connected, $h^{0,i}(Z)=0$ for all $i>0$, it is enough to apply Lemma \ref{HodgeLemma} to the $\mathbb{P}^1$-bundle $Y\to Z$ and then to the blow up $X\to Y$. 
	\end{proof}
	Therefore, in order to be able to compute the Hodge numbers of $X$, we need to compute the Hodge numbers of $A$. The following Lemma provides some tools in that direction.
	\begin{lemma}\label{LemmaA}
		Let $A$ be a smooth irreducible surface in a smooth rationally connected 3-fold $Z$. Then $h^{0,1}(A)=h^1(\mathcal{O}_Z(K_Z+A))$, and $h^{0,2}(A)=h^0(\mathcal{O}_Z(K_Z+A))$.
		
		Moreover 
		$h^{1,1}(A)=\chi_{top}(A)+4h^{0,1}(A)-2h^{0,2}(A)-2
		$, where $\chi_{top}(A):=\sum_{i}(-1)^ib_i(A)$ is the topological characteristic of $A$. 
	\end{lemma}
	\begin{proof}
		Consider the short exact sequence 
		\begin{align}
			0\to \mathcal{O}_Z(-A)\to \mathcal{O}_Z\to\mathcal{O}_A\to 0.
		\end{align}
		Since $Z$ is rationally connected, we have $H^i(Z,\mathcal{O}_Z)=0$ for all $i>0$, so that the induced long exact secquence in cohomology yields $H^1(\mathcal{O}_A)\simeq H^2(\mathcal{O}_Z(-A))$ and $H^2(\mathcal{O}_A)\simeq H^3(\mathcal{O}_Z(-A))$. By Serre duality, \sloppy${H^i(Z,\mathcal{O}_Z(-A))\simeq H^{3-i}(Z,\mathcal{O}_Z(K_Z+A))^*}$. 
		
		Finally, using the fact that $b_i(A)=b_{4-i}(A)$,  $h^{p,q}(A)=h^{q,p}(A)$, and $b_i(A)=\sum_k h^{k,i-k}(A)$, we get
		\begin{align*}
			\chi_{top}(A)=2b_0(A)-2b_1(A)+b_2(A)= 2-4 h^{0,1}(A)+2h^{0,2}(A)+h^{1,1}(A).  
		\end{align*}
		
	\end{proof}
	\begin{corollary}
		Let $A$ be a smooth irreducible surface in a smooth rationally connected 3-fold $Z$. 
		
		If $K_Z+A$ is not effective, then 
		\begin{align*}
			h^{0,2}(A)=0\;\;\;\;\;\;
			h^{0,1}(A)=1-\frac{1}{12}(K_A^2+\chi_{top}(A))\;\;\;\;\;\;
			h^{1,1}(A)=\chi_{top}(A)+4h^{0,1}(A)-2
		\end{align*}
		
		If $A\in |-K_Z|$, then
		\begin{align*}
			h^{0,2}(A)=1\;\;\;\;\;\;
			h^{0,1}(A)=0\;\;\;\;\;\;
			h^{1,1}(A)=\chi_{top}(A)-4
		\end{align*}
	\end{corollary}
	\begin{proof}
		It is enough to apply Lemma \ref{LemmaA} and notice that if $K_Z+A$ is not effective, then $h^0(K_Z+A)=0$. For computing $h^{0,1}(A)$, observe that, by the Noether formula, $\chi(\mathcal{O}_A)=\frac{1}{12}(K_A^2+\chi_{top}(A))$. After recalling that $\chi(\mathcal{O}_A)=h^{0,0}(A)-h^{0,1}(A)+h^{0,2}(A)$, we get the formula for $h^{0,1}$. 
		
		On the other hand, if $A\in|-K_Z|$, then (again by Lemma \ref{LemmaA}), we get $h^{0,1}(A)=h^{1}(Z,\mathcal{O}_Z)=0$ because $Z$ is rationally connected, and $h^{0,2}(A)=h^0(Z,\mathcal{O}_Z)=1$. 
	\end{proof}
	
	The following facts provide us all necessary tools for computing the invariants $c_1(X)^4$, $c_1(X)^2c_2(X)$ and $\chi(-K_X)$. Recall that, if $X$ is Fano, by Kodaira vanishing we have $\chi(-K_X)=h^0(-K_X)$. 
	
	\begin{proposition}[{{\cite{CR22}, Lemma 3.2}}]\label{PropositionblowupW}
		Let \(W\) be a smooth projective variety, \(\dim W = 4\), and let \(\alpha : \tilde{W} \to W\) be the blow-up of \(W\) along a smooth irreducible surface \(V\). Then
		\begin{itemize}
			\item[(i)] \(K^4_{\tilde{W}} = K^4_W - 3(K_W|_V)^2 - 2K_V \cdot K_W|_V + c_2(\mathcal{N}_{V/W}) - K_V^2\);
			\item[(ii)] \(K^2_{\tilde{W}} \cdot c_2(\tilde{W}) = K^2_W \cdot c_2(W) - 12 \chi(O_V) + 2K_V^2 - 2K_V \cdot K_W|_V - 2c_2(\mathcal{N}_{V/W})\);
			\item[(iii)] \(\chi(O_{\tilde{W}}(-K_{\tilde{W}})) = \chi(O_W(-K_W)) - \chi(O_V) - \frac{1}{2}((K_W|_V)^2 + K_V \cdot K_W|_V)\).
		\end{itemize}
	\end{proposition}
	\begin{remark}
		In the next sections, we will use Proposition \ref{PropositionblowupW} in a setting where $V$ is a complete intersection of two divisors, hence the normal bundle $\mathcal{N}_{V/W}$ is decomposable, and its Chern classes are easy to compute.  
	\end{remark}
	\begin{proposition}[{{\cite{Sec23}, Proposition 3.4}}]
		Let \(W\) be a smooth projective variety of dimension 3, and $\mathcal{E}$ a rank 2 vector bundle on \(W\). Let \(\beta : \mathbb{P}(\mathcal{E}) \to W\) be a \(\mathbb{P}^1\)-bundle over \(W\). Then
		\begin{itemize}
			\item[(i)] \(K^4_{\mathbb{P}(\mathcal{E})} = -8K_W \cdot c_1(\mathcal{E})^2 + 32K_W \cdot c_2(\mathcal{E}) - 8K_W^3\);
			\item[(ii)] \(K^2_{\mathbb{P}(\mathcal{E})} \cdot c_2(\mathbb{P}(\mathcal{E})) = -2K_W \cdot c_1(\mathcal{E})^2 + 8K_W \cdot c_2(\mathcal{E}) - 2K_W^3 - 4K_W \cdot c_2(W)\);
			\item[(iii)] \(\chi(O_{\mathbb{P}(\mathcal{E})}(-K_{\mathbb{P}(\mathcal{E})})) = \chi(O_{\mathbb{P}(\mathcal{E})}) + 6K_W \cdot c_2(\mathcal{E}) - \frac{1}{2}(3K_W^3 + 3K_W \cdot c_1(\mathcal{E})^2) - \frac{1}{3} K_W \cdot c_2(W)\).
		\end{itemize}
	\end{proposition}
\section{Fano 4-folds arising from Construction A with $\rho= 6$ and $\delta=2$}
	In Theorem \ref{theorem4folds}, we proved that the only Fano 4-folds $X$ with $\rho_X= 6$ and $\delta_X=2$ arising from Construction A are the ones which project onto a smooth Fano base $Z$ that is isomorphic to $\text{Bl}_2\mathbb{P}^2\times\mathbb{P}^1$. We now determine all choices of $A$ and $D$ in $Z$ that yield a construction of type A. For this purpose, we describe the effective and ample cones of $Z$. In $N_1( \text{Bl}_2\mathbb{P}^2)$, let $\mathbf{f}$ be the class of the strict transform of the line passing through the blown up points in $\mathbb{P}^2$, and let $\mathbf{e_1}$, $\mathbf{e_2}$ be the classes of the two exceptional divisors. 
	
	Call
	$U_1:=\mathbf{f}\times \mathbb{P}^1$, $U_2:=\mathbf{e_1}\times\mathbb{P}^1$,
	$U_3:=\mathbf{e_2}\times\mathbb{P}^1$ $U_4:=\text{Bl}_2\mathbb{P}^2\times \{*\}$. Then
	$N^1(Z)$ is generated by $U_1,U_2,U_3, U_4$. 
	
	Call $u_1:=\mathbf{f}\times\{*\}$, $u_2:=\mathbf{e_1}\times \{*\}$,
	$u_3:=\mathbf{e_2}\times \{*\}$ $u_4:=\{*\}\times\mathbb{P}^1$. Then $N_1(Z)$ is generated by $u_1,u_2,u_3, u_4$.
	
	If $D\in N^1(Z)$, it can be written as $p_1^{*}(D_1)\otimes p_2^{*}(D_2)$ where $D_1\in N^1(\mathbb{F}_1)$ and $D_2\in N^1(\mathbb{P}^1)$. By the K\"unneth formula,
	\[
	H^0(Z,\mathcal{O}_Z(D))=H^0(\mathbb{F}_1,\mathcal{O}_Z(D_1))\otimes H^0(\mathbb{P}^1, \mathcal{O}_Z(D_2)).
	\]
	Hence, $\Eff(Z)=\{p_1^{*}D_1\otimes p_2^{*}D_2| \text{ }D_1\text{ is effective in }{\Bl}_{2}\mathbb{P}^2\text{, and }D_2\text{ is effective in }\mathbb{P}^1\}=\sum_{i=1}^{4}\mathbb{R}_{\ge 0}[U_i]$.
	The following intersection table holds:
	
	\begin{center}
		\begin{tabular}{|c||c|c|c|c|}
			\hline
			$\cdot$ & $u_1$ & $u_2$ & $u_3$ & $u_4$\\
			\hline\hline
			$U_1$ &  -1 & 1 & 1 &0  \\
			\hline
			$U_2$ & 1 & -1 & 0 & 0 \\
			\hline 
			$U_3$ & 1 & 0 & -1 & 0 \\
			\hline		
			$U_4$ & 0 & 0&0 & 1\\
			\hline
		\end{tabular}
	\end{center}
	
	It is then easy to check that $\NE(Z)$ is generated by $u_1,u_2,u_3,u_4$.
	By simple computations, we get $\Nef(Z)=\mathbb{R}_{\ge 0}[U_1+U_2+U_3]+\mathbb{R}_{\ge 0}[U_1+U_2]+\mathbb{R}_{\ge 0}[U_1+U_3]+\mathbb{R}_{\ge 0}[U_4]$. 
	
	Recall that
	\[-K_Z\sim p_{1}^{*}(-K_{\text{Bl}_2\mathbb{P}^2})+p_2^{*}(-K_{\mathbb{P}^1})\sim 3U_1+2U_2+2U_3+2U_4=(U_1+U_2+U_3)+(U_1+U_2)+(U_1+U_3)+2U_4.\]
	
	Let $(a_1,a_2,a_3,a_4)$ be the coordinates of a smooth irreducible surface $A$ in $Z$ with respect to the basis $\{U_i\}$. Let $(d_1,d_2,d_3,d_4)$ be the coordinates of a divisor $D$, $(d_1',d_2',d_3',d_4')$ be the coordinates of a divisor $D'$. We now study all possible choices of $A$, $D$, $D'$ such that $D+D'=A$ is a smooth irreducible surface, and $-K_Z-D$ and $-K_Z-D'$ are ample. In other words, we look for data that satisfy the conditions in Proposition \ref{ConstructionA} which yield $X=\mathscr{C}_A(Z;A,D)$ Fano.
	
	Now, 
	\begin{align*}
		-K_Z-D =&  (3-d_1)U_1+(2-d_2)U_2+(2-d_3)U_3+(2-d_4)U_4\\
		=& (1+d_1-d_2-d_3)(U_1+U_2+U_3)+(1-d_1+d_3)(U_1+U_2)+\\
		&+(1-d_1+d_2)(U_1+U_3)+(2-d_4)U_4.\\
	\end{align*}
	Therefore
	\begin{align*}
-K_Z-D\text{ is ample }\iff
d_1\ge d_2+d_3,\;
d_1\le d_2,\;
d_1\le d_3,\text{ and }
d_4\le 1.
	\end{align*}
	The same holds for $D'$'s coordinates. Observe that since $A$ is effective, all its coordinates must be non-negative, and since it is the sum of $D$ and $D'$, we must have $0\le a_2+a_3\le a_1\le \min(a_2,a_3)$ and $0\le a_4 \le  2$.
	
	We may assume $0\le a_2\le a_3$, so that $a_2+a_3\le a_1\le a_2$ implies that $a_3 = 0$, which yields $a_1=a_2=0$ as well. Therefore, the only admissible choices of $A$ are the following: $A=U_4$ and $A=2U_4$. However, the latter cannot represent a smooth irreducible hypersurface in $Z$, because it is the pull back by the projection $p_2$ of a non-reduced divisor in $\mathbb{P}^1$. Therefore the only possibility for $A$ is $A=U_4$. There is only one choice of $D, D'$ for this hypersurface: $D,D'\in\{0, U_4\}$, which yields $X\simeq \text{Bl}_2\mathbb{P}^2\times\text{Bl}_2\mathbb{P}^2$. This proves the following theorem: 
	\begin{theorem}\label{Rho6}
		The only Fano 4-fold with $\rho_X\ge6$ and $\delta_X=2$ arising from Construction A is \sloppy ${X\simeq \text{Bl}_2\mathbb{P}^2\times\text{Bl}_2\mathbb{P}^2}$. 
	\end{theorem}  
    
\section{Fano 4-folds arising from Construction A with $\rho=5$ and $\delta=2$}\label{Standard4foldsRho5}
	
	According to Theorem \ref{theorem4folds}, given a Fano 4-fold $X$ arising from Construction A with $\rho_X=5$, there are two possibilities for $Z$. 
	\begin{itemize}
		\item[(B.1)] $Z$ has a $\mathbb{P}^1$-bundle structure $p:Z\to W$ over a del Pezzo surface $W$. 
		\item[(B.2)] $Z$ arises from a construction of type $A$ with base $\mathbb{P}^2$.  	
	\end{itemize}	
	
	A careful study of those cases leads to a complete classification of Fano 4-folds with $\rho=5$ that arise from Construction A. All of them have Lefschetz defect $\delta =2$. Indeed, let $X=\mathscr{C}_A(Z;A,D)$ be a Fano 4-fold with $\rho_X=5$. The base $Z$ is a Fano 3-fold with $\rho_Z=3$, therefore \sloppy ${\delta_Z\in\{\rho_Z-2,\rho_Z-1\}=\{1,2\}}$. By Lemma \ref{deltaLemma}, we know that $2\le \delta_X\le 3$. Now, Fano 4-folds with defect $\delta=3$ and $\rho =5$ are known (see \cite{CR22}), and none of them appears in the list of our 50 families. 
	
	We ultimately obtain the following theorem.
	\begin{theorem}\label{AllRho5}
		There are 50 distinct families of Fano 4-folds $X$ with $\rho_X=5$ arising from Construction A, and they are all listed in Table \ref{FinalTableRho5}. All of them have $\delta_X=2$. Among them, 27 are not toric nor products of lower dimensional varieties.    
	\end{theorem} 
	
	We find that 25 of those 27 \emph{new} Fano 4-folds project onto a Fano $\mathbb{P}^1$-bundle $Z$ over a del Pezzo surface (see Section \ref{Case1}), while the other 2 project onto a Fano variety $Z$ that is the blowing up of a smooth point on a smooth 3-fold (see Section \ref{Case2}). 
	
	As a side comment, it is worth noticing that the Fano polytopes of the toric Fano 4-folds arising from Construction A  with $\rho=5$ are all of type $Q_n$ with $n\in\mathbb{N}$ (see \cite{Bat99}).

	\subsection{Case (B.1)}\label{Case1}
	In this section, we are going to list every Fano 4-fold $X=\mathscr{C}_A(Z;A,D)$ arising from a Fano variety $Z$ that has a $\mathbb{P}^1$-bundle structure $p:Z\to W$ over a del Pezzo surface (case (B.1) of Theorem \ref{theorem4folds}).  
	Since $\rho_Z=\rho_X-2=3$, we need $\rho_W=\rho_Z-1=2$. This implies that $W\simeq \mathbb{P}^1\times\mathbb{P}^1$ or $W\simeq \mathbb{F}_1$. 
	In their paper \cite{SW90}, M. Szurek and J. Wi\'sniewski studied Fano bundles of rank 2 on surfaces.
	From their analysis, we deduce that there are only seven families of Fano threefolds $Z$ that satisfy our assumptions: \#3-17, \#3-24, \#3-25, \#3-27, \#3-28, \#3-30, \#3-31. For each family, it is possible to describe the effective cone, as well as the nef cone. This allows us to find for each $Z$ all choices of $A$ and $D$ that satisfy the conditions in Proposition \ref{ConstructionA} that yield $X$ Fano. We will explain the details of such process in the following paragraphs.      
	
	\paragraph{1. \fbox{$Z=$Divisor of degree $(1,1,1)$ in $\mathbb{P}^1\times\mathbb{P}^1\times\mathbb{P}^2$ (\#3-17).}}
	Since $Z$ is an ample divisor, by the Lefschetz hyperplane theorem we have an isomorphism of the spaces of curves $N_1(\mathbb{P}^1\times\mathbb{P}^1\times\mathbb{P}^2)\simeq N_1(Z)$. Similarly to the case of a $(1,1,1,1)$-divisor in $(\mathbb{P}^1)^4$ (see Section \ref{General3folds}), it can be proved that the elementary contractions of $Z$ are the restrictions of the projections of $\mathbb{P}^1\times\mathbb{P}^1\times\mathbb{P}^2$ onto its factors. We denote by $H_1$, $H_2$ and $H_3$ the canonical generators of the nef cone of $\mathbb{P}^1\times\mathbb{P}^1\times\mathbb{P}^2$ and let $H_i^Z$ be their restrictions to $Z$, which generate $\Nef(Z)$. Observe that $p_{12}|_Z:Z\to \mathbb{P}^1\times\mathbb{P}^1$ is a $\mathbb{P}^1$-bundle ($\mathcal{O}(Z|_{p_{12}^{-1}(\{*\})})\simeq \mathcal{O}_{\mathbb{P}^2}(1)$).
	On the other hand $p_{23}|_Z:Z\to \mathbb{P}^1\times\mathbb{P}^2$ is a birational divisorial contraction which is the blow up of a curve. Let $E_1$ be its exceptional divisor. Then $-K_Z=p_{23}^{*}(-K_{\mathbb{P}^1\times\mathbb{P}^2})-E_1 $.
	Hence
	$E_1\equiv-H_1^Z+H_2^Z+H_3^Z$ (recall that $-K_Z=-K_{\mathbb{P}^1\times\mathbb{P}^1\times\mathbb{P}^2}|_Z-Z|_Z=H_1^Z+H_2^Z+2H_3^Z$).
	Similarly $E_2\equiv H_1^Z-H_2^Z+H_3^Z$ is the exceptional divisor of $p_{13}|_Z$. 
	
	Hence $\Nef(Z)=\mathbb{R}_{\ge 0}H_1^Z+\mathbb{R}_{\ge 0}H_2^Z+\mathbb{R}_{\ge 0}H_3^Z$ and $\Eff(Z)=\mathbb{R}_{\ge 0}H_1^Z+\mathbb{R}_{\ge 0}H_2^Z+\mathbb{R}_{\ge 0}E_1+\mathbb{R}_{\ge 0}E_2$.
	This is enough to apply Proposition \ref{ConstructionA} and find all possible choices of $A$, $D$ in $Z$ that yield a Fano 4-fold $X$ (see Table \ref{3-17}).
	\begin{small}
		\begin{center}
			\begin{longtable}{|c | c ||c| c |c|c|c|c|c|c|c|} 
				\hline
				$A$ & $D$ & \# &$c_1^4$ & $c_1^2c_2$&$h^0(-K_X)$&$h^{1.2}$&$h^{2,2}$ & $h^{1,3}$&Toric?&Product?  \\ 
				\hline\hline
				$H_3^Z$& $0/H_3^Z$ &1& 236&152&53&0&10&0&No&No\\
				\hline
				$2H_3^Z$ & $H_3^Z$ &2& 184&136&43&0&12&0&No&No\\
				\hline
				\caption{Choices of $A$, $D$ in $Z$= divisor (1,1,1) in $\mathbb{P}^1\times\mathbb{P}^1\times\mathbb{P}^2$.}
				\label{3-17}
			\end{longtable}
		\end{center}
	\end{small}
	\paragraph{2. \fbox{$Z=$Complete intersection of degree $(1,1,0)$ and $(0,1,1)$ in $\mathbb{P}^1\times\mathbb{P}^2\times\mathbb{P}^2$ (\#3-24).}}
	Let $p_i$ be the projection of $Z$ onto its $i$-th factor ($i=1,2,3$). We use the following notation for line bundles on $Z$: $\mathcal{O}(a,b,c):=p_1^{*}\mathcal{O}_{\mathbb{P}^1}(a)\otimes p_2^{*}\mathcal{O}_{\mathbb{P}^2}(b)\otimes p_3^{*}\mathcal{O}_{\mathbb{P}^2}(c)$. 
	
	First of all, notice that if $U_1=\mathcal{O}(1,1,0)$ and $U_2=\mathcal{O}(0,1,1)$, then $U_1\simeq \mathbb{F}_1\times\mathbb{P}^2$ and $U_2\simeq \mathbb{P}^1\times P$,
	where $P$ is a divisor of degree (1,1) in $\mathbb{P}^2\times\mathbb{P}^2$, i.e.~is the projective bundle $\mathbb{P}_{\mathbb{P}^2}(T_{\mathbb{P}^2})$. 
	
	Consider the projection $p_{12}:\mathbb{P}^1\times\mathbb{P}^2\times\mathbb{P}^2\to\mathbb{P}^1\times\mathbb{P}^2$. Clearly 
	$p_{12}(U_1)=\mathbb{F}_1\subset \mathbb{P}^1\times\mathbb{P}^2$ and $p_{12}(U_2)=\mathbb{P}^1\times p_{1}(P)=\mathbb{P}^1\times\mathbb{P}^2$. 
	Hence $p_{12}(Z)=\mathbb{F}_1\subset \mathbb{P}^1\times\mathbb{P}^2$, and $p_{12}|_{Z}: Z\to \mathbb{F}_1$ is a $\mathbb{P}^{1}$-bundle. It is an extremal contraction of fiber type that contracts all curves in the class of a fiber of the projective bundle $p_1:P\to \mathbb{P}^2$, i.e.~every line of the following form: $\{*\}\times\{*\}\times \mathcal{O}_{\mathbb{P}^2}(1)$ contained in $Z$. 
	
	Consider now the projection $p_{13}: \mathbb{P}^1\times\mathbb{P}^2\times\mathbb{P}^2\to\mathbb{P}^1\times\mathbb{P}^2$. Clearly $p_{13}|_{U_1}:\mathbb{F}_1\times\mathbb{P}^2\to\mathbb{P}^1\times\mathbb{P}^2$ coincides with the classical $\mathbb{P}^1$-bundle from $\mathbb{F}_1$ to $\mathbb{P}^1$ on the first component, and is the identity of $\mathbb{P}^2$ on the second component. We have $
	p_{13}(U_1)=\mathbb{P}^1\times\mathbb{P}^2$, and $
	p_{13}(U_2)=\mathbb{P}^1\times p_2(P)=\mathbb{P}^1\times\mathbb{P}^2$.
	Therefore $p_{13}(Z)=\mathbb{P}^1\times\mathbb{P}^2$, and $p_{13}|_{Z}$ is a birational divisorial extremal contraction that contracts curves in the class of $\{*\}\times\mathcal{O}_{\mathbb{P}^2}(1)\times\{*\}$, contained in $Z$.
	
	Finally, consider the projection $p_{23}: \mathbb{P}^1\times\mathbb{P}^2\times\mathbb{P}^2\to\mathbb{P}^2\times\mathbb{P}^2$. The restriction $p_{23}|_{U_1}:\mathbb{F}_1\times\mathbb{P}^2\to\mathbb{P}^2\times\mathbb{P}^2$ concides with the blowing up $\mathbb{F}_1\to\mathbb{P}^2$ on the first component and the identity of $\mathbb{P}^2$ on the second component. We have
	$p_{23}(U_1)=\mathbb{P}^2\times\mathbb{P}^2$ and $p_{23}(U_2)=P$.   
	Hence $p_{23}(Z)=P$ and $p_{23}|_{Z}:Z\to P$ is a birational divisorial extremal contraction that contracts all curves in the class of $\mathbb{P}^1\times\{*\}\times\{*\}$, contained in $Z$.
	
	So we have  that every extremal contraction of $Z$ is the restriction of an extremal contraction of $\mathbb{P}^1\times\mathbb{P}^2\times\mathbb{P}^2$, hence the restriction $i_{*}:\Nef(\mathbb{P}^1\times\mathbb{P}^2\times\mathbb{P}^2)\to \Nef(Z)$ is an isomorphism. Let $H_1:=\mathcal{O}(1,0,0)$, $H_2:=\mathcal{O}(0,1,0)$, $H_3:=\mathcal{O}(0,0,1)$ be the generators of the nef cone in $\mathbb{P}^1\times\mathbb{P}^2\times\mathbb{P}^2$, and $H_1^Z$, $H_2^Z$, $H_3^Z$ their restrictions to $Z$.
	
	Call $E_1$ and $E_2$ the exceptional divisor respectively of $p_{12}$ and $p_{13}$. It is not difficult to prove that $E_1\equiv -H_1^Z+H_2^Z$ and $E_2\equiv H_1^Z-H_2^Z+H_3^Z$. We have $\Nef(Z)=\mathbb{R}_{\ge 0}H_1^Z+\mathbb{R}_{\ge 0}H_2^Z+\mathbb{R}_{\ge 0}H_3^Z$ and $
	\Eff(Z)=\mathbb{R}_{\ge 0}H_1^Z+\mathbb{R}_{\ge 0}E_1+\mathbb{R}_{\ge 0}E_2$.
	
	Recall that, by the adjunction formula, $-K_Z=-K_{\mathbb{P}^1\times\mathbb{P}^2\times\mathbb{P}^2}|_Z -\det(\mathcal{N}_{Z/\mathbb{P}^1\times\mathbb{P}^2\times\mathbb{P}^2})= 2H_1^Z+3H_2^Z+3H_3^Z-(H_1^Z+2H_2^Z+H_3^Z)=H_1^Z+H_2^Z+2H_3^Z$. This is enough to apply Proposition \ref{ConstructionA} and find all possible choices of $A$, $D$ in $Z$ that yield a Fano 4-fold $X$. They are listed in Table \ref{3-24}.
	\begin{small}
		\begin{center}
			\begin{longtable}{|c | c ||c| c |c|c|c|c|c|c|c|} 
				\hline
				$A$ & $D$ &\#& $c_1^4$ &$c_1^2c_2$&$h^0(-K_X)$& $h^{1,2}$&$h^{2,2}$&$h^{1,3}$&Toric?&Product? \\ 
				\hline\hline
				$H_3^Z$& $0/H_3^Z$ &1& 278&164&61&0&9&0&No&No\\
				\hline
				$2H_3^Z$ & $H_3^Z$ &2& 220&148&50&0&10&0&No&No\\
				\hline
				\caption{Choices of $A$, $D$ in $Z$=complete intersection of degree (1,1,0) and (0,1,1) in $\mathbb{P}^1\times\mathbb{P}^2\times\mathbb{P}^2$.}
				\label{3-24}
			\end{longtable}
		\end{center}
	\end{small}
	\paragraph{3. \fbox{$Z=\mathbb{P}_{\mathbb{P}^1\times\mathbb{P}^1}(\mathcal{O}(-1,0)\oplus\mathcal{O}(0,-1))$ (\#3-25).}} 
	Observe that $Z$ is isomorphic to the blow up of $\mathbb{P}^3$ along two disjoint lines $l_1$ and $l_2$. We denote by $E_1$ and $E_2$ the exceptional divisors over the two blown up lines in $\mathbb{P}^3$, and by $H$ the transform of a general hyperplane in $\mathbb{P}^3$.   
	
	In $N_1(Z)$, let $\mathbf{e_1}$ and $\mathbf{e_2}$ be the classes of one dimensional fibers respectively of $E_1$ and $E_2$, and let $\mathbf{f}$ be the class of the transform of a generic line intersecting both the blown up lines non trivially. They generate the cone of effective 1-cycles $\NE(Z)$ as well as $N_1(Z)$. 
	
	The classes of $H$, $E_1$, $E_2$ generate $N^1(Z)$ and \sloppy${\Nef(Z)=\mathbb{R}_{\ge 0}[H]+\mathbb{R}_{\ge 0}[H-E_1]+\mathbb{R}_{\ge 0}[H-E_2]}$. Moreover the effective cone is \sloppy${\Eff(Z)=\mathbb{R}_{\ge 0}[H-E_1]+\mathbb{R}_{\ge 0}[H-E_2]+\mathbb{R}_{\ge 0}[E_1]+\mathbb{R}_{\ge 0}[E_2]}$ (see \cite{DPU17}).
	Recall that $-K_Z\equiv2H+(H-E_1)+(H-E_2)$. This is enough to use Proposition \ref{ConstructionA} and find all possible choices of $A$ and $D$ that yield a Fano 4-fold $X$. They are listed in Table \ref{3-25}.
	\begin{small}
		\begin{center}
			\begin{longtable}{|c|c||c|c|c|c|c|c|c|c|c|}
				\hline
				$A$ & $D$ &\#& $c_1^4$ & $c_1^2c_2$ & $h^0(-K_X)$ & $h^{1,2}$& $h^{2,2}$&$h^{1,3}$&Toric?&Product?\\
				\hline\hline
				\multirow{2}{*}{$E_1$} 
				&$H$/$-H+E_1$ &1& 331 & 178 &71 &0&8&0&$Q_{12}$&No\\
				\hhline{|~|-|-|-|-|-|-|-|-|-|-|}
				& $E_1$/$0$ &2& 310 & 172 & 67 &0&8&0&$Q_{16}$&No\\
				\hline\hline
				$H$ & $H$/$0$ &3& 283 & 166 &62&0&9&0&No&No \\
				\hline\hline 
				$2H$
				& $H$ &4& 214 & 148 &49&0&12&0&No&No\\
				\hline
				\caption{Choices of $A$, $D$ in $Z=\mathbb{P}_{\mathbb{P}^1\times\mathbb{P}^1}(\mathcal{O}(-1,0)\oplus\mathcal{O}(0,-1))$.}
				\label{3-25}
			\end{longtable}
		\end{center}
	\end{small}
	\paragraph{4. \fbox{$Z=\mathbb{P}^1\times\mathbb{P}^1\times\mathbb{P}^1$ (\#3-27).}}
	Call
	$H_1:=\{*\}\times \mathbb{P}^1\times \mathbb{P}^1$, $H_2:=\mathbb{P}^1\times\{*\}\times\mathbb{P}^1$, $H_3:=\mathbb{P}^1\times\mathbb{P}^1\times \{*\}$. Both
	$\Nef(Z)$ and $\Eff(Z)$ are generated by $H_1,H_2,H_3$. 
	Recall that \sloppy ${-K_Z\equiv 2H_1+2H_2+2H_3}$. This is enough to apply Proposition \ref{ConstructionA} and find all possible choices of $A$ and $D$ (denoted by their coordinates with respect to $H_1$, $H_2$, $H_3$). They are listed in Table \ref{3-27}.
	\begin{small}
		\begin{center}
			\begin{longtable}{|c|c||c|c|c|c|c|c|c|c|c|}
				\hline
				$A$ & $D$ &\#& $c_1^4$ & $c_1^2c_2$ & $h^0(-K_X)$& $h^{1,2}$&$h^{2,2}$&$h^{1,3}$&Toric?&Product?\\
				\hline\hline
				\multirow{6}{*}{(0,0,1)} & (-1,-1,0)/(1,1,1) &1& 394 & 196 & 83 &0&8&0&$Q_{3}$&No\\
				\hhline{|~|-|-|-|-|-|-|-|-|-|-|}
				& (-1,0,0)/(1,0,1) &2& 352 & 184 & 75 &0&8&0&$Q_{8}$&(\#4-11)$\times\mathbb{P}^1$\\
				\hhline{|~|-|-|-|-|-|-|-|-|-|-|}
				& (0,0,0)/(0,0,1) &3& 336 & 180 &  72 &0&8&0&$Q_{11}$&(\#4-10)$\times\mathbb{P}^1$\\
				\hhline{|~|-|-|-|-|-|-|-|-|-|-|} 
				& (-1, -1, 1)/(1,1,0) &4& 330 & 180 & 71 &0&8&0&$Q_{13}$&No \\
				\hhline{|~|-|-|-|-|-|-|-|-|-|-|}
				& (-1,0,1)/(1,0,0) &5& 320 & 176 & 69 &0&8&0&$Q_{15}$&(\#4-9)$\times\mathbb{P}^1$\\
				\hhline{|~|-|-|-|-|-|-|-|-|-|-|}
				& (0,1,0)/(0,-1,1) &6& 320 & 176 & 69 &0&8&0&$Q_{15}$&(\#4-9)$\times\mathbb{P}^1$\\
				\hhline{|~|-|-|-|-|-|-|-|-|-|-|}
				& (-1,1,0)/(1,-1,1) &7& 310 & 172 & 67 &0&8&0&$Q_{16}$&No\\
				\hline\hline
				\multirow{5}{*}{(0,1,1)} & (-1,0,0)/(1,1,1) &8& 330 & 180 &71 &0&8&0&No&No\\
				\hhline{|~|-|-|-|-|-|-|-|-|-|-|}
				& (0,0,0)/(0,1,1) &9& 304 & 172 &66&0&8&0&No&(\#4-8)$\times\mathbb{P}^1$\\
				\hhline{|~|-|-|-|-|-|-|-|-|-|-|}
				& (0,0,1)/(0,1,0) &10& 288 & 168 &63&0&8&0&No&(\#4-7)$\times\mathbb{P}^1$\\
				\hhline{|~|-|-|-|-|-|-|-|-|-|-|}
				& (-1,0,1)/(1,1,0) &11& 288 & 168 &63&0&8&0&No&No\\
				\hhline{|~|-|-|-|-|-|-|-|-|-|-|}
				& (-1,1,1)/(1,0,0) &12& 278 & 164 &61&0&8&0&No&No\\
				
				\hline\hline 
				\multirow{3}{*}{(0,1,2)} & (-1,0,1)/(1,1,1) &13& 266 & 164 &59 &0&8&0&No&No\\
				\hhline{|~|-|-|-|-|-|-|-|-|-|-|}
				& (0,0,1)/(0,1,1) &14& 256 & 160 &57&0&8&0&No&(\#4-5)$\times\mathbb{P}^1$\\
				\hhline{|~|-|-|-|-|-|-|-|-|-|-|}
				& (-1,1,1)/(1,0,1) &15& 246 & 156 &55&0&8&0&No&No\\
				\hline\hline
				\multirow{2}{*}{(0,2,2)} & (-1,1,1)/(1,1,1) &16& 224 & 152 & 51&1&8&0&No&No\\
				\hhline{|~|-|-|-|-|-|-|-|-|-|-|}
				& (0,1,1) &17& 224 & 152 &51&1&8&0&No&(\#4-2)$\times\mathbb{P}^1$\\
				\hline\hline
				\multirow{2}{*}{(1,1,1)} & (0,0,0)/(1,1,1) &18& 282 & 168 & 62&0&10&0&No&No\\
				\hhline{|~|-|-|-|-|-|-|-|-|-|-|}
				& (0,0,1)/(1,1,0) &19& 256 & 160 & 57&0&10&0&No&No\\
				\hline\hline
				\multirow{2}{*}{(1,1,2)} & (0,0,1)/(1,1,1) &20& 234 & 156 &53& 0&12&0&No&No\\
				\hhline{|~|-|-|-|-|-|-|-|-|-|-|}
				& (0,1,1)/(1,0,1) &21& 224 & 152 &51& 0&12&0&No&No\\
				\hline\hline 
				(1,2,2) & (0,1,1)/(1,1,1) &22& 202 & 148 &47&0&16&0&No&No\\
				\hline\hline 
				(2,2,2) & (1,1,1) &23& 180 & 144 &43&0&26&1&No&No\\
				\hline
				\caption{Choices of $A$, $D$ in $Z=\mathbb{P}^1\times\mathbb{P}^1\times\mathbb{P}^1$.}
				\label{3-27}
			\end{longtable}
		\end{center}
	\end{small}
	\paragraph{5. \fbox{$Z=\mathbb{F}_1\times\mathbb{P}^1$ (\#3-28).}}
	Call
	$U_1:=\mathbf{f}\times \mathbb{P}^1$, $U_2:=\mathbf{e}\times\mathbb{P}^1$, $U_3:=\mathbb{F}_1\times \{*\}$, so that
	$N^1(Z)$ is generated by $U_1,U_2,U_3$. 
	Call $u_1:=\mathbf{f}\times\{*\}$, $u_2:=\mathbf{e}\times \{*\}$, $u_3:=\{*\}\times\mathbb{P}^1$. Then $N_1(Z)$ is generated by $u_1,u_2,u_3$.
	We have $\Eff(Z)=\mathbb{R}_{\ge 0}[U_1]+\mathbb{R}_{\ge 0}[U_2]+\mathbb{R}_{\ge 0}[U_3]$.
	
	It can be seen that $\NE(Z)$ is generated by $u_1,u_2,u_3$ and $\Nef(Z)=\mathbb{R}_{\ge 0}[U_1]+\mathbb{R}_{\ge 0}[U_1+U_2]+\mathbb{R}_{\ge 0}[U_3]$. 
	Moreover, we have $-K_Z= p_{1}^{*}(-K_{\mathbb{F}_1})+p_2^{*}(-K_{\mathbb{P}^1})\equiv 3U_1+2U_2+2U_3$.
	This is enough to apply Proposition \ref{ConstructionA}, which yields the choices for $A$ and $D$ (denoted by their coordinates with respect to the the base of $N^1(Z)$ given by $U_1$, $U_2$, $U_3$) listed in Table \ref{3-28}.
    \newpage
	\begin{small}
		\begin{center}
			\begin{longtable}{|c|c||c|c|c|c|c|c|c|c|c|}
				\hline
				$A$ & $D$ & \#&$c_1^4$ & $c_1^2c_2$ & $h^0(-K_X)$ &$h^{1,2}$&$h^{2,2}$&$h^{1,3}$&Toric?&Product?\\
				\hline\hline
				\multirow{3}{*}{(0,0,1)} & (-1,-1,0)/(1,1,1) &1& 373 & 190 &79&0&8&0&$Q_{5}$&No\\
				\hhline{|~|-|-|-|-|-|-|-|-|-|-|} 
				& (0,0,0)/(0,0,1) &2& 336 & 180 &72&0&8&0&$Q_{10}$&$\mathbb{F}_1\times\text{Bl}_2\mathbb{P}^2$\\
				\hhline{|~|-|-|-|-|-|-|-|-|-|-|}
				& (-1, -1, 1)/(1,1,0) &3& 325 & 178 &70&0&8&0&$Q_{14}$&No\\
				\hline\hline
				\multirow{6}{*}{(0,1,0)} & (-1,0,-1)/(1,1,1) &4& 405 & 198 &85&0&8&0&$Q_{4}$&No \\
				\hhline{|~|-|-|-|-|-|-|-|-|-|-|}
				& (-1,0,0)/(1,1,0) &5& 368 & 188 &78&0&8&0&$Q_{6}$&(\#4-12)$\times\mathbb{P}^1$\\
				\hhline{|~|-|-|-|-|-|-|-|-|-|-|}
				& (0,1,1)/(0,0,-1) &6& 363 & 186 &77&0&8&0&$Q_{7}$&No\\
				\hhline{|~|-|-|-|-|-|-|-|-|-|-|}
				& (0,1,0)/(0,0,0) &7& 352 & 184 &75&0&8&0&$Q_{8}$&(\#4-11)$\times\mathbb{P}^1$\\
				\hhline{|~|-|-|-|-|-|-|-|-|-|-|}
				& (0,1,-1)/(0,0,1) &8& 341 & 182 &73&0&8&0&$Q_{9}$&No\\
				\hhline{|~|-|-|-|-|-|-|-|-|-|-|}
				& (-1,0,1)/(1,1,-1) &9& 331 & 178 &71&0&8&0&$Q_{12}$&No\\
				\hline\hline 
				\multirow{3}{*}{(1,1,0)} & (0,0,-1)/(1,1,1) &10& 341 & 182 &73&0&8&0&$Q_{9}$&No\\
				\hhline{|~|-|-|-|-|-|-|-|-|-|-|}
				& (0,0,0)/(1,1,0) &11& 320 & 176 &69&0&8&0&$Q_{15}$&(\#4-9)$\times\mathbb{P}^1$\\
				\hhline{|~|-|-|-|-|-|-|-|-|-|-|}
				& (0,0,1)/(1,1,-1) &12& 299 & 170 &65&0&8&0&$Q_{17}$&No\\
				\hline\hline
				(1,1,1) & (0,0,0)/(1,1,1) &13& 293 & 170 &64&0&9&0&No&No\\
				\hline\hline
				(1,1,2) & (0,0,1)/(1,1,1) &14& 245 & 158 &55&0&10&0&No&No\\
				\hline\hline
				\multirow{2}{*}{(2,2,0)} & (1,1,-1)/(1,1,1) &15& 256 & 160 &57&0&8&0&No&No\\
				\hhline{|~|-|-|-|-|-|-|-|-|-|-|}
				& (1,1,0) &16& 256 & 160 &57& 0&8&0&No&(\#4-4)$\times\mathbb{P}^1$\\
				\hline\hline 
				(2,2,1) & (1,1,0)/(1,1,1) &17&  229 & 154 &52&0&12&0&No&No\\
				\hline\hline 
				(2,2,2) & (1,1,1) &18& 202 & 148 &47&0&16&0&No&No\\
				\hline
				\caption{Choices of $A$, $D$ in $Z=\mathbb{F}_1\times\mathbb{P}^1$.}
				\label{3-28}
			\end{longtable}
		\end{center}
	\end{small}
	\paragraph{6. \fbox{$Z=\mathbb{P}_{\mathbb{F}_1}(\mathcal{O}\oplus \mathcal{O}(-\mathbf{l}))$ (\#3-30).}}
	Call $\pi\colon Z\to \mathbb{F}_1$ the projection. Consider $H:=\pi^{*}\mathbf{l}$, $\tilde{E}:=\pi^{*}\mathbf{e}$ and $E$ the negative section of $\pi$. We denote by $\mathbf{f}$ the class in $N_1(Z)$ of the generic fiber of $\pi$, $\mathbf{l}$ the class of $\pi^{-1}(\mathbf{l})\cap E$ and $\mathbf{e}$ the class of $\pi^{-1}(\mathbf{e})\cap E$. We have that $N^1(Z)$ is generated by the classes of $H$, $\tilde{E}$ and $E$. Moreover NE$(Z)$ is generated by $\mathbf{f}$, $\mathbf{l}-\mathbf{e}$ and $\mathbf{e}$.
	The following intersection table holds:
	\begin{center}
		\begin{tabular}{|c|c|c|c|}
			\hline
			$\cdot$ & $\mathbf{f}$ & $\mathbf{l}$ & $\mathbf{e}$ \\
			\hline
			$H$ & 0 & 1 & 0\\
			\hline 
			$\tilde{E}$ & 0 & 0 & -1\\
			\hline
			$E$ & 1 & -1 & 0\\
			\hline 
		\end{tabular}
	\end{center}  
	Therefore we have \sloppy  ${\Nef(Z)=\mathbb{R}_{\ge 0}[H]+\mathbb{R}_{\ge 0}[H-\tilde{E}]+\mathbb{R}_{\ge 0}[H+E]}$. Moreover, the effective cone is \sloppy ${\Eff(Z)=\mathbb{R}_{\ge 0}[E]+\mathbb{R}_{\ge 0}[H-\tilde{E}]+\mathbb{R}_{\ge 0}[\tilde{E}]}$. 
	
	Recall that ${-K_Z= 2c_1(\mathcal{O}_Z(1))+\pi^{*}(-K_{\mathbb{F}_1}-\det(\mathcal{O}\oplus\mathcal{O}(-\mathbf{l})))\equiv 4H-\tilde{E}+2E}$. This is enough to apply Proposition \ref{ConstructionA} and obtain the list of admissible choices of $A$ and $D$ in $Z$ contained in Table \ref{3-30}. 
    \newpage
	\begin{small}
		\begin{center}
			\begin{longtable}{|c|c||c|c|c|c|c|c|c|c|c|}
				\hline
				$A$ & $D$ & \#& $c_1^4$ & $c_1^2c_2$ & $h^0(-K_X)$ & $h^{1,2}$& $h^{2,2}$&$h^{1,3}$&Toric?&Product?\\
				\hline\hline
				\multirow{2}{*}{$E$} 
				&$-H$/$H+E$ &1& 405 & 198 &85&
				0&8&0&$Q_2$&No\\
				\hhline{|~|-|-|-|-|-|-|-|-|-|-|}
				& $E$/$0$ &2& 373 & 190 & 79 &0 &8&0&$Q_{5}$&No\\
				\hline\hline
				$H+E$ & $H+E$/$0$ &3& 325 & 178 &70&0&8&0&$Q_{14}$&No \\
				\hline\hline 
				$2H+2E$
				& $H+E$ &4& 250 & 160 &56&0&10&0&No&No\\
				\hline
				\caption{Choices of $A$, $D$ in Z=$\mathbb{P}_{\mathbb{F}_1}(\mathcal{O}\oplus\mathcal{O}(-l))$.}
				\label{3-30}
			\end{longtable}
		\end{center}
	\end{small}
	\paragraph{7. \fbox{$Z=\mathbb{P}_{\mathbb{P}^1\times\mathbb{P}^1}(\mathcal{O}\oplus \mathcal{O}(-1,-1))$ (\#3-31).}}
	Call $\pi\colon  Z\to \mathbb{P}^1\times\mathbb{P}^1$ the projection, $H_1:= \pi^{-1}(\{*\}\times\mathbb{P}^1)$, $H_2:= \pi^{-1}(\mathbb{P}^1\times\{*\})$ and $E$ the section corresponding to the quotient $\mathcal{O}\oplus\mathcal{O}(-1,-1)\twoheadrightarrow \mathcal{O}(-1,-1)$. Observe that $\mathcal{O}_Z(E)=\mathcal{O}_Z(1)$, and $H_1$, $H_2$, $E$ generate $N^1(Z)$. 
	
	In $N_1(Z)$, let $\mathbf{f}$ be the class of a one dimensional fiber of $\pi$, and let $\mathbf{e_1}$ and $\mathbf{e_2}$ be the copies in $E$ of the two rulings in $\mathbb{P}^1\times\mathbb{P}^1$. They generate $N_1(Z)$. 
	
	We have
	$\NE(Z)=\mathbb{R}_{\ge 0}\mathbf{f}+\mathbb{R}_{\ge 0}\mathbf{e_1}+\mathbb{R}_{\ge 0}\mathbf{e_2}$, \sloppy${\Nef(Z)=\mathbb{R}_{\ge 0}[H_1]+\mathbb{R}_{\ge 0}[H_2]+\mathbb{R}_{\ge 0}[H_1+H_2+E]}$. 
	Moreover the effective cone is $\Eff(Z)=\mathbb{R}_{\ge 0}[H_1]+\mathbb{R}_{\ge 0}[H_2]+\mathbb{R}_{\ge 0}[E]$. 
	
	Recall that 
	\[-K_Z= 2\mathcal{O}_Z(1)+\pi^{*}(-K_{\mathbb{P}^1\times\mathbb{P}^1}-\det(\mathcal{O}\oplus\mathcal{O}(-1,-1)))\equiv 3H_1+3H_2+2E.
	\]This is enough to be able to apply Proposition \ref{ConstructionA}, from which it follows that the only (up to symmetries) admissible choices of $A$ and $D$ in $Z$ are the ones listed in Table \ref{3-31}. None of the Fano 4-folds obtained this way is isomorphic to a product of lower dimensional varieties. 
    \newline
	\begin{small}
		\centering
		\begin{longtable}{|>{\centering\arraybackslash}m{2.5 cm}|>{\centering\arraybackslash}m{2.5 cm}||c|c|c|c|c|c|c|>{\centering\arraybackslash}m{0.9 cm}|}
			\hline
			$A$ & \centering $D$ &\# & $c_1^4$ & $c_1^2c_2$ & $h^0(-K_X)$ & $h^{1,2}$& $h^{2,2}$&$h^{1,3}$& Toric?\\
			\hline\hline
			\multirow{3}{*}{$E$} & $-H_1-H_2$/
			$H_1+H_2+E$ & 1 & 442 & 208 & 92 &0&8&0&$Q_1$\\
			\hhline{|~|-|-|-|-|-|-|-|-|-|}
			& $-H_1$/${H_1+E}$ &2& 405 & 198 &85&0&8&0&$Q_{4}$\\
			\hhline{|~|-|-|-|-|-|-|-|-|-|}
			& $0$/$E$ &3& 394 & 196 &83&0&8&0&$Q_{3}$\\
			\hline\hline
			$H_1+H_2+E$ & $0$/${H_1+H_2+E}$ &4& 330 & 180 &71&0&8&0&$Q_{13}$ \\
			\hline\hline 
			$2H_1+2H_2+2E$
			& $H_1+H_2+E$ &5& 244 & 160 &55&0&12&0&No\\
			\hline
			\caption{Choices of $A$, $D$ in Z=$\mathbb{P}_{\mathbb{P}^1\times\mathbb{P}^1}(\mathcal{O}\oplus\mathcal{O}(-1,-1))$}
			\label{3-31}
		\end{longtable}
	\end{small}
	
	\paragraph{Ambiguities}
	In Tables \ref{3-17}-\ref{3-31} there are some families of Fano 4-folds arising from different constructions of type A that have the same invariants. One can check that all non-toric Fano 4-folds listed in said tables are distinct. 
	
	For example, there are pairs of families with identical numerical invariants, where one family consists of varieties that admit a clear product structure  $W\times\mathbb{P}^1$, with $W$ a Fano 3-fold, while the other family does not. This is the case for the following couples: 
	\begin{center}
		(\#15 Table \ref{3-28}, \#16 Table \ref{3-28}) ; (\#10 Table \ref{3-27}, \#11 Table \ref{3-27}) ; (\#16 Table \ref{3-27}, \#17 Table \ref{3-27}).
	\end{center}
	Families \#16 in Table \ref{3-28}, \#10 in Table \ref{3-27} and \#17 in Table \ref{3-27} are clearly products, as it emerges from the data of Construction A that yield them.  
	If $X$ is a member of one of the remaining families, it is possible to see that in all their constructions, $Y$ and $A_Y$  are not in the right form for $X$ to be a product. Thus in all three pairs listed above, each family is distinct. 
	
	The last non-toric ambiguous case, is that of the couple (\#22 Table \ref{3-27} , \#18 Table \ref{3-28}), and one can check though geometric arguments that these families are distinct. 
	
	Finally, the study of ambiguities among toric 4-folds can be done by comparing their descriptions with the ones in \cite{Bat99}. 
\subsection{Case (B.2)}\label{Case2}
	In this section, we are going to list every Fano 4-fold $V$ obtained through the iteration of two constructions of type A, starting from $\mathbb{P}^2$ (case (B.2) of Theorem \ref{theorem4folds}). Observe that by the proof of Theorem \ref{theorem4folds}, we may just look for all choices of $A$ and $D$ in \#3-14, \#3-19, \#3-26 that satisfy the conditions in Proposition \ref{ConstructionA}, similarly to what we did for case (B.1). However it is much more straightforward to use the iterative Construction A with base $\mathbb{P}^2$, since all Fano 4-folds we are looking for arise from it.  We refer to Section \ref{Recursion} for the notation used in this section. By Lemma \ref{A',D'}, recall that $V$ is completely determined by the choice of four Cartier divisors in $Z:=\mathbb{P}^2$: $A$, $A'$ ( two smooth irreducible surfaces in $Z$) and $D$, $D'$. Now, let $a,d,a',d'$ be integers such that 
	$A=\mathcal{O}_{\mathbb{P}^2}(a)$, $A'=\mathcal{O}_{\mathbb{P}^2}(a')$
	$D=\mathcal{O}_{\mathbb{P}^2}(d)$
	$D'=\mathcal{O}_{\mathbb{P}^2}(d')$.
	Observe that, being $A$ and $A'$ effective, we need $a, a'\ge 1$. 
	In this setting, the conditions for $V$ to be Fano listed in Proposition \ref{DoubleProp} read 
	\begin{align}\label{ConditionsP^2}
		\begin{cases}
			3-(a'-d')-(a-d)\ge 1\\
			3-(a'-d')-d\ge 1\\
			3-d'-(a-d)\ge 1\\
			3-d'-d\ge 1\\
			3-(a'-d')\ge 1\\
			3-d'\ge 1
		\end{cases}
	\end{align} 
	Up to exchanging $d$ with $a-d$ or $d'$ with $a'-d'$, the only values of $(a,d,a',d')$ that satisfy (\ref{ConditionsP^2}) are the following: $\{(1,1,1,1),(1,1,2,1),(2,1,1,1),(2,1,2,1)\}$. Observe that   the choices of $(a,d,a',d')=(1,1,2,1)$ and $(a,d,a',d')=(2,1,1,1)$ yield the same family of Fano 4-folds by Proposition \ref{SwitchProposition}. We thus obtain three distinct families, listed in Table \ref{DoubleTable}.

	\begin{table}[H]\small
		\begin{center}
			\begin{tabular}{|c|c||c|c||c|c|c|c|c|c|c|}
				\hline
				$\mathcal{O}_Z(A)$ & $\mathcal{O}_Z(D)$ & $\mathcal{O}_Z(A')$ & $\mathcal{O}_Z(D')$ &\#& $c_1^4$ & $c_1^2c_2$ & $h^0(-K_X)$ & $h^{2,2}$ & Toric?&Product?\\
				
				\hline\hline
				\multirow{2}{*}{$\mathcal{O}_{\mathbb{P}^2}(1)$} & \multirow{2}{*}{$\mathcal{O}_{\mathbb{P}^2}(1)/\mathcal{O}_{\mathbb{P}^2}$}& $\mathcal{O}_{\mathbb{P}^2}(1)$ &$\mathcal{O}_{\mathbb{P}^2}(1)/\mathcal{O}_{\mathbb{P}^2}$ &1& 310 & 172 & 67 & 9 & Yes\tablefootnote{number 108 in \cite{Bat99}, \S4.}&No\\
				\hhline{|~|~|-|-|-|-|-|-|-|-|-|}
				& &  $\mathcal{O}_{\mathbb{P}^2}(2)$ & $\mathcal{O}_{\mathbb{P}^2}(1)$ &2& 252 & 156 & 56 & 10 &No&No\\
				\hline 
				$\mathcal{O}_{\mathbb{P}^2}(2)$ & $\mathcal{O}_{\mathbb{P}^2}(1)$&   $\mathcal{O}_{\mathbb{P}^2}(2)$ & $\mathcal{O}_{\mathbb{P}^2}(1)$ &3&200 & 140 & 46 & 12 & No&No\\
				\hline 
			\end{tabular}
			\caption{Choices of $A$, $D$, $A'$, $D'$ in $\mathbb{P}^2$ that yield a Fano variety through a composition of two constructions of type A. All resulting Fano 4-folds have $h^{1,2}$=$h^{1,3}=0$.}
			\label{DoubleTable}
		\end{center}
	\end{table} 
	
\section{Fano 4-folds arising from Construction A with $\rho=4$ and $\delta=2$}\label{Standard4foldsRho4}
	According to Theorem \ref{theorem4folds}, there are four possibilities for a Fano 4-fold $X$ arising from Construction A with $\rho_X=4$. 
	\begin{itemize}
		\item[(A.1)] There exists a $\mathbb{P}^1$-bundle $Z\to \mathbb{P}^2$.
		\item[(A.2)] There exists a fibration $Z\to \mathbb{P}^1$ whose generic fiber is isomorphic to $\mathbb{P}^1\times\mathbb{P}^1$ and whose singular fibers are isomorphic to a quadric cone $\mathcal{Q}\subset\mathbb{P}^3$. 
		\item[(A.3)] There exists a $\mathbb{P}^2$-bundle $Z\to \mathbb{P}^1$. 
		\item[(A.4)] $Z\simeq \text{Bl}_{p}\mathcal{Q}_3$, where $\mathcal{Q}_3$ is a smooth quadric threefold in $\mathbb{P}^4$ and $p$ is a point in $\mathcal{Q}_3$. 
	\end{itemize}
	By studying the classification of Fano 3-folds with $\rho=2$ by Mori and Mukai and the work of Szurek-Wi\'sniewski \cite{SW90}, we get that the families satisfying conditions (A.1)-(A.4) are the following:
	\begin{itemize}
		\item (A.1) \#2-24, \#2-27, \#2-31, \#2-32, \#2-34, \#2-35, \#2.36.
		\item (A.2) \#2-18, \#2-25, \#2-29.
		\item (A.3) \#2-33, \#2-34.
		\item (A.4) \#2-30.  
	\end{itemize}
	In the following paragraphs, we will find all choices of $A$ and $D$ in Fano 3-folds $Z$ belonging to the families above that yield a Fano 4-fold through Construction A.
	
	For every $Z$, we call $\gamma_1\colon Z\to W_1$ and $\gamma_2\colon Z\to W_2$ its elementary contractions, and we call $H_i$ the pull back through $\gamma_i$ of the ample generator of $\Pic(W_i)$ for $i=1,2$, so that $\Nef(Z)=\mathbb{R}_{\ge 0}H_1+\mathbb{R}_{\ge 0}H_2$. We know that $\{H_1,H_2\}$ is also a $\mathbb{Z}$-basis for $\Pic(Z)$, and $-K_Z\equiv \mu_2H_1+\mu_1H_2$, where $\mu_i$ is the length of the extremal ray associated to $\gamma_i$ (see \cite{MM81}). We also know that smooth irreducible hypersurfaces in $Z$ must be equivalent to a linear combination of $H_1$ and $H_2$ with non-negative integer coefficients, or to one of the exceptional divisors in $Z$, if there is any. This is enough to apply Proposition \ref{ConstructionA} and obtain Tables \ref{2-18}-\ref{2-36}.
	
	Observe that every Fano 4-fold $X$ constructed this way has $\delta_X=2$. Indeed, $\rho_Z=2$ implies that $\delta_Z\le 1$, and then by Lemma \ref{deltaLemma} we get $\delta_X\in\{2,3\}$. However, there cannot be Fano 4-folds with $\rho_X=4$ and $\delta_X=3$ (see \cite{CRS22}), so $\delta_X=2$.  
	
	We ultimately obtain the following theorem. 
	\begin{theorem}\label{AllRho4}
		There are 94 distinct families of Fano 4-folds $X$ with $\rho_X=4$ arising from Construction A, and they are all listed in Table \ref{FinalTableRho4}. All of them have $\delta_X=2$. Among them, 66 are not toric nor products of lower dimensional varieties.    
	\end{theorem}	   
	
	As a side comment, it is worth noticing that the Fano polytopes of the toric Fano 4-folds arising from Construction A  with $\rho=4$ are all of type $H_n$ or $I_n$, for some $n\in\mathbb{N}$ (see \cite{Bat99}). 
	
	\paragraph{1. \fbox{$Z=$ Double cover of $\mathbb{P}^1\times\mathbb{P}^2$ with branch locus a divisor of degree $(2,2)$ (\#2-18).}} 
	The elementary contractions of $Z$ are $\gamma_1\colon Z\to \mathbb{P}^1$ and $\gamma_2\colon Z\to \mathbb{P}^2$, where $\gamma_1$ is a quadric bundle, and $\gamma_2$ is a conic bundle. 
	\begin{small}
		\begin{center}
			\begin{longtable}{|c|c||c|c|c|c|c|c|c|c|c|}
				\hline
				$A$ & $D$ &\#& $c_1^4$ & $c_1^2c_2$ & $h^0(-K_X)$ & $h^{1,2}$& $h^{2,2}$&$h^{1,3}$&Toric?&Product?\\
				\hline\hline
				$H_2$&$H_2$/$0$& 1&152&128&37&2&10&0&No&No\\
				\hline\hline
				$2H_2$&$H_2$&2&112&112&29&2&14&0&No&No\\
				\hline
				\caption{Choices of $A$, $D$ in $Z=$ double cover of $\mathbb{P}^1\times\mathbb{P}^2$ with branch locus a $(2,2)$-divisor.}
				\label{2-18}
			\end{longtable}
		\end{center}
	\end{small}
	\paragraph{2. \fbox{$Z=(1,2)$-divisor in $\mathbb{P}^2\times\mathbb{P}^2$ (\# 2-24).}}
	The elementary contractions of $Z$ are $\gamma_1\colon Z\to \mathbb{P}^2$ and $\gamma_2\colon Z\to \mathbb{P}^2$, where $\gamma_1$ is a conic bundle, and $\gamma_2$ is a $\mathbb{P}^1$-bundle.
	\begin{small}
		\begin{center}
			\begin{longtable}{|c|c||c|c|c|c|c|c|c|c|c|}
				\hline
				$A$ & $D$ &\#& $c_1^4$ & $c_1^2c_2$ & $h^0(-K_X)$ & $h^{1,2}$& $h^{2,2}$&$h^{1,3}$&Toric?&Product?\\
				\hline\hline
				$H_1$&$H_1$/$0$&1& 194 &140 &45&0&9&0&No&No\\
				\hline\hline
				$2H_1$&$H_1$&2&148&124&36&0&12&0&No&No\\
				\hline
				\caption{Choices of $A$, $D$ in $Z=(1,2)$-divisor in $\mathbb{P}^2\times\mathbb{P}^2$.}
				\label{2-24}
			\end{longtable}
		\end{center}
	\end{small}
	
	\paragraph{3. \fbox{$Z=(1,2)$-divisor in $\mathbb{P}^1\times\mathbb{P}^3$ (\#2-25).}}
	The elementary contractions of $Z$ are $\gamma_1\colon Z\to \mathbb{P}^1$ and $\gamma_2\colon Z\to \mathbb{P}^3$, where $\gamma_1$ is a quadric bundle, $\gamma_2$ is the blow up of the intersection of two quadrics in $\mathbb{P}^3$. 
	\begin{small}
		\begin{center}
			\begin{longtable}{|c|c||c|c|c|c|c|c|c|c|c|}
				\hline
				$A$ & $D$ &\#& $c_1^4$ & $c_1^2c_2$ & $h^0(-K_X)$ & $h^{1,2}$& $h^{2,2}$&$h^{1,3}$&Toric?&Product?\\
				\hline\hline
				\multirow{1}{*}{$-H_1+2H_2$}
				&$-H_1+H_2$/$H_2$&1&164&128&39&2&6&0&No&No\\
				\hline\hline
				$H_2$&$H_2$/$0$& 2&199&142&46&1&9&0&No&No\\
				\hline\hline
				$2H_2$&$H_2$&3&142&124&35&1&14&0&No&No\\
				\hline
				\caption{Choices of $A$, $D$ in $Z=(1,2)$ divisor in $\mathbb{P}^1\times\mathbb{P}^3$.}
				\label{2-25}
			\end{longtable}
		\end{center}
	\end{small} 
	\paragraph{4. \fbox{$Z=$ Blow up of a twisted cubic in $\mathbb{P}^3$ (\#2-27).}} 
	The elementary contractions of $Z$ are $\gamma_1\colon Z\to \mathbb{P}^2$ and $\gamma_2\colon Z\to \mathbb{P}^3$, where $\gamma_1$ is a $\mathbb{P}^1$-bundle and $\gamma_2$ is the blow up of a twisted cubic in $\mathbb{P}^3$. 
    \newpage
	\begin{small}
		\begin{center}
			\begin{longtable}{|c|c||c|c|c|c|c|c|c|c|c|}
				\hline
				$A$ & $D$ &\#& $c_1^4$ & $c_1^2c_2$ & $h^0(-K_X)$ & $h^{1,2}$& $h^{2,2}$&$h^{1,3}$&Toric?&Product?\\
				\hline\hline
				$H_2$
				&$H_2$/$0$ &1& 241 & 154 &54 &0&8&0&No&No\\
				\hline\hline
				$-H_1+2H_2$ & $H_2$/$-H_1+H_2$ &2& 221 & 146 &50&0&6&0&No&No \\
				\hline\hline
				$2H_2$& $H_2$ &3& 178 & 136 & 42 &0&12&0&No&No\\
				\hline
				\caption{Choices of $A$, $D$ in $Z=$ Blow up of $\mathbb{P}^3$ along a twisted cubic.}
				\label{2-27}
			\end{longtable}
		\end{center}
	\end{small}
	\paragraph{5. \fbox{$Z=$ Blow up of a conic in a quadric in $\mathbb{P}^4$ (\#2-29).}} 
	The elementary contractions of $Z$ are $\gamma_1\colon Z\to \mathbb{P}^1$ and $\gamma_2\colon Z\to \mathcal{Q}\subset\mathbb{P}^4$, where $\gamma_1$ is a quadric bundle over $\mathbb{P}^1$ and $\gamma_2$ is the blow up of a conic in a smooth quadric $\mathcal{Q}\subset \mathbb{P}^4$. 
	\begin{small}
		\begin{center}
			\begin{longtable}{|c|c||c|c|c|c|c|c|c|c|c|}
				\hline
				$A$ & $D$ &\#& $c_1^4$ & $c_1^2c_2$ & $h^0(-K_X)$ & $h^{1,2}$& $h^{2,2}$&$h^{1,3}$&Toric?&Product?\\
				\hline\hline
				\multirow{2}{*}{$-H_1+H_2$}
				&$-H_1$/$H_2$&1&294&168&64&0&6&0&No&No\\
				\hhline{|~|-|-|-|-|-|-|-|-|-|-|}
				&$-H_1+H_2$/$0$&2&268&160&59&0&6&0&No&No\\
				\hline\hline
				$H_2$
				&$H_2$/$0$ &3& 246 & 156 &55 &0&8&0&No&No\\
				\hline\hline
				$2H_2$& $H_2$ &4& 172 & 136 & 41 &0&14&0&No&No\\
				\hline
				\caption{Choices of $A$, $D$ in $Z=$ Blow up of a quadric in $\mathbb{P}^3$ along a conic.}
				\label{2-29}
			\end{longtable}
		\end{center}
	\end{small}
	\paragraph{6. \fbox{$Z=$ Blow up of a conic in $\mathbb{P}^3$ (\#2-30).}} 
	The elementary contractions of $Z$ are $\gamma_1\colon Z\to \mathbb{P}^3$ and $\gamma_2\colon Z\to \mathcal{Q}\subset\mathbb{P}^4$, where $\gamma_1$ is a the blow up of a plane conic in $\mathbb{P}^3$, and $\gamma_2$ is the blow up of a point in a smooth quadric $\mathcal{Q}\subset\mathbb{P}^4$.
	\begin{small}
		\begin{center}
			\begin{longtable}{|c|c||c|c|c|c|c|c|c|c|c|}
				\hline
				$A$ & $D$ &\#& $c_1^4$ & $c_1^2c_2$ & $h^0(-K_X)$ & $h^{1,2}$& $h^{2,2}$&$h^{1,3}$&Toric?&Product?\\
				\hline\hline
				$2H_1-H_2$/$0$&$H_1$/$H_1-H_2$&1&316&172&68&0&6&0&No&No\\
				\hline\hline
				$H_1$
				&$H_1$/$0$ &2& 299 & 170 &65 &0&7&0&No&No\\
				\hline\hline
				$2H_1$& $H_1$ &3& 230 & 152 & 52 &0&10&0&No&No\\
				\hline
				\caption{Choices of $A$, $D$ in $Z=$ Blow up of $\mathbb{P}^3$ along a conic.}
				\label{2-30}
			\end{longtable}
		\end{center}
	\end{small}
	\paragraph{7. \fbox{$Z=$ Blow up of a quadric in $\mathbb{P}^4$ along a line (\# 2-31).}}
	The elementary contractions of $Z$ are $\gamma_1\colon Z\to \mathbb{P}^2$ and $\gamma_2\colon Z\to \mathcal{Q}$, where $\gamma_1$ is a $\mathbb{P}^1$-bundle over $\mathbb{P}^2$ and $\gamma_2$ is the blow up of a smooth quadric $\mathcal{Q}$ in $\mathbb{P}^4$ along a line.
    \newline
    \newline
	\begin{small}
		\begin{center}
			\begin{longtable}{|c|c||c|c|c|c|c|c|c|c|c|}
				\hline
				$A$ & $D$ &\#& $c_1^4$ & $c_1^2c_2$ & $h^0(-K_X)$ & $h^{1,2}$& $h^{2,2}$&$h^{1,3}$&Toric?&Product?\\
				\hline\hline
				\multirow{2}{*}{$-H_1+H_2$}
				&$H_2$/$-H_1$ &1& 368 & 188 &78 &0&6&0&No&No\\
				\hhline{|~|-|-|-|-|-|-|-|-|-|-|}
				& $-H_1+H_2$/$0$ &2& 331 & 178 &71&0&6&0&No&No \\
				\hline\hline
				$H_2$& $H_2$/$0$ &3& 288 & 168 & 63 &0&7&0&No&No\\
				\hline\hline
				$2H_2$&$H_2$&4& 208 &148 &48&0&12&0&No&No\\
				\hline
				\caption{Choices of $A$, $D$ in $Z=$ Blow up of a quadric in $\mathbb{P}^4$ along a line.}
				\label{2-31}
			\end{longtable}
		\end{center}
	\end{small}
	\paragraph{8. \fbox{$Z=(1,1)$-divisor in $\mathbb{P}^2\times\mathbb{P}^2$ (\# 2-32).}}
	The elementary contractions $\gamma_1\colon Z\to \mathbb{P}^2$ and $\gamma_2\colon Z\to \mathbb{P}^2$ of $Z$ are the restrictions of the projections of $\mathbb{P}^2\times\mathbb{P}^2$ onto its factors. Both realize $Z$ as $\mathbb{P}(T_{\mathbb{P}^2})$. 
	\begin{small}
		\begin{center}
			\begin{longtable}{|c|c||c|c|c|c|c|c|c|}
				\hline
				$A$ & $D$ &\#& $c_1^4$ & $c_1^2c_2$ & $h^0(-K_X)$ & ($h^{1,2}$, $h^{2,2}$,$h^{1,3})$&Toric?&Product?\\
				\hline\hline
				\multirow{3}{*}{$H_2$}
				&$H_1+H_2$/$-H_1$ &1& 362 & 188 &77 &(0,6,0)&No&No\\
				\hhline{|~|-|-|-|-|-|-|-|-|}
				&$H_2$/$0$ &2& 320 & 176 &69 &(0,6,0)&No&No\\
				\hhline{|~|-|-|-|-|-|-|-|-|}
				& $-H_1+H_2$/$H_1$ &3& 304 & 172 &66&(0,6,0)&No&No \\
				\hline\hline
				\multirow{2}{*}{$2H_2$}&$H_1+H_2$/$-H_1+H_2$&4&266&164&59&(0,6,0)&No&No\\
				\hhline{|~|-|-|-|-|-|-|-|-|}
				&$H_2$&5&256&160&57&(0,6,0)&No&No\\
				\hline\hline			
				\multirow{2}{*}{$H_1+H_2$}& $H_1+H_2$/$0$ &6& 282 & 168 & 62 &(0,8,0)&No&No\\
				\hhline{|~|-|-|-|-|-|-|-|-|}
				&$H_2$/$H_1$ &7& 256&160&57&(0,8,0)&No&No\\
				\hline\hline
				$H_1+2H_2$&$H_2$/$H_1+H_2$&8& 218 &152 &50&(0,12,0)&No&No\\
				\hline\hline
				$2H_1+2H_2$&$H_1+H_2$&9&180&144&43&(0,24,1)&No&No\\
				\hline
				\caption{Choices of $A$, $D$ in $Z=(1,1)$-divisor in $\mathbb{P}^2\times\mathbb{P}^2$.}
				\label{2-32}
			\end{longtable}
		\end{center}
	\end{small}
	\paragraph{9. \fbox{$Z=$ Blow up of a line in $\mathbb{P}^3$ (\#2-33).}} 
	The elementary contractions of $Z$ are $\gamma_1\colon Z\to \mathbb{P}^1$ and $\gamma_2\colon Z\to \mathbb{P}^3$, where $\gamma_1$ is a $\mathbb{P}^2$-bundle that realizes $Z$ as $\mathbb{P}_{\mathbb{P}^1}(\mathcal{O}\oplus\mathcal{O}\oplus\mathcal{O}(1))$, and $\gamma_2$ is the blow up of a line in $\mathbb{P}^3$.
	\begin{small}
		\begin{center}
			\begin{longtable}{|c|c||c|c|c|c|c|c|c|c|c|}
				\hline
				$A$ & $D$ &\#& $c_1^4$ & $c_1^2c_2$ & $h^0(-K_X)$ & $h^{1,2}$& $h^{2,2}$&$h^{1,3}$&Toric?&Product?\\
				\hline\hline
				\multirow{4}{*}{$-H_1+H_2$}
				&$-H_1-H_2$/$2H_2$&1&496&220&102&0&6&0&$I_{1}$&No\\
				\hhline{|~|-|-|-|-|-|-|-|-|-|-|}
				&$-H_2$/$-H_1+2H_2$&2&433&202&90&0&6&0&$I_{4}$&No\\
				\hhline{|~|-|-|-|-|-|-|-|-|-|-|}
				&$-H_1$/$H_2$&3&411&198&86&0&6&0&$I_{6}$&No\\
				\hhline{|~|-|-|-|-|-|-|-|-|-|-|}
				&$-H_1+H_2$/$0$&4&390&192&82&0&6&0&$I_{9}$&No\\
				\hline\hline
				\multirow{2}{*}{$H_2$}
				&$2H_2$/$-H_2$ &5& 415 & 202 &87 &0&6&0&$I_{5}$&No\\
				\hhline{|~|-|-|-|-|-|-|-|-|-|-|}
				&$H_2$/$0$&6& 357&186&76&0&6&0&$I_{14}$&No\\
				\hline\hline
				\multirow{2}{*}{$2H_2$}
				& $2H_2$/$0$ &7& 308 & 176 & 67 &0&8&0&No&No\\
				\hhline{|~|-|-|-|-|-|-|-|-|-|-|}
				&$H_2$&8&282&168&62&0&8&0&No&No\\
				\hline\hline
				$3H_2$&$H_2$/$2H_2$& 9&233&158&53&0&14&0&No&No\\
				\hline\hline
				$4H_2$&$2H_2$&10&184&148&44&0&28&1&No&No\\
				\hline
				\caption{Choices of $A$, $D$ in $Z=$ Blow up of a line in $\mathbb{P}^3$.}
				\label{2-33}
			\end{longtable}
		\end{center}
	\end{small}
	\paragraph{10. \fbox{$Z=\mathbb{P}^1\times\mathbb{P}^2$ (\# 2-34).}}
	The elementary contractions of $Z$ are the projections $\gamma_1\colon Z\to \mathbb{P}^1$ and $\gamma_2\colon Z\to \mathbb{P}^2$. 
	\begin{small}
		\begin{center}
			\begin{longtable}{|c|>{\centering\arraybackslash}m{3 cm}||c|c|c|c|c|c|c|}
				\hline
				$A$ & $D$ &\#& $c_1^4$ & $c_1^2c_2$ & $h^0(-K_X)$ & ($h^{1,2}$, $h^{2,2}$,$h^{1,3}$)&Toric?&Product?\\
				\hline\hline
				\multirow{5}{*}{$H_1$}& $-2H_2$/$H_1+2H_2$ &1& 478 & 220 & 99 &(0,5,0)&$H_{3}$&No\\
				\hhline{|~|-|-|-|-|-|-|-|-|}
				& $-H_2$/$H_1+H_2$ &2& 415 & 202 & 87 &(0,5,0)&$H_5$&No\\
				\hhline{|~|-|-|-|-|-|-|-|-|}
				& $2H_2$/$H_1-2H_2$ &3& 382 & 196 & 81&(0,5,0)&$H_7$&No\\
				\hhline{|~|-|-|-|-|-|-|-|-|}
				& $H_1$/$0$ &4& 378 & 192 & 80 &(0,5,0)&$H_{8}$&$\text{Bl}_2\mathbb{P}^2\times\mathbb{P}^2$\\
				\hhline{|~|-|-|-|-|-|-|-|-|}
				& $H_2$/$H_1-H_2$ &5& 367 & 190 & 78 &(0,5,0)&$H_{9}$&No\\
				\hline\hline
				\multirow{6}{*}{$H_2$}
				&$-H_1-H_2$/$H_1+2H_2$ &6& 463 & 214 &96 &(0,6,0)&$I_{2}$&No\\
				\hhline{|~|-|-|-|-|-|-|-|-|}
				& $-H_2$/$2H_2$ &7& 400 & 196 &84&(0,6,0)&$I_{7}$&$\mathbb{P}^1\times$(\#3-29) \\
				\hhline{|~|-|-|-|-|-|-|-|-|}
				& $-H_1$/$H_1+H_2$ &8& 389 & 194 &82&(0,6,0)&$I_{10}$&No \\
				\hhline{|~|-|-|-|-|-|-|-|-|}
				& $H_2$/$0$ &9& 368 & 188 &78&(0,6,0)&$I_{13}$&$\mathbb{P}^1\times$(\#3-26) \\
				\hhline{|~|-|-|-|-|-|-|-|-|}
				& $-H_1+H_2$/$H_1$ &10& 347 & 182 &74&(0,6,0)&$I_{12}$&No \\
				\hhline{|~|-|-|-|-|-|-|-|-|}
				& $-H_1+2H_2$/$H_1-H_2$ &11& 337 & 178 &72&(0,6,0)&$I_{15}$&No \\
				\hline\hline		
				\multirow{5}{*}{$2H_2$} 
				& $-H_1$/$H_1+2H_2$ &12&356 & 188 &76&(0,6,0)&No&No \\
				\hhline{|~|-|-|-|-|-|-|-|-|}
				& $2H_2$/$0$ &13&320 & 176 &69&(0,6,0)&No&$\mathbb{P}^1\times$(\#3-22) \\
				\hhline{|~|-|-|-|-|-|-|-|-|}
				& $H_2$ &14&304 & 172 &66&(0,6,0)&No&$\mathbb{P}^1\times$(\#3-19) \\
				\hhline{|~|-|-|-|-|-|-|-|-|}
				& $-H_1+H_2$/$H_1+H_2$ &15&304 & 172 &66&(0,6,0)&No&No \\
				\hhline{|~|-|-|-|-|-|-|-|-|}
				& $-H_1+2H_2$/$H_1$ &16&284 & 164 &62&(0,6,0)&No&No \\
				\hline\hline
				\multirow{3}{*}{$3H_2$}
				& $-H_1+H_2$/$H_1+2H_2$ &17&271 &166 &60&(1,6,0)&No&No \\
				\hhline{|~|-|-|-|-|-|-|-|-|}
				& $H_2$/$2H_2$ &18&256 & 160 &57&(1,6,0)&No&$\mathbb{P}^1\times$(\#3-14) \\
				\hhline{|~|-|-|-|-|-|-|-|-|}
				& $-H_1+2H_2$/$H_1+H_2$ &19&241 & 154 &54&(1,6,0)&No&No \\
				\hline\hline
				\multirow{2}{*}{$4H_2$}
				&$-H_1+2H_2$/$H_1+2H_2$&20&208&148&48&(3,6,0)&No&No\\
				\hhline{|~|-|-|-|-|-|-|-|-|}
				&$2H_2$&21&208&148&48&(3,6,0)&No&$\mathbb{P}^1\times$(\#3-9)\\
				\hline\hline
				\multirow{4}{*}{$H_1+H_2$}
				&$-H_2$/$H_1+2H_2$&22&382&196&81&(0,6,0)&No&No\\
				\hhline{|~|-|-|-|-|-|-|-|-|}
				&$H_1+H_2$/$0$&23&335&182&72&(0,6,0)&No&No\\
				\hhline{|~|-|-|-|-|-|-|-|-|}
				&$2H_2$/$H_1-H_2$&24&319&178&69&(0,6,0)&No&No\\
				\hhline{|~|-|-|-|-|-|-|-|-|}
				&$H_2$/$H_1$&25&314&176&68&(0,6,0)&No&No\\
				\hline\hline
				\multirow{3}{*}{$H_1+2H_2$}
				&$H_1+2H_2$/$0$&26&302&176&66&(0,9,0)&No&No\\
				\hhline{|~|-|-|-|-|-|-|-|-|}
				&$H_1$/$H_1+H_2$&27&271&166&60&(0,9,0)&No&No\\
				\hhline{|~|-|-|-|-|-|-|-|-|}
				&$2H_2$/$H_1$&28&266&164&59&(0,9,0)&No&No\\
				\hline\hline
				\multirow{2}{*}{$H_1+3H_2$}
				&$H_2$/$H_1+2H_2$& 29&238&160&54&(0,14,0)&No&No\\
				\hhline{|~|-|-|-|-|-|-|-|-|}
				&$2H_2$/$H_1+H_2$&30&223&154&51&(0,14,0)&No&No\\
				\hline\hline
				$H_1+4H_2$&$2H_2$/$H_1+2H_2$&31&190&148&45&(0,21,0)&No&No\\
				\hline\hline
				\multirow{2}{*}{$2H_1+H_2$}
				&$H_1-H_2$/$H_1+2H_2$&32&301&178&66&(0,6,0)&No&No\\
				\hhline{|~|-|-|-|-|-|-|-|-|}
				&$H_1$/$H_1+H_2$&33&281&170&62&(0,6,0)&No&No\\
				\hline\hline
				\multirow{2}{*}{$2H_1+2H_2$}
				&$H_1$/$H_1+2H_2$&34&248&164&56&(0,12,0)&No&No\\
				\hhline{|~|-|-|-|-|-|-|-|-|}
				&$H_1+H_2$&35&238&160&54&(0,12,0)&No&No\\
				\hline\hline
				$2H_1+3H_2$&$H_1+H_2$/$H_1+2H_2$ & 36&205&154&48&(0,24,1)&No&No\\
				\hline\hline
				$2H_1+4H_2$&$H_1+2H_2$&37&172&148&42&(0,42,3)&No&No\\
				\hline
				\caption{Choices of $A$, $D$ in $Z=\mathbb{P}^1\times\mathbb{P}^2$.}
				\label{2-34}
			\end{longtable}
		\end{center}
	\end{small}
	\paragraph{11. \fbox{$Z=\mathbb{P}_{\mathbb{P}^2}(\mathcal{O}\oplus\mathcal{O}(1))$ (\# 2-35).}} 
	The elementary contractions of $Z$ are $\gamma_1\colon Z\to \mathbb{P}^2$ and $\gamma_2\colon Z\to \mathbb{P}^3$, where $\gamma_1$ realizes $Z$ as $\mathbb{P}_{\mathbb{P}^2}(\mathcal{O}\oplus\mathcal{O}(1))$, while $\gamma_2$ is the blow up of a point in $\mathbb{P}^3$.
    \newline
    \newline
	\begin{small}
		
		\begin{center}
			\begin{longtable}{|c|c||c|c|c|c|c|c|c|}
				\hline
				$A$ & $D$ &\#& $c_1^4$ & $c_1^2c_2$ & $h^0(-K_X)$ &($h^{1,2}$,$h^{2,2}$,$h^{1,3}$)&Toric?&Product?\\
				\hline\hline
				\multirow{2}{*}{$-H_1+H_2$}
				&$-H_1$/$H_2$ &1& 447 & 210 &93 &(0,5,0)&$H_{4}$&No\\
				\hhline{|~|-|-|-|-|-|-|-|-|}
				& $-H_1+H_2$/$0$ &2& 415 & 202 &87&(0,5,0)&$H_{5}$&No \\
				\hline\hline
				\multirow{3}{*}{$H_1$}& $H_1+H_2$/$-H_2$ &3& 442 & 208 & 92 &(0,6,0)&$I_{3}$&No\\
				\hhline{|~|-|-|-|-|-|-|-|-|}
				&$H_1$/$0$ &4& 384&192&81&(0,6,0)&$I_{8}$&No\\
				\hhline{|~|-|-|-|-|-|-|-|-|}
				&$H_1-H_2$/$H_2$ &5& 384&192&81&(0,6,0)&$I_8$&No\\
				\hline\hline
				\multirow{2}{*}{$2H_1$}
				&$H_1-H_2$/$H_1+H_2$&6&346&184&74&(0,6,0)&No&No\\
				\hhline{|~|-|-|-|-|-|-|-|-|}
				&$H_1$&7&320&176&69&(0,6,0)&No&No\\
				\hline\hline
				\multirow{3}{*}{$H_2$}
				&$-H_1$/$H_1+H_2$&8&409&202&86&(0,5,0)&$H_{6}$&No\\
				\hhline{|~|-|-|-|-|-|-|-|-|}
				&$H_2$/$0$&9&367&190&78&(0,5,0)&$H_{9}$&No\\
				\hhline{|~|-|-|-|-|-|-|-|-|}
				&$-H_1+H_2$/$H_1$&10&351&186&75&(0,5,0)&$H_{10}$&No\\
				\hline\hline
				\multirow{2}{*}{$H_1+H_2$}
				&$H_1+H_2$/$0$&11&329&182&71&(0,7,0)&No&No\\
				\hhline{|~|-|-|-|-|-|-|-|-|}
				&$H_1$/$H_2$&12&303&174&66&(0,7,0)&No&No\\
				\hline\hline
				$2H_1+H_2$&$H_1+H_2$/$H_1$&13&265&166&59&(0,11,0)&No&No\\
				\hline\hline
				\multirow{2}{*}{$2H_2$}
				&$-H_1+H_2$/$H_1+H_2$&14&296&176&65&(0,6,0)&No&No\\
				\hhline{|~|-|-|-|-|-|-|-|-|}
				&$H_2$&15&286&172&63&(0,6,0)&No&No\\
				\hline\hline
				$H_1+2H_2$&$H_2$/$H_1+H_2$&16&248&164&56&(0,12,0)&No&No\\
				\hline\hline
				$2H_1+2H_2$ &$H_1+H_2$ &17& 210&156&49&(0,24,1)&No&No\\
				\hline
				\caption{Choices of $A$, $D$ in $Z=\mathbb{P}_{\mathbb{P}^2}(\mathcal{O}\oplus\mathcal{O}(1))$.}
				\label{2-35}
			\end{longtable}
		\end{center}
	\end{small}
	\paragraph{12. \fbox{$Z=\mathbb{P}_{\mathbb{P}^2}(\mathcal{O}\oplus\mathcal{O}(2))$ (\# 2-36).}} 
	The elementary contractions of $Z$ are $\gamma_1\colon Z\to \mathbb{P}^2$ and $\gamma_2\colon Z\to W$, where $\gamma_1$ realizes $Z$ as $\mathbb{P}_{\mathbb{P}^2}(\mathcal{O}\oplus\mathcal{O}(2))$, while $\gamma_2$ is the blow up of a quadruple non-Gorenstein point in a threefold $W$. 
	\begin{small}
		\begin{center}
			\begin{longtable}{|c|c||c|c|c|c|c|c|c|c|c|}
				\hline
				$A$ & $D$ &\#& $c_1^4$ & $c_1^2c_2$ & $h^0(-K_X)$ & $h^{1,2}$& $h^{2,2}$&$h^{1,3}$&Toric?&Product?\\
				\hline\hline
				\multirow{3}{*}{$-2H_1+H_2$}
				&$-2H_1$/$H_2$ &1& 558 & 240 &114 &0&5&0&$H_1$&No\\
				\hhline{|~|-|-|-|-|-|-|-|-|-|-|}
				& $-H_1+H_2$/$-H_1$ &2& 505 & 226 &104&0&5&0&$H_{2}$&No \\
				\hhline{|~|-|-|-|-|-|-|-|-|-|-|}
				&$-2H_1+H_2$/$0$&3&478&220&99&0&5&0&$H_{3}$&No\\
				\hline\hline
				$H_2$&$H_2$/$0$&4&382&196&81&0&5&0&$H_{7}$&No\\
				\hline
				$2H_2$ &$H_2$ &5& 268&172&60&0&12&0&No&No\\
				\hline
				\caption{Choices of $A$, $D$ in $Z=\mathbb{P}_{\mathbb{P}^2}(\mathcal{O}\oplus\mathcal{O}(2))$.}
				\label{2-36}
			\end{longtable}
		\end{center}
	\end{small}
	\paragraph{Ambiguities}
	In Tables \ref{2-18}-\ref{2-36} there are some families of Fano 4-folds arising from different constructions of type A that have the same invariants. We now list all ambiguities among non-toric varieties, namely all couples or triples of families that share the same invariants and are not toric.  
	\begin{itemize}
		\item (\#6 in Table \ref{2-32} , \#8 in Table \ref{2-33}).
		\item (\#3 Table \ref{2-32} , \#14 Table \ref{2-34}, \#15 in Table \ref{2-34})
	 
		\item (\#20 Table \ref{2-34} , \#21 Table \ref{2-34}).
		
		\item (\#34 in Table \ref{2-33} , \#16 in Table \ref{2-34}). 
		
		\item (\#2 Table \ref{2-32} , \#13 Table \ref{2-34} , \#7 Table \ref{2-35}).
	\end{itemize}	
    
One can check through geometric arguments that all the families listed above are distinct. 

The study of ambiguities among the toric 4-folds can be done by comparing their descriptions with the ones in \cite{Bat99}.
\section*{Appendix A. Final tables}
	\addcontentsline{toc}{section}{Appendix A: Final Tables}
	In Tables \ref{FinalTableRho4} and \ref{FinalTableRho5}  we list every family of Fano 4-folds $X$ (respectively with $\rho=4$ and $\rho=5$) that arise from Construction A. They all have $\delta=2$.  Under the column ``$Z$'', we make explicit the Fano 3-fold $Z$ from which $X$ arises through Construction A, by using the classical notation. Under the column ``Description of $X$'', we refer to the detailed description of the construction of type A that yields $X$, for example ``\#7-23'' means that the description of $X$ can be found in Table 7, row number 23.

	\begin{small}	
		\begin{longtable}{|c||c|>{\centering\arraybackslash}m{0.4cm}|c|c|>{\centering\arraybackslash}m{2.1cm}|c|c|C{2.1cm}|C{2.1cm}|}
			\caption{List of all families of Fano 4-folds arising from Construction A with $\rho=4$}
			\label{FinalTableRho4}\\
			\hline
			& $\rho$ & $c_1^4$ & $c_1^2c_2$ & $h^0(-K_X)$ &($h^{1,2}$, $h^{1,3}$,$h^{2,2}$)&Toric?&Product?& $Z$&Description of $X$\\
			\hline
			\endhead
			1&4&112&112&29&(2,0,14)&No&No&\#2-18&\#\ref{2-18}-2\\
			\hline
			2&4&142&124&35&(1,0,14)&No&No&\#2-25&\#\ref{2-25}-3\\
			\hline
			3&4&148&124&36&(0,0,12)&No&No&\#2-24&\#\ref{2-24}-2\\
			\hline
			4&4&152&128&37&(2,0,10)&No&No&\#2-18&\#\ref{2-18}-1\\
			\hline
			5&4&164&128&39&(2,0,6)&No&No&\#2-25&\#\ref{2-25}-1\\
			\hline
			6&4&172&136&41&(0,0,14)&No&No&\#2-29&\#\ref{2-29}-4\\
			\hline
			7&4&172&148&42&(0,3,42)&No&No&\#2-34&\#\ref{2-34}-37\\
			\hline
			8&4&178&136&42&(0,0,12)&No&No&\#2-27&\#\ref{2-27}-3\\
			\hline
			9&4&180&144&43&(0,1,24)&No&No&\#2-32&\#\ref{2-32}-9\\
			\hline
			10&4&184&148&44&(0,1,28)&No&No&\#2-33&\#\ref{2-33}-10\\
			\hline
			11&4&190&148&45&(0,0,21)&No&No&\#2-34&\#\ref{2-34}-31\\
			\hline
			12&4&194&140&45&(0,0,9)&No&No&\#2-24&\#\ref{2-24}-1\\
			\hline
			13&4&199&142&46&(1,0,9)&No&No&\#2-25&\#\ref{2-25}-2\\
			\hline
			14&4&205&154&48&(0,1,24)&No&No&\#2-34&\#\ref{2-34}-36\\
			\hline
			15&4&208&148&48&(0,0,12)&No&No&\#2-31&\#\ref{2-31}-4\\
			\hline
			16&4&208&148&48&(3,0,6)&No&No&\#3-24&\#\ref{2-34}-20\\
			\hline
			17&4&208&148&48&(3,0,6)&No&(\#3-9)$\times\mathbb{P}^1$&\#3-24&\#\ref{2-34}-21\\
			\hline
			18&4&210&156&49&(0,1,24)&No&No&\#2-35&\#\ref{2-35}-17\\
			\hline
			19&4&218&152&50&(0,0,12)&No&No&\#2-32&\#\ref{2-32}-8\\
			\hline
			20&4&221&146&50&(0,0,6)&No&No&\#2-27&\#\ref{2-27}-2\\
			\hline
			21&4&223&154&51&(0,0,14)&No&No&\#2-34&\#\ref{2-34}-30\\
			\hline
			22&4&230&152&52&(0,0,10)&No&No&\#2-30&\#\ref{2-30}-3\\
			\hline
			23&4&233&158&53&(0,0,14)&No&No&\#2-33&\#\ref{2-33}-9\\
			\hline
			24&4&238&160&54&(0,0,14)&No&No&\#2-34&\#\ref{2-34}-29\\
			\hline
			25&4&238&160&54&(0,0,12)&No&No&\#2-34&\#\ref{2-34}-35\\
			\hline
			26&4&241&154&54&(0,0,8)&No&No&\#2-27&\#\ref{2-27}-1\\
			\hline
			27&4&241&154&54&(1,0,6)&No&No&\#2-34&\#\ref{2-34}-19\\
			\hline
			28&4&246&156&55&(0,0,8)&No&No&\#2-29&\#\ref{2-29}-3\\
			\hline
			29&4&248&164&56&(0,0,12)&No&No&\#2-34&\#\ref{2-34}-34\\
			\hline
			30&4&248&164&56&(0,0,12)&No&No&\#2-35&\#\ref{2-35}-16\\
			\hline
			31&4&256&160&57&(0,0,6)&No&No&\#2-32&\#\ref{2-32}-5\\
			\hline
			32&4&256&160&57&(0,0,8)&No&No&\#2-32&\#\ref{2-32}-7\\
			\hline
			33&4&256&160&57&(1,0,6)&No&(\#3-14)$\times\mathbb{P}^1$&\#2-34&\#\ref{2-34}-18\\
			\hline
			34&4&265&166&59&(0,0,11)&No&No&\#2-35&\#\ref{2-35}-13\\
			\hline
			35&4&266&164&59&(0,0,6)&No&No&\#2-32&\#\ref{2-32}-4\\
			\hline
			36&4&266&164&59&(0,0,9)&No&No&\#2-34&\#\ref{2-34}-28\\
            \hline
			37&4&268&160&59&(0,0,6)&No&No&\#2-29&\#\ref{2-29}-2\\
			\hline
			38&4&268&172&60&(0,0,12)&No&No&\#2-36&\#\ref{2-36}-5\\
			\hline
			39&4&271&166&60&(1,0,6)&No&No&\#2-34&\#\ref{2-34}-17\\
			\hline
			40&4&271&166&60&(0,0,9)&No&No&\#2-34&\#\ref{2-34}-27\\
			\hline
			41&4&281&170&62&(0,0,6)&No&No&\#2-34&\#\ref{2-34}-33\\
			\hline
			42&4&282&168&62&(0,0,8)&No&No&\#2-32&\#\ref{2-32}-6\\
			\hline
			43&4&282&168&62&(0,0,8)&No&No&\#2-33&\#\ref{2-33}-8\\
			\hline
			44&4&284&164&62&(0,0,6)&No&No&\#2-34&\#\ref{2-34}-16\\
			\hline
			45&4&286&172&63&(0,0,6)&No&No&\#2-35&\#\ref{2-35}-15\\
			\hline
			46&4&288&168&63&(0,0,7)&No&No&\#2-31&\#\ref{2-31}-3\\
            \hline
			47&4&294&168&64&(0,0,6)&No&No&\#2-29&\#\ref{2-29}-1\\
			\hline
			48&4&296&176&65&(0,0,6)&No&No&\#2-35&\#\ref{2-35}-14\\
			\hline
			49&4&299&170&65&(0,0,7)&No&No&\#2-30&\#\ref{2-30}-2\\
			\hline
			50&4&301&178&66&(0,0,6)&No&No&\#2-34&\#\ref{2-34}-32\\
			\hline
			51&4&302&176&66&(0,0,9)&No&No&\#2-34&\#\ref{2-34}-26\\
			\hline
			52&4&303&174&66&(0,0,7)&No&No&\#2-35&\#\ref{2-35}-12\\
			\hline
			53&4&304&172&66&(0,0,6)&No&No&\#2-32&\#\ref{2-32}-3\\
			\hline
			54&4&304&172&66&(0,0,6)&No&(\#3-19)$\times\mathbb{P}^1$&\#2-34&\#\ref{2-34}-14\\
			\hline
			55&4&304&172&66&(0,0,6)&No&No&\#2-34&\#\ref{2-34}-15\\
			\hline
			56&4&308&176&67&(0,0,8)&No&No&\#2-33&\#\ref{2-33}-7\\
			\hline
			57&4&314&176&68&(0,0,6)&No&No&\#2-34&\#\ref{2-34}-25\\
			\hline
			58&4&316&172&68&(0,0,6)&No&No&\#2-30&\#\ref{2-30}-1\\
			\hline
			59&4&319&178&69&(0,0,6)&No&No&\#2-34&\#\ref{2-34}-24\\
			\hline
			60&4&320&176&69&(0,0,6)&No&No&\#2-32&\#\ref{2-32}-2\\
			\hline
			61&4&320&176&69&(0,0,6)&No&(\#3-22)$\times\mathbb{P}^1$&\#2-34&\#\ref{2-34}-13\\
			\hline
			62&4&320&176&69&(0,0,6)&No&No&\#2-35&\#\ref{2-35}-7\\
			\hline
			63&4&329&182&71&(0,0,7)&No&No&\#2-35&\#\ref{2-35}-11\\
			\hline
			64&4&331&178&71&(0,0,6)&No&No&\#2-31&\#\ref{2-31}-2\\
			\hline
			65&4&335&182&72&(0,0,6)&No&No&\#2-34&\#\ref{2-34}-23\\
			\hline
			66&4&337&178&72&(0,0,6)&$I_{15}$&No&\#2-34&\#\ref{2-34}-11\\
			\hline
			67&4&346&184&74&(0,0,6)&No&No&\#2-35&\#\ref{2-35}-6\\
			\hline
			68&4&347&182&74&(0,0,6)&$I_{12}$&No&\#2-34&\#\ref{2-34}-10\\
			\hline
			69&4&351&186&75&(0,0,5)&$H_{10}$&No&\#2-35&\#\ref{2-35}-10\\
			\hline
			70&4&356&188&76&(0,0,6)&No&No&\#2-34&\#\ref{2-34}-12\\
			\hline
			71&4&357&186&76&(0,0,6)&$I_{14}$&No&\#2-33&\#\ref{2-33}-6\\
			\hline
			72&4&362&188&77&(0,0,6)&No&No&\#2-33&\#\ref{2-32}-1\\
			\hline
			73&4&367&190&78&(0,0,5)&$H_{9}$&No&\#2-34, \#2-35&\#\ref{2-34}-5, \#\ref{2-35}-9\\
			\hline
			74&4&368&188&78&(0,0,6)&No&No&\#2-31&\#\ref{2-31}-1\\
			\hline
			75&4&368&188&78&(0,0,6)&$I_{13}$&(\#3-26)$\times\mathbb{P}^1$&\#2-34&\#\ref{2-34}-9\\
			\hline
			76&4&378&192&80&(0,0,5)&$H_{8}$&$\mathbb{P}^2\times\Bl_2\mathbb{P}^2$&\#2-34&\#\ref{2-34}-4\\
			\hline
			77&4&382&196&81&(0,0,5)&$H_{7}$&No&\#2-34, \#2-36&\#\ref{2-34}-3, \#\ref{2-36}-4\\
			\hline
			78&4&382&196&81&(0,0,6)&No&No&\#2-34&\#\ref{2-34}-22\\
			\hline
			79&4&384&192&81&(0,0,6)&$I_{8}$&No&\#2-35&\#\ref{2-35}-4, \#\ref{2-35}-5\\
			\hline
			80&4&389&194&82&(0,0,6)&$I_{10}$&No&\#2-34&\#\ref{2-34}-8\\
			\hline
			81&4&390&192&82&(0,0,6)&$I_{9}$&No&\#2-33&\#\ref{2-33}-4\\
			\hline
			82&4&400&196&84&(0,0,6)&$I_{7}$&(\#3-29)$\times\mathbb{P}^1$&\#2-34&\#\ref{2-34}-7\\
			\hline
			83&4&409&202&86&(0,0,5)&$H_{6}$&No&\#2-35&\#\ref{2-35}-8\\
			\hline
			84&4&411&198&86&(0,0,6)&$I_{6}$&No&\#2-33&\#\ref{2-33}-3\\
			\hline
			85&4&415&202&87&(0,0,6)&$I_{5}$&No&\#2-33&\#\ref{2-33}-5\\
			\hline
			86&4&415&202&87&(0,0,5)&$H_{5}$&No&\#2-34, \#2-35&\#\ref{2-34}-2, \#\ref{2-35}-2\\
			\hline
			87&4&433&202&90&(0,0,6)&$I_{4}$&No&\#2-33&\#\ref{2-33}-2\\
			\hline
			88&4&442&208&92&(0,0,6)&$I_{3}$&No&\#2-35&\#\ref{2-35}-3\\
			\hline
			89&4&447&210&93&(0,0,5)&$H_{4}$&No&\#2-35&\#\ref{2-35}-1\\
			\hline
			90&4&463&214&96&(0,0,6)&$I_{2}$&No&\#2-34&\#\ref{2-34}-6\\
			\hline
			91&4&478&220&99&(0,0,5)&$H_{3}$&No&\#2-34, \#2-36&\#\ref{2-34}-1, \#\ref{2-36}-3\\
			\hline
			92&4&496&220&102&(0,0,6)&$I_{1}$&No&\#2-33&\#\ref{2-33}-1\\
			\hline
			93&4&505&226&104&(0,0,5)&$H_{2}$&No&\#2-36&\#\ref{2-36}-2\\
			\hline
			94&4&558&240&114&(0,0,5)&$H_1$&No&\#2-36&\#\ref{2-36}-1\\
			\hline
		\end{longtable}
	\end{small}

	\begin{small}
		\begin{longtable}{|c||c|>{\centering\arraybackslash}m{0.4cm}|c|c|>{\centering\arraybackslash}m{2.1cm}|c|c|C{2.1cm}|C{2.1cm}|}
			\caption{List of all families of Fano 4-folds arising from Construction A with $\rho=5$}\label{FinalTableRho5}\\
			\hline
			& $\rho$ & $c_1^4$ & $c_1^2c_2$ &$h^0(-K_X)$&($h^{1,2}$,$h^{1,3}$,$h^{2,2}$)&Toric?&Product?& $Z$&Description of $X$\\
			\hline 
			\endhead
			1&5&180&144&43&(0,1,26)&No&No&\#3-27&\#\ref{3-27}-23\\
			\hline
			2&5&184&136&43&(0,0,12)&No&No&\#3-17&\#\ref{3-17}-2\\
			\hline
			3&5&200&140&46&(0,0,12)&No&No&\#3-19&\#\ref{DoubleTable}-3\\
			\hline
			4&5&202&148&47&(0,0,16)&No&No&\#3-28&\#\ref{3-28}-18\\
			\hline
			5&5&202&148&47&(0,0,16)&No&No&\#3-27&\#\ref{3-27}-23\\
			\hline
			6&5&214&148&49&(0,0,12)&No&No&\#3-25&\#\ref{3-25}-4\\
			\hline
			7&5&220&148&50&(0,0,10)&No&No&\#3-24&\#\ref{3-24}-2\\
			\hline
			8&5&224&152&51&(1,0,8)&No&No&\#3-27&\#\ref{3-27}-16\\
			\hline
			9&5&224&152&51&(1,0,8)&No&(\#4-2)$\times\mathbb{P}^1$&\#3-27&\#\ref{3-27}-17\\
			\hline
			10&5&224&152&51&(0,0,12)&No&No&\#3-27&\#\ref{3-27}-21\\
			\hline
			11&5&229&154&52&(0,0,12)&No&No&\#3-28&\#\ref{3-28}-17\\
			\hline
			12&5&234&156&53&(0,0,12)&No&No&\#3-27&\#\ref{3-27}-20\\
			\hline
			13&5&236&152&53&(0,0,10)&No&No&\#3-17&\#\ref{3-17}-1\\
			\hline
			14&5&244&160&55&(0,0,12)&No&No&\#3-31&\#\ref{3-31}-5\\
			\hline
			15&5&245&158&55&(0,0,10)&No&No&\#3-28&\#\ref{3-28}-14\\
			\hline
			16&5&246&156&55&(0,0,8)&No&No&\#3-27&\#\ref{3-27}-15\\
			\hline
			17&5&250&160&56&(0,0,10)&No&No&\#3-30&\#\ref{3-30}-4\\
			\hline
			18&5&252&156&56&(0,0,10)&No&No&\#3-26, \#3-19&\#\ref{DoubleTable}-2\\
			\hline
			19&5&256&160&57&(0,0,8)&No&No&\#3-28&\#\ref{3-28}-15\\
			\hline
			20&5&256&160&57&(0,0,8)&No&(\#4-4)$\times\mathbb{P}^1$&\#3-28&\#\ref{3-28}-16\\
			\hline
			21&5&256&160&57&(0,0,8)&No&(\#4-5)$\times\mathbb{P}^1$&\#3-27&\#\ref{3-27}-14\\
			\hline
			22&5&256&160&57&(0,0,10)&No&No&\#3-27&\#\ref{3-27}-19\\
			\hline
			23&5&266&164&59&(0,0,8)&No&No&\#3-27&\#\ref{3-27}-13\\
			\hline
			24&5&278&164&61&(0,0,8)&No&No&\#3-27&\#\ref{3-27}-12\\
			\hline
			25&5&278&164&61&(0,0,9)&No&No&\#3-24&\#\ref{3-24}-1\\
			\hline
			26&5&282&168&62&(0,0,10)&No&No&\#3-27&\#\ref{3-27}-18\\
			\hline
			27&5&283&166&62&(0,0,9)&No&No&\#3-25&\#\ref{3-25}-3\\
			\hline
			28&5&288&168&63&(0,0,8)&No&(\#4-7)$\times\mathbb{P}^1$&\#3-27&\#\ref{3-27}-10\\
			\hline
			29&5&288&168&63&(0,0,8)&No&No&\#3-27&\#\ref{3-27}-11\\
			\hline
			30&5&293&170&64&(0,0,9)&No&No&\#3-28&\#\ref{3-28}-13\\
			\hline
			31&5&299&170&65&(0,0,8)&$Q_{17}$&No&\#3-28&\#\ref{3-28}-12\\
			\hline
			32&5&304&172&66&(0,0,8)&No&(\#4-8)$\times\mathbb{P}^1$&\#3-27&\#\ref{3-27}-9\\
			\hline
			33&5&310&172&67&(0,0,8)&$Q_{16}$&No&\#3-27, \#3-25&\#\ref{3-27}-7, \#\ref{3-25}-2\\
			\hline
			34&5&310&172&67&(0,0,9)&$\text{Yes}^3$\tnote{a}&No&\#3-26&\#\ref{DoubleTable}-1\\
			\hline
			35&5&320&176&69&(0,0,8)&$Q_{15}$&(\#4-9)$\times\mathbb{P}^1$&\#3-28, \#3-27&\#\ref{3-28}-11, \#\ref{3-27}-5, \#\ref{3-27}-6\\
			\hline
			36&5&325&178&70&(0,0,8)&$Q_{14}$&No&\#3-30, \#3-28&\#\ref{3-30}-3, \#\ref{3-28}-3\\
			\hline
			37&5&330&180&71&(0,0,8)&$Q_{13}$&No&\#3-31, \#3-27&\#\ref{3-31}-4, \#\ref{3-27}-4\\
			\hline
			38&5&330&180&71&(0,0,8)&No&No&\#3-27&\#\ref{3-27}-8\\
			\hline
			39&5&331&178&71&(0,0,8)&$Q_{12}$&No&\#3-28, \#3-25&\#\ref{3-28}-9, \#\ref{3-25}-1\\
			\hline
			40&5&336&180&72&(0,0,8)&$Q_{10}$&$\mathbb{F}_1\times\Bl_2\mathbb{P}^2$&\#3-28&\#\ref{3-28}-2\\
			\hline
			41&5&336&180&72&(0,0,8)&$Q_{11}$&(\#4-10)$\times\mathbb{P}^1$&\#3-27&\#\ref{3-27}-3\\
			\hline
			42&5&341&182&73&(0,0,8)&$Q_{9}$&No&\#3-28&\#\ref{3-28}-8, \#\ref{3-28}-10\\
			\hline
			43&5&352&184&75&(0,0,8)&$Q_{8}$&(\#4-11)$\times\mathbb{P}^1$&\#3-28, \#3-27&\#\ref{3-28}-7, \#\ref{3-27}-2\\
			\hline
			44&5&363&186&77&(0,0,8)&$Q_{7}$&No&\#3-28&\#\ref{3-28}-6\\
			\hline
			45&5&368&188&78&(0,0,8)&$Q_{6}$&(\#4-12)$\times\mathbb{P}^1$&\#3-28&\#\ref{3-28}-5\\
			\hline
			46&5&373&190&79&(0,0,8)&$Q_{5}$&No&\#3-30, \#3-28&\#\ref{3-30}-2, \#\ref{3-28}-1\\
			\hline
			47&5&394&196&83&(0,0,8)&$Q_{3}$&No&\#3-31, \#3-27&\#\ref{3-31}-3, \#\ref{3-27}-1\\
			\hline
			48&5&405&198&85&(0,0,8)&$Q_{4}$&No&\#3-31, \#3-28&\#\ref{3-31}-2, \#\ref{3-28}-4\\
			\hline
			49&5&405&198&85&(0,0,8)&$Q_2$&No&\#3-30&\#\ref{3-30}-1\\
			\hline
			50&5&442&208&92&(0,0,8)&$Q_1$&No&\#3-31&\#\ref{3-31}-1\\
			\hline
		\end{longtable}
	\end{small}
	
	\begin{tablenotes}
		\footnotesize
		\item[a] $^{\text{3}}$number 108 in \cite{Bat99}, \S4.
	\end{tablenotes}

\paragraph{Acknowledgments}
I would like to thank my advisor, Cinzia Casagrande, for suggesting the problem and for her continuous guidance and support. I am also grateful to Saverio A. Secci for many helpful comments and suggestions.

	\newcommand{\etalchar}[1]{$^{#1}$}
\providecommand{\bysame}{\leavevmode\hbox to3em{\hrulefill}\thinspace}
\providecommand{\MR}{\relax\ifhmode\unskip\space\fi MR }
\providecommand{\MRhref}[2]{%
	\href{http://www.ams.org/mathscinet-getitem?mr=#1}{#2}
}
\providecommand{\href}[2]{#2}

\vspace{1em}

Università di Torino, Dipartimento di Matematica, via Carlo Alberto 10, 10123 Torino - Italy\newline
Email address: pierroberto.pastorino@unito.it

\end{document}